\documentclass[oneside,10pt,]{article}
\usepackage{geometry}
\usepackage{ifthen}
\usepackage{etoolbox}
\usepackage{ifxetex,ifluatex}
\usepackage{graphicx}
\PassOptionsToPackage{dvipsnames,svgnames,table}{xcolor}
\usepackage{xcolor}
\usepackage{tcolorbox}
\tcbuselibrary{skins}
\tcbuselibrary{breakable}
\tcbuselibrary{raster}
\tcbset{ bwminimalstyle/.style={size=minimal, boxrule=-0.3pt, frame empty,
colback=white, colbacktitle=white, coltitle=black, opacityfill=0.0} }
\tcbset{ runintitlestyle/.style={fonttitle=\blocktitlefont\upshape\bfseries, attach title to upper} }
\tcbset{ exercisespacingstyle/.style={before skip={1.5ex plus 0.5ex}} }
\tcbset{ blockspacingstyle/.style={before skip={2.0ex plus 0.5ex}} }
\tcbuselibrary{xparse}
\usetikzlibrary{calc}
\tcbuselibrary{hooks}
\newlength{\normalparindent}
\newlength{\normalparskip}
\AtBeginDocument{\setlength{\normalparindent}{\parindent}}
\AtBeginDocument{\setlength{\normalparskip}{\parskip}}
\newcommand{\setparstyle}{\setlength{\parindent}{\normalparindent}\setlength{\parskip}{\normalparskip}}
\IfFileExists{latexrelease.sty}{}{\usepackage{fixltx2e}}
\usepackage{xstring}
\ifdefined\counterwithin
\else
    \usepackage{chngcntr}
\fi
\ifthenelse{\boolean{xetex} \or \boolean{luatex}}{%
\ifxetex\usepackage{xltxtra}\fi
\ifluatex\usepackage{realscripts}\fi
}{
\usepackage[utf8]{inputenc}
}
\newcommand{\divisionfont}{\relax}
\newcommand{\blocktitlefont}{\relax}
\newcommand{\contentsfont}{\relax}

\newcommand{\tabularfont}{\relax}
\newcommand{\xreffont}{\relax}

\ifthenelse{\boolean{xetex} \or \boolean{luatex}}{%
\usepackage{fontspec}
\usepackage{polyglossia}
\setmainlanguage[variant=usmax]{english}
}{%
\usepackage{lmodern}
\usepackage[T1]{fontenc}
}
\usepackage{microtype}
\usepackage{amsmath}
\usepackage{amscd}
\usepackage{amssymb}
\allowdisplaybreaks[4]
\usepackage[explicit, newparttoc, pagestyles]{titlesec}
\usepackage{titletoc}
\newcommand{\divisionnameptx}{\relax}%
\newcommand{\authorsptx}{\relax}%
\NewDocumentEnvironment{sectionptx}{mmmmmmm}
{%
\renewcommand{\divisionnameptx}{#1}%
\renewcommand{\titleptx}{#2}%
\renewcommand{\subtitleptx}{#3}%
\renewcommand{\shortitleptx}{#4}%
\renewcommand{\authorsptx}{#5}%
\renewcommand{\epigraphptx}{#6}%
\section[{#4}]{#2}%
\label{#7}%
}{}%
\NewDocumentEnvironment{subsectionptx}{mmmmmmm}
{%
\renewcommand{\divisionnameptx}{#1}%
\renewcommand{\titleptx}{#2}%
\renewcommand{\subtitleptx}{#3}%
\renewcommand{\shortitleptx}{#4}%
\renewcommand{\authorsptx}{#5}%
\renewcommand{\epigraphptx}{#6}%
\subsection[{#4}]{#2}%
\label{#7}%
}{}%
\NewDocumentEnvironment{references-section}{mmmmmmm}
{%
\renewcommand{\divisionnameptx}{#1}%
\renewcommand{\titleptx}{#2}%
\renewcommand{\subtitleptx}{#3}%
\renewcommand{\shortitleptx}{#4}%
\renewcommand{\authorsptx}{#5}%
\renewcommand{\epigraphptx}{#6}%
\section[{#4}]{#2}%
\label{#7}%
}{}%
\NewDocumentEnvironment{references-section-numberless}{mmmmmmm}
{%
\renewcommand{\divisionnameptx}{#1}%
\renewcommand{\titleptx}{#2}%
\renewcommand{\subtitleptx}{#3}%
\renewcommand{\shortitleptx}{#4}%
\renewcommand{\authorsptx}{#5}%
\renewcommand{\epigraphptx}{#6}%
\section*{#2}%
\addcontentsline{toc}{section}{#4}
\label{#7}%
}{}%
\titleformat{\part}[display]
{\divisionfont\Huge\bfseries\centering}{\divisionnameptx\space\thepart}{30pt}{\Huge#1}
[{\Large\centering\authorsptx}]
\titleformat{\chapter}[display]
{\divisionfont\huge\bfseries}{\divisionnameptx\space\thechapter}{20pt}{\Huge#1}
[{\Large\authorsptx}]
\titleformat{name=\chapter,numberless}[display]
{\divisionfont\huge\bfseries}{}{0pt}{#1}
[{\Large\authorsptx}]
\titlespacing*{\chapter}{0pt}{50pt}{40pt}
\titleformat{\section}[hang]
{\divisionfont\Large\bfseries}{\thesection}{1ex}{#1}
[{\large\authorsptx}]
\titleformat{name=\section,numberless}[block]
{\divisionfont\Large\bfseries}{}{0pt}{#1}
[{\large\authorsptx}]
\titlespacing*{\section}{0pt}{3.5ex plus 1ex minus .2ex}{2.3ex plus .2ex}
\titleformat{\subsection}[hang]
{\divisionfont\large\bfseries}{\thesubsection}{1ex}{#1}
[{\normalsize\authorsptx}]
\titleformat{name=\subsection,numberless}[block]
{\divisionfont\large\bfseries}{}{0pt}{#1}
[{\normalsize\authorsptx}]
\titlespacing*{\subsection}{0pt}{3.25ex plus 1ex minus .2ex}{1.5ex plus .2ex}
\titleformat{\subsubsection}[hang]
{\divisionfont\normalsize\bfseries}{\thesubsubsection}{1em}{#1}
[{\small\authorsptx}]
\titleformat{name=\subsubsection,numberless}[block]
{\divisionfont\normalsize\bfseries}{}{0pt}{#1}
[{\normalsize\authorsptx}]
\titlespacing*{\subsubsection}{0pt}{3.25ex plus 1ex minus .2ex}{1.5ex plus .2ex}
\titleformat{\paragraph}[hang]
{\divisionfont\normalsize\bfseries}{\theparagraph}{1em}{#1}
[{\small\authorsptx}]
\titleformat{name=\paragraph,numberless}[block]
{\divisionfont\normalsize\bfseries}{}{0pt}{#1}
[{\normalsize\authorsptx}]
\titlespacing*{\paragraph}{0pt}{3.25ex plus 1ex minus .2ex}{1.5em}
\titlecontents{part}%
[0pt]{\contentsmargin{0em}\addvspace{1pc}\contentsfont\bfseries}%
{\Large\thecontentslabel\enspace}{\Large}%
{}%
[\addvspace{.5pc}]%
\titlecontents{chapter}%
[0pt]{\contentsmargin{0em}\addvspace{1pc}\contentsfont\bfseries}%
{\large\thecontentslabel\enspace}{\large}%
{\hfill\bfseries\thecontentspage}%
[\addvspace{.5pc}]%
\dottedcontents{section}[3.8em]{\contentsfont}{2.3em}{1pc}%
\dottedcontents{subsection}[6.1em]{\contentsfont}{3.2em}{1pc}%
\dottedcontents{subsubsection}[9.3em]{\contentsfont}{4.3em}{1pc}%

\newcommand{\terminology}[1]{\textbf{#1}}
\newcommand{\pubtitle}[1]{\textsl{#1}}
\newcommand{\lititle}[1]{{\slshape#1}}
\numberwithin{equation}{section}
\tcbset{ imagestyle/.style={bwminimalstyle} }
\NewTColorBox{tcbimage}{mmm}{imagestyle,left skip=#1\linewidth,width=#2\linewidth}
\NewDocumentEnvironment{image}{mmmm}{\notblank{#4}{\leavevmode\nopagebreak\vspace{#4}}{}\begin{tcbimage}{#1}{#2}{#3}}{\end{tcbimage}%
}
\usepackage{array}
\newcommand{\hrulethin}  {\noalign{\hrule height 0.04em}}

\let\oldsetlength\setlength
\newlength{\Oldarrayrulewidth}
\newcommand{\crulethin}[1]%
{\noalign{\global\oldsetlength{\Oldarrayrulewidth}{\arrayrulewidth}}%
\noalign{\global\oldsetlength{\arrayrulewidth}{0.04em}}\cline{#1}%
\noalign{\global\oldsetlength{\arrayrulewidth}{\Oldarrayrulewidth}}}%
\newcommand{\crulemedium}[1]%
{\noalign{\global\oldsetlength{\Oldarrayrulewidth}{\arrayrulewidth}}%
\noalign{\global\oldsetlength{\arrayrulewidth}{0.07em}}\cline{#1}%
\noalign{\global\oldsetlength{\arrayrulewidth}{\Oldarrayrulewidth}}}
\newcommand{\crulethick}[1]%
{\noalign{\global\oldsetlength{\Oldarrayrulewidth}{\arrayrulewidth}}%
\noalign{\global\oldsetlength{\arrayrulewidth}{0.11em}}\cline{#1}%
\noalign{\global\oldsetlength{\arrayrulewidth}{\Oldarrayrulewidth}}}
\newcolumntype{A}{!{\vrule width 0.04em}}
\newcolumntype{B}{!{\vrule width 0.07em}}
\newcolumntype{C}{!{\vrule width 0.11em}}
\tcbset{ tabularboxstyle/.style={bwminimalstyle,} }
\newtcolorbox{tabularbox}[3]{tabularboxstyle, left skip=#1\linewidth, width=#2\linewidth,}
\counterwithin*{footnote}{section}
\usepackage{enumitem}
\newlist{referencelist}{description}{4}
\setlist[referencelist]{leftmargin=!,labelwidth=!,labelsep=0ex,itemsep=1.0ex,topsep=1.0ex,partopsep=0pt,parsep=0pt}
\usepackage[hyperfootnotes=false]{hyperref}
\hypersetup{colorlinks=true,linkcolor=blue,citecolor=blue,filecolor=blue,urlcolor=blue}
\hypersetup{hypertexnames=false}
\makeatletter
\patchcmd\Hy@EveryPageBoxHook{\Hy@EveryPageAnchor}{\Hy@hypertexnamestrue\Hy@EveryPageAnchor}{}{\fail}
\makeatother
\hypersetup{pdftitle={Currently there are no reasons to doubt the Riemann Hypothesis}}

\newtcolorbox[auto counter, number within=section]{block}{}
\newtcolorbox[auto counter, number within=section]{project-distinct}{}
\makeatletter
\newtcolorbox[auto counter, number within=tcb@cnt@block, number freestyle={\noexpand\thetcb@cnt@block(\noexpand\alph{\tcbcounter})}]{subdisplay}{}
\makeatother
\tcbset{ lemmastyle/.style={bwminimalstyle, runintitlestyle, blockspacingstyle, after title={\space}, before upper app={\setparstyle}, } }
\newtcolorbox[use counter from=block]{lemma}[4]{title={{#1~\thetcbcounter\notblank{#2#3}{\space}{}\notblank{#2}{\space#2}{}\notblank{#3}{\space(#3)}{}}}, phantomlabel={#4}, breakable, after={\par}, fontupper=\itshape, lemmastyle, }
\tcbset{ theoremstyle/.style={bwminimalstyle, runintitlestyle, blockspacingstyle, after title={\space}, before upper app={\setparstyle}, } }
\newtcolorbox[use counter from=block]{theorem}[4]{title={{#1~\thetcbcounter\notblank{#2#3}{\space}{}\notblank{#2}{\space#2}{}\notblank{#3}{\space(#3)}{}}}, phantomlabel={#4}, breakable, after={\par}, fontupper=\itshape, theoremstyle, }
\tcbset{ propositionstyle/.style={bwminimalstyle, runintitlestyle, blockspacingstyle, after title={\space}, before upper app={\setparstyle}, } }
\newtcolorbox[use counter from=block]{proposition}[4]{title={{#1~\thetcbcounter\notblank{#2#3}{\space}{}\notblank{#2}{\space#2}{}\notblank{#3}{\space(#3)}{}}}, phantomlabel={#4}, breakable, after={\par}, fontupper=\itshape, propositionstyle, }
\tcbset{ proofstyle/.style={bwminimalstyle, fonttitle=\blocktitlefont\itshape, attach title to upper, after title={\space}, after upper={\space\space\hspace*{\stretch{1}}\(\blacksquare\)},
} }
\newtcolorbox{proof}[3]{title={\notblank{#2}{#2}{#1.}}, phantom={\hypertarget{#3}{}}, breakable, after={\par}, proofstyle, before upper app={\setparstyle} }
\tcbset{ principlestyle/.style={bwminimalstyle, runintitlestyle, blockspacingstyle, after title={\space}, before upper app={\setparstyle}, } }
\newtcolorbox[use counter from=block]{principle}[4]{title={{#1~\thetcbcounter\notblank{#2#3}{\space}{}\notblank{#2}{\space#2}{}\notblank{#3}{\space(#3)}{}}}, phantomlabel={#4}, breakable, after={\par}, fontupper=\itshape, principlestyle, }
\tcbset{ conjecturestyle/.style={bwminimalstyle, runintitlestyle, blockspacingstyle, after title={\space}, before upper app={\setparstyle}, } }
\newtcolorbox[use counter from=block]{conjecture}[4]{title={{#1~\thetcbcounter\notblank{#2#3}{\space}{}\notblank{#2}{\space#2}{}\notblank{#3}{\space(#3)}{}}}, phantomlabel={#4}, breakable, after={\par}, fontupper=\itshape, conjecturestyle, }
\tcbset{ axiomstyle/.style={bwminimalstyle, runintitlestyle, blockspacingstyle, after title={\space}, before upper app={\setparstyle}, } }
\newtcolorbox[use counter from=block]{axiom}[4]{title={{#1~\thetcbcounter\notblank{#2#3}{\space}{}\notblank{#2}{\space#2}{}\notblank{#3}{\space(#3)}{}}}, phantomlabel={#4}, breakable, after={\par}, fontupper=\itshape, axiomstyle, }
\tcbset{ heuristicstyle/.style={bwminimalstyle, runintitlestyle, blockspacingstyle, after title={\space}, before upper app={\setparstyle}, } }
\newtcolorbox[use counter from=block]{heuristic}[4]{title={{#1~\thetcbcounter\notblank{#2#3}{\space}{}\notblank{#2}{\space#2}{}\notblank{#3}{\space(#3)}{}}}, phantomlabel={#4}, breakable, after={\par}, fontupper=\itshape, heuristicstyle, }
\tcbset{ notestyle/.style={bwminimalstyle, runintitlestyle, blockspacingstyle, after title={\space}, before upper app={\setparstyle}, } }
\newtcolorbox[use counter from=block]{note}[3]{title={{#1~\thetcbcounter\notblank{#2}{\space\space#2}{}}}, phantomlabel={#3}, breakable, after={\par}, notestyle, }
\tcbset{ warningstyle/.style={bwminimalstyle, runintitlestyle, blockspacingstyle, after title={\space}, before upper app={\setparstyle}, } }
\newtcolorbox[use counter from=block]{warning}[3]{title={{#1~\thetcbcounter\notblank{#2}{\space\space#2}{}}}, phantomlabel={#3}, breakable, after={\par}, warningstyle, }
\tcbset{ observationstyle/.style={bwminimalstyle, runintitlestyle, blockspacingstyle, after title={\space}, before upper app={\setparstyle}, } }
\newtcolorbox[use counter from=block]{observation}[3]{title={{#1~\thetcbcounter\notblank{#2}{\space\space#2}{}}}, phantomlabel={#3}, breakable, after={\par}, observationstyle, }
\tcbset{ problemstyle/.style={bwminimalstyle, runintitlestyle, blockspacingstyle, after title={\space}, after upper={\space\space\hspace*{\stretch{1}}\(\square\)}, before upper app={\setparstyle}, } }
\newtcolorbox[use counter from=block]{problem}[3]{title={{#1~\thetcbcounter\notblank{#2}{\space\space#2}{}}}, phantomlabel={#3}, breakable, after={\par}, problemstyle, }
\tcbset{ figureptxstyle/.style={bwminimalstyle, middle=1ex, blockspacingstyle, fontlower=\blocktitlefont} }
\newtcolorbox[use counter from=block]{figureptx}[4]{lower separated=false, before lower={{\textbf{#1~\thetcbcounter}\space#2}}, phantomlabel={#3}, unbreakable, figureptxstyle, }
\tcbset{ tableptxstyle/.style={bwminimalstyle, middle=1ex, blockspacingstyle, coltitle=black, bottomtitle=2ex, titlerule=-0.3pt, fonttitle=\blocktitlefont, breakable} }
\newtcolorbox[use counter from=block]{tableptx}[4]{title={{\textbf{#1~\thetcbcounter}\space#2}}, phantomlabel={#3}, unbreakable, tableptxstyle, }
\NewDocumentEnvironment{introduction}{m}
{\notblank{#1}{\noindent\textbf{#1}\space}{}}{\par\medskip}
\titleformat{\subparagraph}[runin]{\normalfont\normalsize\bfseries}{\thesubparagraph}{1em}{#1}
\titlespacing*{\subparagraph}{0pt}{3.25ex plus 1ex minus .2ex}{1em}
\NewDocumentEnvironment{paragraphs}{mm}
{\subparagraph*{#1}\hypertarget{#2}{}}{}
\tcbset{ sbsstyle/.style={raster before skip=2.0ex, raster equal height=rows, raster force size=false, raster after skip=0.7\baselineskip} }
\tcbset{ sbspanelstyle/.style={bwminimalstyle, fonttitle=\blocktitlefont, before upper app={\setparstyle}} }
\NewDocumentEnvironment{sidebyside}{mmmm}
  {\begin{tcbraster}
    [sbsstyle,raster columns=#1,
    raster left skip=#2\linewidth,raster right skip=#3\linewidth,raster column skip=#4\linewidth]}
  {\end{tcbraster}}
\NewTColorBox{sbspanel}{mO{top}}{sbspanelstyle,width=#1\linewidth,valign=#2}
\usepackage{extpfeil}
\newcommand{\R}{\mathbb R}

\newcommand{\abs}[1]{|#1|}

\newcommand{\oointerval}[2]{(#1, #2)}
\newcommand{\diag}{\operatorname{diag}}
\newcommand{\Tr}{\operatorname{Tr}}
\newcommand{\CbetaE}{\operatorname{C{\beta}E}}
\newcommand{\GbetaE}{\operatorname{G{\beta}E}}
\newcommand{\COE}{\operatorname{COE}}
\newcommand{\CUE}{\operatorname{CUE}}
\newcommand{\CSE}{\operatorname{CSE}}
\newcommand{\GOE}{\operatorname{GOE}}
\newcommand{\GUE}{\operatorname{GUE}}
\newcommand{\GSE}{\operatorname{GSE}}
\newcommand{\lt}{<}
\newcommand{\gt}{>}
\newcommand{\amp}{&}
\title{Currently there are no reasons to doubt\\the Riemann Hypothesis\\
{\large The zeta function beyond the realm of computation}}
\author{David W. Farmer
}
\date{}
\begin{document}
\raggedbottom
\hypertarget{root-1-2}{}
\maketitle
\thispagestyle{empty}
\renewcommand*{\abstractname}{Abstract}
\begin{abstract}
We examine published arguments which suggest that the Riemann Hypothesis may not be true. In each case we provide evidence to explain why the claimed argument does not provide a good reason to doubt the Riemann Hypothesis. The evidence we cite involves a mixture of theorems in analytic number theory, theorems in random matrix theory, and illustrative examples involving the characteristic polynomials of large random unitary matrices. Similar evidence is provided for four mistaken notions which appear repeatedly in the literature concerning computations of the zeta-function. A fundamental question which underlies some of the arguments is: what does the graph of the Riemann zeta-function look like in a neighborhood of its largest values? We explore that question in detail and provide a survey of results on the relationship between \(L\)-functions and the characteristic polynomials of random matrices. We highlight the key role played by the emergent phenomenon of carrier waves, which arise from fluctuations in the density of zeros. The main point of this paper is that it is possible to understand some aspects of the zeta function at large heights, but the computational evidence is misleading. \footnote{MSC2020: 11M26, 11M50\label{root-1-2-3-2-1-2}}%
\end{abstract}
\typeout{************************************************}
\typeout{Section 1 Introduction}
\typeout{************************************************}
\begin{sectionptx}{Section}{Introduction}{}{Introduction}{}{}{intro}
Should one believe the Riemann Hypothesis (RH)? Since it is a conjecture with no proposed roadmap to prove it, one point of view is that it should neither be believed nor disbelieved.  Yet many mathematicians have an opinion, presumably backed up by logical reasoning.%
\par
Here we consider all published arguments for doubting~RH.  A paper of Ivić~[\hyperlink{Iv1}{{\xreffont 72}}, \hyperlink{Iv2}{{\xreffont 73}}] lists 4 reasons, and a paper of Blanc~\hyperlink{Blanc1}{[{\xreffont 18}]}  provides a 5th reason.  Three of those reasons involve speculation about the distribution of zeros and their relationship to the value distribution of the \(\zeta\)-function.  The fundamental question there is:  what does the \(\zeta\)-function look like in a neighborhood of its largest values?  The majority of this paper is a survey of prior results, and speculations and heuristics based on those results, which lead to an answer to that question.%
\par
Our primary goal is to provide intuition and to be persuasive.  The arguments against RH generally take the form ``It would be surprising if \(\mathbf{X}\).'' So, the burden we bear in refuting that argument is to give good reasons why \(\mathbf{X}\) is not surprising. To crystallize our main points we present 33 \emph{Principles} which we hope are also useful for future reference.%
\par
Is the purpose of this paper to persuade that RH is true?  Certainly not.  But perhaps those who continue to doubt RH will realize that their belief is not based on good evidence.  As for those who believe~RH: perhaps someone will write a companion paper: \emph{Currently there are no good reasons to believe the Riemann Hypothesis}. Note the subtle difference from the opposite of the title to this paper. Also useful would be a paper explaining why every currently known equivalence to RH is unlikely to be helpful for proving~RH (as in~\hyperlink{jensen}{[{\xreffont 47}]}).%
\par
The \(\zeta\)-function is the simplest example of an \emph{L-function}. All \(L\)-functions have properties similar to the \(\zeta\)-function, and all \(L\)-functions have an analogue of the Riemann Hypothesis.  As much as possible we try to discuss the \(\zeta\)-function in isolation, but in a few places it is necessary to expand our perspective.  We attempt to keep this paper self-contained, providing definitions and background as needed.%
\begin{paragraphs}{The themes.}{intro-7}%
Many of the surprising properties of the \(\zeta\)-function arise because there are different facets to the same object.  A classic example is the function usually denoted \(S(t)\): it is the error term in the counting function of the zeros of the \(\zeta\)-function, and it also is the imaginary part of the logarithm of~\(\zeta(\frac12 + i t)\). That equivalence is basic complex analysis.  But when pondering  a specific question about \(S(t)\), sometimes one perspective gives good intuition, and sometimes the other.  Combining perspectives can lead to surprising relations, such as \hyperref[prin_Z_S]{Principle~{\xreffont\ref{prin_Z_S}}}: in regions where \(\abs{\zeta(\frac12 + i t)}\) is large, \(S(t)\) tends to be decreasing.%
\par
A second theme is multiple levels of randomness.  We will make extensive use of the fact that the \(\zeta\)-function and characteristic polynomials of random unitary matrices have a similar sort of randomness (and also some differences, which we will describe). In the random matrix world, it is a theorem that everything just follows from the quadratic repulsion between the eigenvalues.  Yes really: everything. But (thank you to an anonymous referee for helping me see it this way), that perspective is reductive and it is beneficial to have other points of view, particularly on larger scales where the individual zeros are not visible. In the \(\zeta\)-function world this means that sometimes the primes will enter the discussion, and other times they will be ignored.%
\end{paragraphs}%
\begin{paragraphs}{The sections.}{intro-8}%
In \hyperref[overview]{Section~{\xreffont\ref{overview}}} we introduce the main theme by pulling together various ideas in the paper to answer the question: What does the zeta function look like beyond the realm of computation? The remainder of the paper is primarily devoted to providing intuition which will dispel misconceptions about the answer to that question \textemdash{} misconceptions largely due to poor extrapolation from available computations. In \hyperref[basics]{Section~{\xreffont\ref{basics}}} we provide basic definitions and background. In \hyperref[misleading]{Section~{\xreffont\ref{misleading}}} we describe four Mistaken Notions which appear repeatedly in discussions of computations of the \(\zeta\)-function, some of which play an important role in the claimed reasons to doubt~RH. In \hyperref[waves]{Section~{\xreffont\ref{waves}}} we describe the connection between the distribution of zeros and the size of the \(\zeta\)-function, introducing \emph{carrier waves} as a way to separate local from long-range behavior. In \hyperref[rmt_connections]{Section~{\xreffont\ref{rmt_connections}}} we briefly describe the connection between the \(\zeta\)-function and unitary polynomials, and in \hyperref[RMThistory]{Section~{\xreffont\ref{RMThistory}}} we provide an historical account of the connections to Random Matrix Theory. This leads to \hyperref[rmtwaves]{Section~{\xreffont\ref{rmtwaves}}}, where we use large unitary matrices to illustrate phenomena which occur far outside the range in which we can compute the \(\zeta\)-function.  By the end of \hyperref[rmtwaves]{Section~{\xreffont\ref{rmtwaves}}} we have a good understanding of the ``typical'' large values of the \(\zeta\)-function, and the relationship between the carrier wave, the density wave, and \(S(t)\), but it is not until \hyperref[extreme]{Section~{\xreffont\ref{extreme}}} that we address the most extreme values. The primes are only briefly mentioned up to this point, a shortcoming we address in \hyperref[primes]{Section~{\xreffont\ref{primes}}}. The zeros have dominated most of our discussion, but in \hyperref[multiscale]{Section~{\xreffont\ref{multiscale}}} we discuss randomness of the \(\zeta\)-function or characteristic polynomials without reference to zeros or eigenvalues. After all that preparation, in \hyperref[rhissues_Z]{Section~{\xreffont\ref{rhissues_Z}}} we use information from the prior sections to refute the three arguments against RH based on the distribution of zeros and values of the \(\zeta\)-function, and for completeness in \hyperref[rhissues_other]{Section~{\xreffont\ref{rhissues_other}}} we cite recent results to refute the other two arguments against~RH. Finally, in \hyperref[misleadingrevisited]{Section~{\xreffont\ref{misleadingrevisited}}} we use the Principles to explain why the Mistaken Notions of \hyperref[misleading]{Section~{\xreffont\ref{misleading}}} are, in fact, mistaken.%
\end{paragraphs}%
\begin{paragraphs}{Acknowledgments.}{intro-9}%
I thank Louis-Pierre Arguin, Juan Arias de Reyna, Emma Bailey, Sir Michael Berry, Philippe Blanc, Richard Brent, Brian Conrey, Jon Keating, Hugh Montgomery, Eero Saksman, Tim Trudgian, and Christian Webb for clarifying several points in this article. I also thank Jonathan Bober, Xavier Gourdon, and Ghaith Hiary for making available extensive data from their computations of the \(\zeta\)-function. In addition, useful suggestions from a referee led to many improvements. This paper was written in PreTeXt~\hyperlink{PTX}{[{\xreffont 91}]}.%
\end{paragraphs}%
\end{sectionptx}
\typeout{************************************************}
\typeout{Section 2 The \(\zeta\)-function for absurdly large inputs}
\typeout{************************************************}
\begin{sectionptx}{Section}{The \(\zeta\)-function for absurdly large inputs}{}{The \(\zeta\)-function for absurdly large inputs}{}{}{overview}
\begin{introduction}{}%
We present some of the main points in this paper in a thought experiment to give intuition about the behavior of the \(\zeta\)-function far beyond the range where it can be computed. It is hoped that the reader will view the claims in this overview with some skepticism, leading to closer scrutiny of the ideas and perspectives in the remainder of the paper.%
\end{introduction}%
\typeout{************************************************}
\typeout{Subsection 2.1 Snapshots of the \(\zeta\)-function}
\typeout{************************************************}
\begin{subsectionptx}{Subsection}{Snapshots of the \(\zeta\)-function}{}{Snapshots of the \(\zeta\)-function}{}{}{overview-3}
As described in \hyperref[basics]{Section~{\xreffont\ref{basics}}}, we will consider the function \(Z(t)\), which contains the same information as the Riemann \(\zeta\)-function, but it is a real-valued function of a real variable, so we can graph it. What might the graph of \(Z(t)\) look like for \(t \approx 100^{100^{100}}\)? Current methods fail well before \(10^{40}\), so such a computation is not remotely plausible. But suppose in some distant future a computer algebra package allowed one to graph the \(Z\)-function for such a large \(t\), on an interval of, for example, width \(40/(2 \pi \log(t/2\pi))\) \textemdash{} an interval which should contain around 40 zeros. By \hyperref[guehypothesis]{Principle~{\xreffont\ref{guehypothesis}}} and \hyperref[KSlaw]{Principle~{\xreffont\ref{KSlaw}}}, combined with \hyperref[prin_three_things]{Principle~{\xreffont\ref{prin_three_things}}} or \hyperref[short-long]{Principle~{\xreffont\ref{short-long}}}, it is plausible that the graph might look like this:%
\begin{figureptx}{Figure}{What \(Z(t)\) looks like, on some interval which is expected to contain around 40 zeros, near \(t\approx 100^{100^{100}}\). The tick marks are separated by the average gap between zeros at that height.}{fig_fictional}{}%
\begin{image}{0}{1}{0}{}%
\includegraphics[width=\linewidth]{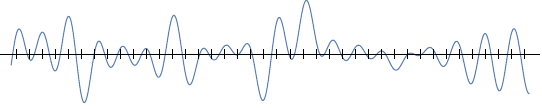}
\end{image}%
\tcblower
\end{figureptx}%
Neither axis in \hyperref[fig_fictional]{Figure~{\xreffont\ref{fig_fictional}}} has a scale, although we are told enough information to determine the width covered by the horizontal axis. The scale on the vertical axis is the mystery.%
\par
By \hyperref[prin_three_things]{Principle~{\xreffont\ref{prin_three_things}}}, the \(Z\)-function is close to being a polynomial.  In particular, locally the behavior of the \(Z\)-function is determined by its zeros, up to a local constant factor. So, if the relative spacing of the zeros in \hyperref[fig_fictional]{Figure~{\xreffont\ref{fig_fictional}}} is plausible for \(t\) of that size, then the given graph is accurate:  as is typical for computer algebra systems, the vertical scale is automatically adjusted to show the features of the graph.%
\par
By \hyperref[guehypothesis]{Principle~{\xreffont\ref{guehypothesis}}}, the distribution of zeros in \hyperref[fig_fictional]{Figure~{\xreffont\ref{fig_fictional}}} is indeed plausible, and furthermore we expect to find many intervals of that width at that height, on which the graph of the Z-function looks similar to \hyperref[fig_fictional]{Figure~{\xreffont\ref{fig_fictional}}}. The only thing we don't know is: what is the vertical scale of that graph?%
\par
By Selberg's Central Limit Theorem (CLT) \hyperref[selberg_gaussian]{({\xreffont\ref{selberg_gaussian}})}, \(\log Z(t)\) is normally distributed with mean zero and variance \(\sqrt{\frac12 \log\log t}\).  So, at height \(t\approx 100^{100^{100}}\), half the time \(\log Z(t)\) is likely to be of size \(\approx 15\), and the other half of the time it is likely to be~\(\approx -15\). Exponentiating, we find that the vertical scale in \hyperref[fig_fictional]{Figure~{\xreffont\ref{fig_fictional}}}, or in any other interval at that height where the graph looks like \hyperref[fig_fictional]{Figure~{\xreffont\ref{fig_fictional}}}, is likely to be around \(10^6\), or it is likely to be around~\(10^{-6}\), and both possibilities are equally probable. Both cases lead to the same snapshot, so we see that the snapshot depends only on the local arrangement of zeros. In particular, the local arrangement of zeros is not the main factor in the size of the \(\zeta\)-function.%
\par
This fundamental fact, that the size of the \(\zeta\)-function is not strongly dependent on the local distribution of zeros, is hard to believe when your intuition only comes from computations in the modest range currently accessible by computers.%
\begin{paragraphs}{Snapshots do not show the big picture.}{overview-3-9}%
By \terminology{snapshot} we refer to an image showing the graph of a function, having an aspect ratio between \(1/10\) and \(2\), over an interval where the function has a moderate number of interesting features, with the vertical scale adjusted so that those features are visible.  \hyperref[fig_fictional]{Figure~{\xreffont\ref{fig_fictional}}} is an example, as are the numerous other graphs in this paper.  A snapshot of the \(\zeta\)-function typically covers an interval containing between 10 and 100 zeros.  As mentioned above, a snapshot of the \(Z\)-function depends only on the location of the nearby zeros.  More precisely, it depends only on the relative sizes of the gaps between zeros. Near the middle of the snapshot all that matters is the zero gaps within the snapshot; near the edges, the gaps immediately outside the snapshot are relevant.%
\par
Any specific (finite) sequence of normalized zero gaps will never happen (except once, if it is specifically constructed in that way). But precision is not relevant, because a snapshot is a continuous function of the zero gaps: a very small change in the relative spacing of the zeros will cause a small change in the appearance of the graph. Given an open neighborhood of any specific (finite) gap sequence, such an approximate gap sequence will occur infinitely many times, and in fact will occur a positive proportion of the time among all normalized gap sequences of that length \textemdash{} the precise proportion depending on the given sequence and the size of the neighborhood. See \hyperref[guehypothesis]{Principle~{\xreffont\ref{guehypothesis}}} or \hyperref[KSlaw]{Principle~{\xreffont\ref{KSlaw}}}. For example, \hyperref[fig_brentZ]{Figure~{\xreffont\ref{fig_brentZ}}} will occur a positive proportion of the time (to within the resolution of the human eye) as a snapshot of the \(Z\)-function.%
\par
To summarize:%
\begin{principle}{Principle}{Snapshots of the \(\zeta\)-function are well understood.}{}{snapshots}%
Any given finite sequence of normalized gaps between zeros will (approximately) occur a positive proportion of the time, as predicted by the Random Matrix Model for zeros of the zeta function.  The appearance of the graph of the \(Z\)-function over that region, ignoring the vertical scale, is primarily determined by the relative spacing of those zeros.%
\end{principle}
In particular, one does not expect to see anything surprising when looking at a new snapshot of the \(Z\)-function.%
\end{paragraphs}%
\par\medskip
What counts as ``surprising'' changes over time.   In the early computations of zeros of the \(\zeta\)-function, it was surprising that occasionally two zeros would be very close together, known as \terminology{Lehmer pairs}. Today such pairs of close zeros are expected, and those occur at the frequency predicted by \hyperref[guehypothesis]{Principle~{\xreffont\ref{guehypothesis}}}.%
\par
In this section we have viewed the \(\zeta\)-function on an interval as a random function, modulo an overall scaling factor.  Equivalently, we are looking at a handful of consecutive zeros. For much of this paper, that perspective is sufficient for our purposes. In \hyperref[gmc]{Subsection~{\xreffont\ref{gmc}}} we briefly discuss a random model for the \(\zeta\)-function which goes beyond snapshots to also incorporate the vertical scale.%
\end{subsectionptx}
\end{sectionptx}
\typeout{************************************************}
\typeout{Section 3 Background on the \(\zeta\)- function}
\typeout{************************************************}
\begin{sectionptx}{Section}{Background on the \(\zeta\)- function}{}{Background on the \(\zeta\)- function}{}{}{basics}
\begin{introduction}{}%
The Riemann zeta-function, which we will call the \(\zeta\)-function, is defined by%
\begin{equation}
\zeta(s)=\sum_{n=1}^\infty \frac{1}{n^s}\label{basics-2-1-2}
\end{equation}
for \(\sigma \gt 1\), where \(s=\sigma + i t\) is a complex variable. The \(\zeta\)-function has a meromorphic continuation to the complex plane, with a simple pole at \(s=1\) with residue~\(1\).%
\par
The \(\zeta\)-function has a symmetry, known as the \terminology{functional equation}, which can be expressed in several ways. With%
\begin{equation}
X(s) := \pi^{s-\frac12}\frac{\Gamma(\frac12 - \frac{1}{2} s)}{\Gamma(\frac12 s)}\label{basics-2-2-3}
\end{equation}
where \(\Gamma\) is the Euler Gamma-function, we have%
\begin{equation}
\zeta(s) = X(s)\zeta(1-s)\text{,}\label{zeta_fe}
\end{equation}
or equivalently%
\begin{align}
\xi(s) \mathstrut :=\mathstrut\amp \tfrac12 s (s-1)\pi^{-\frac12 s}\Gamma(\tfrac12 s) \zeta(s)\label{eqn_xi_def}\\
=\mathstrut\amp \xi(1-s)\text{,}\label{basics-2-2-6-2}
\end{align}
or equivalently%
\begin{equation}
Z(t):= X(\tfrac12 + i t)^{-\frac12} \zeta(\tfrac12 + it)
\ \ \ \ \text{is real if}\ \ \ \ t\in \R\text{.}\label{eqn_Zdef}
\end{equation}
The factor \(\tfrac12 s (s-1)\) in \hyperref[eqn_xi_def]{({\xreffont\ref{eqn_xi_def}})} is irrelevant to the invariance under \(s \leftrightarrow 1-s\), but it is traditionally included so that \(\xi(s)\) is an entire function. In \hyperref[eqn_Zdef]{({\xreffont\ref{eqn_Zdef}})} the square root is chosen so that \(Z(0) = \zeta(\frac12) \approx -1.46\) and \(Z(t)\) is analytic for \(\abs{\Im(t)} \lt \frac12\). It is common to refer to \(|t|\), the magnitude of the imaginary part of \(s=\sigma + i t\), as the \terminology{height} when referring to the behavior of \(\zeta(s)\) or~\(Z(t)\) in a particular region.%
\par
The \terminology{Hardy Z-function} \hyperref[eqn_Zdef]{({\xreffont\ref{eqn_Zdef}})} is useful because it can be graphed, and it tells us essentially everything we might want to know about the \(\zeta\)-function because \(|Z(t)| = |\zeta(\frac12 + i t)|\) if \(t \in \R\). \hyperref[fig_z5000]{Figure~{\xreffont\ref{fig_z5000}}} shows \(Z(t)\) and \(\log\abs{Z(t)}\) for \(5429.29 \lt t \lt 5466.44\), along with the function \(S(t)\) which we will introduce shortly.%
\begin{figureptx}{Figure}{\(Z(t)\), \(\log\abs{Z(t)}\), and \(S(t)\) near the 5000th zero \(\gamma_{5000} \approx 5447.86\).}{fig_z5000}{}%
\begin{image}{0}{1}{0}{}%
\includegraphics[width=\linewidth]{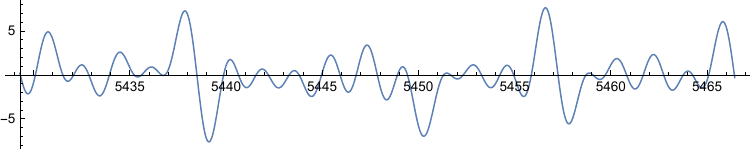}
\end{image}%
\begin{image}{0}{1}{0}{}%
\includegraphics[width=\linewidth]{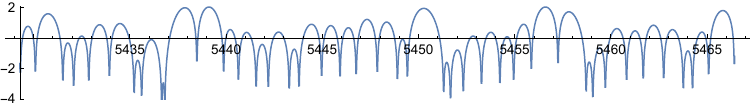}
\end{image}%
\begin{image}{0}{1}{0}{}%
\includegraphics[width=\linewidth]{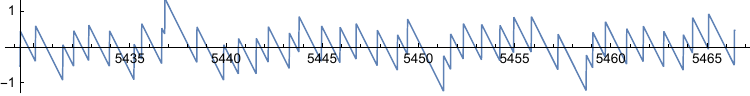}
\end{image}%
\tcblower
\end{figureptx}%
The critical facts (pun intended) about \(Z(t)\) are that it is smooth, and if \(t\) is real then \(Z(t)\) is real and \(\abs{Z(t)} = \abs{\zeta(\frac12 + it)}\).  Those facts do not uniquely determine \(Z(t)\): there is still a global choice of \(\pm 1\).  One might ask whether the choice matters, and if it does, is the standard choice for the square-root in \hyperref[eqn_Zdef]{({\xreffont\ref{eqn_Zdef}})} the correct option?%
\begin{principle}{Principle}{}{}{hardyZpm}%
\(Z(t)\) is statistically indistinguishable from \(\,-Z(t)\).%
\end{principle}
In other words, any measure of distance between probability distributions cannot distinguish between the distributions of \(Z(t)\) and \(\,-Z(t)\).%
\par
Thus, the global choice of sign for \(Z(t)\) does not matter. Indeed, \(\xi(\frac12 + it)\) is also real for \(t \in \R\), and has the same critical zeros as \(Z(t)\), but it has the opposite sign of \(Z(t)\). So, the standard normalizations for those functions disagree with each other. A consequence of \hyperref[hardyZpm]{Principle~{\xreffont\ref{hardyZpm}}} is that if \(k\) is a non-negative integer then%
\begin{equation}
\int_T^{2T} Z(t)^{2 k + 1} \, dt = o(T)\text{,}\label{Zoddmoment}
\end{equation}
and as \(T\to\infty\) that integral changes sign infinitely many times. Equation \hyperref[Zoddmoment]{({\xreffont\ref{Zoddmoment}})} is a theorem for \(k=0\): use the Riemann-Siegel formula \hyperref[RiemannSiegel]{({\xreffont\ref{RiemannSiegel}})} and integrate term-by-term. It is a conjecture for larger~\(k\).%
\par
\hyperref[hardyZpm]{Principle~{\xreffont\ref{hardyZpm}}} is a stronger statement than asserting that to leading order the value distribution of \(Z(t)\) is symmetric. In contrast, the average value of \(\zeta(\frac12 + it)\) is \(1\) in a very strong sense: if \(k\) is a non-negative integer then%
\begin{equation}
\int_T^{2T} \zeta(\tfrac12 +i t)^{k} \, dt \sim T\text{,}\label{zeta1average}
\end{equation}
which can be shown by moving the integral to the region of absolute convergence of the Dirichlet series and then integrating term-by-term. Such small biases (meaning: effects of size 1, when the main term is growing very slowly) are the underlying cause of many misconceptions arising from numerical computations, see \hyperref[gramaverage]{({\xreffont\ref{gramaverage}})} and \hyperref[bigOof1prin]{Principle~{\xreffont\ref{bigOof1prin}}}.%
\end{introduction}%
\typeout{************************************************}
\typeout{Subsection 3.1 We care about zeros because we care about primes}
\typeout{************************************************}
\begin{subsectionptx}{Subsection}{We care about zeros because we care about primes}{}{We care about zeros because we care about primes}{}{}{basics-3}
Riemann's great insight was that the zeros of the \(\zeta\)-function encode information about the primes.  Based at least partially on numerical computation, he formulated the conjecture which is now known as the \terminology{Riemann Hypothesis}:%
\begin{conjecture}{Conjecture}{The Riemann Hypothesis (RH).}{}{basics-3-3}%
The zeros of \(\zeta(s)\) with \(0 \lt \sigma \lt 1\) lie on the line \(\sigma=\frac12\); equivalently, all zeros of \(\xi(s)\) lie on the line \(\sigma = \frac12\); equivalently, all zeros of \(Z(t)\) with \(\abs{\Im(t)} \lt \frac12\) are real.%
\end{conjecture}
Riemann's original formulation involved yet another function: all the zeros of \(\Xi(t):= \xi(\frac12 + i t)\) are real. RH has been rigorously verified for \(t \lt 3\times 10^{12}\), involving more than \(12.3 \times 10^{12}\) zeros~\hyperlink{PlTr}{[{\xreffont 89}]}.%
\par
The zeros of \(\zeta(s)\) with \(0 \lt \sigma \lt 1\) are traditionally denoted by \(\rho = \beta + i \gamma\).  By the functional equation and the fact that \(\zeta(s)\) is real if \(s\) is real, if \(\rho = \beta + i \gamma\) is a zero of \(\zeta(s)\) then so is \(\rho = 1 - \beta + i \gamma\). So either \(\beta=\frac12\), as predicted by RH, or there is a pair of zeros of \(\zeta(s)\) (or \(\xi(s)\)) located symmetrically around the \terminology{critical line} \(\sigma = \frac12\). Equivalently, a failure of RH corresponds to a pair of complex conjugate zeros of \(Z(t)\).  Assuming RH, the zeros of \(\zeta(s)\) with positive imaginary part are denoted \(\frac12 + i \gamma_n\) with \(0 \lt \gamma_1 \lt \gamma_2 \lt \cdots\).  That notation would break down if there were repeated zeros, but no plausible reason has been given to expect a multiple zero.  Various discussions in this paper may implicitly assume all zeros are simple, an assumption which generally is irrelevant to the points being made.%
\par
There are \(41\) zeros of the \(\zeta\)-function visible in \hyperref[fig_z5000]{Figure~{\xreffont\ref{fig_z5000}}}. RH~and the simplicity of zeros imply that (for \(\abs{t} > 3\)) all local maxima of \(Z(t)\) (or \(\xi(\frac12 + it)\)) are strictly positive and all local minima are strictly negative, because the zeros of the derivative interlace the zeros of the function. This is visible in the first two graphs in \hyperref[fig_z5000]{Figure~{\xreffont\ref{fig_z5000}}}. The converse is not necessarily true \textemdash{} a failure of RH need not cause a negative maximum or a positive minimum, although that is mistakenly asserted in the official statement of the Riemann Hypothesis Millennium Problem~\hyperlink{millen}{[{\xreffont 23}]}.  The issue is the distance of a hypothetical non-critical zero from the critical line, see \hyperlink{allreal}{[{\xreffont 48}]}.   So, one cannot trivially deduce from the graph of \(Z(t)\) or \(\log(\abs{Z(t)})\) in \hyperref[fig_z5000]{Figure~{\xreffont\ref{fig_z5000}}} that RH holds for \(5430 \lt t \lt 5448 \).%
\par
It is possible to verify RH on an interval by a computer calculation, based on two properties of the \(\zeta\)-function.  One property we have already seen: \(Z(t)\) is real when \(t\) is real, so one can count critical zeros by looking for sign changes.  The other is that the zeros of the \(\zeta\)-function have a nice counting function with a small and computable error term.  Let \(N(T)\) be the number of zeros of \(\zeta(s)\) with \(0 \lt \gamma_n \le T\).  We have \hyperlink{T}{[{\xreffont 110}]}, assuming \(T\) is not the imaginary part of a zero of \(\zeta(s)\),%
\begin{equation}
N(T) = \frac{T}{2\pi} \log \frac{T}{2\pi e} + \tfrac78 + S(T) 
+  O(T^{-1})\label{NT}
\end{equation}
where%
\begin{equation}
S(T) = \frac{1}{\pi} \arg \zeta(\tfrac12 + i T) = \frac{1}{\pi} \Im \log \zeta(\tfrac12 + i T)\text{.}\label{ST}
\end{equation}
The argument in \hyperref[ST]{({\xreffont\ref{ST}})} is determined by continuous variation along the line \(2 + i t\) for \(0 \le t \le T\), and then along the line \(\sigma + i T\) for \(2 \ge \sigma \ge \frac12\).  One can take either \hyperref[NT]{({\xreffont\ref{NT}})} or \hyperref[ST]{({\xreffont\ref{ST}})} as the definition, with the other as a theorem.%
\par
We will see that the function \(S(t)\) grows very slowly. The third plot in \hyperref[fig_z5000]{Figure~{\xreffont\ref{fig_z5000}}} illustrates the basic properties of \(S(t)\):  it has a jump discontinuity (of height 1) at a simple critical zero, and where it is continuous it is approximately linear with slope \(-\frac{1}{2\pi}\log t\).%
\par
By \hyperref[NT]{({\xreffont\ref{NT}})} and \hyperref[ST]{({\xreffont\ref{ST}})} one can rigorously prove that RH holds on an interval: use sign changes to count real zeros of \(Z(t)\), then compute the change in \(S(t)\) to determine the change in \(N(t)\) and thus find the total number of zeros. Check if those two quantities are equal. A sophisticated version of this idea is known as \terminology{Turing's method}; see \hyperlink{BookerTuring}{[{\xreffont 25}]} for a modern treatment. In \hyperref[fig_z5000]{Figure~{\xreffont\ref{fig_z5000}}} a failure of RH would correspond to a jump discontinuity in \(S(t)\) of height~2, at a point where \(Z(t)\) does not have a zero. Such a jump does not occur in the graph of \(S(t)\), so by also considering either \(Z(t)\) or \(\log \abs{Z(t)}\), one can ``see'' a proof in \hyperref[fig_z5000]{Figure~{\xreffont\ref{fig_z5000}}} that RH is true for \(5430 \lt t \lt 5448\).%
\end{subsectionptx}
\typeout{************************************************}
\typeout{Subsection 3.2 Unfolding the zeros}
\typeout{************************************************}
\begin{subsectionptx}{Subsection}{Unfolding the zeros}{}{Unfolding the zeros}{}{}{basics-4}
By \hyperref[NT]{({\xreffont\ref{NT}})} the zeros at larger height are on average closer together. Specifically, at height \(T\) the average gap between zeros is~\(2\pi/\log T\), and \(\gamma_n \approx 2\pi n /\log n\). When discussing the statistics of the zeros, in particular the gaps between zeros, it is helpful to use the \terminology{normalized} or \terminology{unfolded} zeros \(\tilde{\gamma}_1, \tilde{\gamma}_2, \ldots\), where \(\tilde{\gamma}_n \sim n\) and \(\tilde{\gamma}_{n+1} - \tilde{\gamma}_n\) equals \(1\) on average. More precisely,  we set \(\tilde{\gamma}_n = \widetilde{N}(\gamma_n)\), where%
\begin{equation}
\widetilde{N}(T) = \frac{T}{2\pi} \log \frac{T}{2\pi e} + \tfrac78 \text{.}\label{basics-4-2-13}
\end{equation}
Often we think in terms of the approximation \(\tilde{\gamma}_n \approx \frac{1}{2\pi} \gamma_n \log \gamma_n\). We leave it as an exercise to determine the average value of \(\tilde{\gamma}_n - n\) as \(n\to \infty\). The answer is in \hyperref[spectralrigidity]{Subsection~{\xreffont\ref{spectralrigidity}}}.  That answer also explains the usual definition of \(S(t)\) in the case \(t\) is the imaginary part of a zero of the \(\zeta\)-function.%
\end{subsectionptx}
\typeout{************************************************}
\typeout{Subsection 3.3 The size of \(Z(t)\) and \(S(t)\)}
\typeout{************************************************}
\begin{subsectionptx}{Subsection}{The size of \(Z(t)\) and \(S(t)\)}{}{The size of \(Z(t)\) and \(S(t)\)}{}{}{basics-5}
The function \(S(t)\) is both the error term in the zero counting function \(N(t)\) and the imaginary part of~\(\log \zeta(\frac12 + i t)\).  Viewed as the error term for a function that counts irregularly spaced points, it is perhaps surprising that \(S(t)\) grows very slowly.  A consequence of that slow growth is that the zeros cannot stray too far from their expected location.  We will see that this has profound implications for the behavior of the \(\zeta\)-function.%
\par
Assuming RH, Littlewood~\hyperlink{T}{[{\xreffont 110}]} showed that \(S(T) = O(\log T/\log\log T)\).  It is conjectured \hyperlink{FGH}{[{\xreffont 49}]} that%
\begin{equation}
\abs{S(T)} \le (1+o(1)) \frac{1}{\pi} \sqrt{\frac12 \log{T}\log\log T}\text{,}\label{FGHconjectureS}
\end{equation}
and that bound is sharp. The results for \(Z(t)\) are analogous: on RH we have \hyperlink{T}{[{\xreffont 110}]} \(\log|Z(t)| = O(\log t/ \log\log t)\), and the conjecture (with sharp constant) is%
\begin{equation}
\log \abs{Z(T)} \le (1+o(1)) \sqrt{\frac12 \log{T}\log\log T}\text{.}\label{FGHconjecture}
\end{equation}
In the other direction \hyperlink{zetaomega}{[{\xreffont 22}]}, the current best result is that there exists \(C \gt 0\) such that there exist arbitrarily large \(T\) with:%
\begin{equation}
\log \abs{Z(T)} \gt C \, \sqrt\frac{\log T \log\log\log T}{\log \log T} \text{.}\label{eqn_Omega}
\end{equation}
\par
The typical size of \(Z(t)\) is much smaller. Selberg~\hyperlink{SelS}{[{\xreffont 103}]} proved that if \(t\) is chosen uniformly at random from \([T, 2T]\), then%
\begin{equation}
\frac{\log \zeta(\frac12 + i t)}{\sqrt{\frac12 \log\log T}}
\to N(0,1)
\ \ \ \ \
\text{ as }
\ \ \ \ \ \
T\to \infty\text{,}\label{selberg_gaussian}
\end{equation}
where \(N(0,1)\) is the standard (complex) Gaussian, and \hyperref[selberg_gaussian]{({\xreffont\ref{selberg_gaussian}})} indicates convergence in distribution. We will refer to this as ``Selberg's CLT''. In particular, \(\log\abs{Z(t)}\) and \(S(t)\) each have a (real) Gaussian distribution, which are different only because the definition of \(S(t)\) contains a factor of \(1/\pi\), and those distributions are independent.%
\par
In numerical computations of the \(\zeta\)-function the scale factor in Selberg's CLT is practically irrelevant:  at \(T_{BH}=10^{33}\), which is approximately the largest height where the \(\zeta\)-function has been calculated~\hyperlink{BobHia}{[{\xreffont 21}]}, Selberg's CLT, assuming it is valid at that low height, says that the typical size of \(\log\abs{Z(t)}\) near \(T_{BH}\) is \(\sqrt{\frac12\log\log T_{BH}} \approx 1.5\). The conjectured extreme values at that height are larger, \(\sqrt{\frac12 \log T_{BH} \log\log T_{BH}} \approx 12.8\), so within the realm of current computation one might expect to see \(\abs{Z(t)}\) larger than \(300,000\). Unfortunately, the extreme values are rare, and possibly (due to lower order terms) the predicted largest values do not occur until significantly greater heights.  The largest value of \(Z(t)\) found in \hyperlink{BobHia}{[{\xreffont 21}]} is approximately \(16244\) near \(3.92\times 10^{31}\), which is 3.7\% less than the largest ever computed value~\hyperlink{Tih2}{[{\xreffont 109}]}.  The largest calculated value of \(S(t)\) is \(3.345\), near \(t = 7.7573 \times 10^{27}\).%
\par
\hyperref[fig_boberhiary]{Figure~{\xreffont\ref{fig_boberhiary}}} shows \(Z(t)\), \(\log \abs{Z(t)}\), and \(S(t)\) in a neighborhood of the largest value of \(Z(t)\) found in~\hyperlink{BobHia}{[{\xreffont 21}]}. (Note that there is no good linear scale on which to plot \(Z(t)\) in that region.) The author thanks Jonathan Bober and Ghaith Hiary for providing open access to their extensive data~\hyperlink{Hia}{[{\xreffont 66}]}.%
\begin{figureptx}{Figure}{Plots of \(Z(T_{big}+t)\), \(\log \abs{Z(T_{big}+t)}\), and \(S(T_{big}+t)\) where \(T_{big} = 39246764589894309155251169284084 \approx 3.9 \times 10^{31}\).}{fig_boberhiary}{}%
\begin{image}{0}{1}{0}{}%
\includegraphics[width=\linewidth]{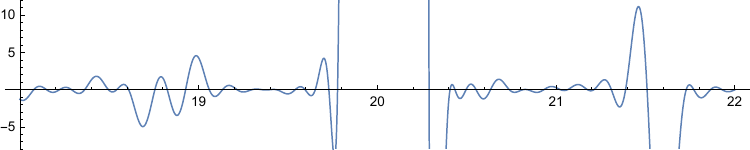}
\end{image}%
\begin{image}{0}{1}{0}{}%
\includegraphics[width=\linewidth]{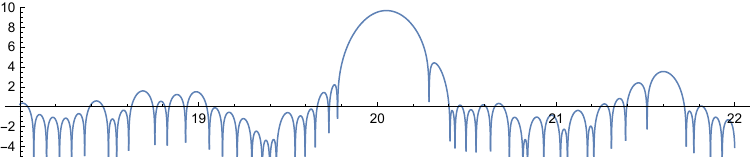}
\end{image}%
\begin{image}{0}{1}{0}{}%
\includegraphics[width=\linewidth]{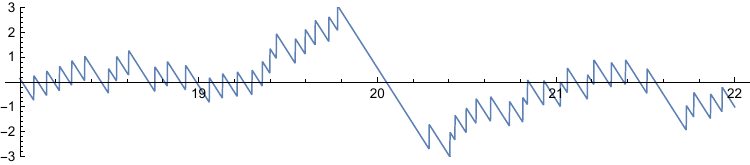}
\end{image}%
\tcblower
\end{figureptx}%
In \hyperref[fig_boberhiary]{Figure~{\xreffont\ref{fig_boberhiary}}} we see that this particular large value of \(Z(t)\) arises from a large gap between neighboring zeros.  The large zero gap also contributes to the large value of \(S(t)\): a zero gap of \(K\) times the local average spacing must be accompanied by a value \(\abs{S(t)} \ge K/2\).  The large zero gap in \hyperref[fig_boberhiary]{Figure~{\xreffont\ref{fig_boberhiary}}} is \(5.93\) times the local average. The relationship between the size of an isolated large zero gap and the local maximum of \(Z(t)\) is subtle: see \hyperref[prin_large_Z_gaps]{Principle~{\xreffont\ref{prin_large_Z_gaps}}}. The relationship between a large gap and \(S(t)\) is almost trivial, but we record it for later use. We write \(S^+(t)\) and \(S^-(t)\) for the right and left limiting values, respectively.%
\begin{principle}{Principle}{}{}{prin_large_S_gaps}%
If \(\tilde{\gamma}_{j+1} - \tilde{\gamma}_{j} = K\), then \(S^+({\gamma}_{j+1}) - S^-({\gamma}_{j}) = K\). So in particular either \(\abs{S^+({\gamma}_{j+1})} \ge K/2\) or \(\abs{S^-({\gamma}_{j})} \ge K/2\). Thus, an upper bound on \(\abs{S(t)}\) implies a comparable upper bound on the size of the normalized zero gaps.%
\end{principle}
Note that \hyperref[prin_large_S_gaps]{Principle~{\xreffont\ref{prin_large_S_gaps}}} does not say that a large value of \(S(t)\) must be accompanied by a large gap between zeros.%
\par
As will be explored in detail, \hyperref[fig_boberhiary]{Figure~{\xreffont\ref{fig_boberhiary}}} does not illustrate the typical behavior for the large values of \(Z(t)\) and \(S(t)\).  One way to see this is from Selberg's CLT~\hyperlink{SelS}{[{\xreffont 103}]} that \(\log\abs{Z(t)}\) and \(S(t)\) are independently distributed. One of those being typically large should have no effect on the other.  \hyperref[fig_boberhiary]{Figure~{\xreffont\ref{fig_boberhiary}}} illustrates that a large zero gap causes large values of \(Z(t)\) and \(S(t)\) to occur in close proximity, therefore that cannot be the typical behavior near a large value. (This discussion does not address the question of independence in the tails of the distributions of \(\log\abs{Z(t)}\) and \(S(t)\).)%
\par
In \hyperref[misleading]{Section~{\xreffont\ref{misleading}}} we briefly explore the history of finding large values of \(Z(t)\) and explain why that work has inadvertently led to a mistaken impression of what the graph of \(Z(t)\) looks like in a neighborhood of its largest values. In \hyperref[waves]{Section~{\xreffont\ref{waves}}} we introduce \emph{carrier waves}, which are the actual cause of the largest values of \(Z(t)\).%
\end{subsectionptx}
\end{sectionptx}
\typeout{************************************************}
\typeout{Section 4 Misleading ideas about large values}
\typeout{************************************************}
\begin{sectionptx}{Section}{Misleading ideas about large values}{}{Misleading ideas about large values}{}{}{misleading}
\begin{introduction}{}%
Computation has been an important tool for studying the \(\zeta\)-function ever since Riemann calculated the first few zeros by hand.  Computers have enabled large-scale computations:  large on the the human scale but minuscule on an absolute scale. In reference to whether existing computations should be seen as  evidence for RH, Andrew Odlyzko~\hyperlink{OdlOld}{[{\xreffont 86}]} has sounded a cautionary note:  the true nature of the \(\zeta\)-function is unlikely to be revealed until we reach regions where \(S(t)\) is routinely larger than \(100\).  Since (by Selberg's CLT) \(S(t)\) is typically of size \(\sqrt{\frac12 \log\log t}\), such regions will be inaccessible for a long time.%
\par
Despite Odlyzko's warning, there are certain aspects of the \(\zeta\)-function which appear in computations and have influenced the direction of research, but which do not accurately portray the true nature of the \(\zeta\)-function. These misconceptions are partially based on the way the \(\zeta\)-function is computed, which we describe next.%
\end{introduction}%
\typeout{************************************************}
\typeout{Subsection 4.1 Computations of \(Z(t)\)}
\typeout{************************************************}
\begin{subsectionptx}{Subsection}{Computations of \(Z(t)\)}{}{Computations of \(Z(t)\)}{}{}{misleading-3}
The earliest large-scale computations of the \(\zeta\)-function calculated \(Z(t)\) using the \terminology{Riemann-Siegel formula}:%
\begin{equation}
Z(t) = 2\sum_{n \lt \sqrt{t/{2\pi}}}
n^{-\frac12} \cos(\theta(t) - t\log n) \ + \ \text{remainder}\label{RiemannSiegel}
\end{equation}
where%
\begin{align}
\theta(t) =\mathstrut \amp\arg\left(\pi^{-i t/2}\Gamma\biggl(\frac14 + i\frac{t}{2}\biggr)\right)\notag\\
=\mathstrut \amp \frac{t}{2}\log \biggl(\frac{t}{2\pi}\biggr)  - \frac{t}{2}  -\frac{\pi}{8} + O(t^{-1})\text{.}\label{eqn_thetat}
\end{align}
The function \(\theta(t)\) arises in an alternate expression for \(Z(t)\):%
\begin{equation}
Z(t) = e^{i \theta(t)} \zeta(\tfrac12 + it)\text{.}\label{misleading-3-2-8}
\end{equation}
\par
The \(n=1\) term in \hyperref[RiemannSiegel]{({\xreffont\ref{RiemannSiegel}})} is the largest, and if \(t\) is small then that term has a strong influence on the overall sum.  The points where \(\cos(\theta(t)) = \pm 1\) are known as \terminology{Gram points}, numbered so that \(\theta(g_m) = m \pi\). Almost equivalently, Gram points are the locations where \(\zeta(\frac12 + it)\) is real but nonzero. (The ``almost'' is because \(0\) is not considered to be a Gram point, and it has not been proven (but certainly it is true) that the \(\zeta\)-function does not vanish at any Gram point.)%
\par
Gram~\hyperlink{gram}{[{\xreffont 61}]} noted that the first several zeros of \(Z(t)\) interlace the Gram points.  In other words, when \(t\) is small the other terms are insufficient to flip the sign of the first term at a Gram point.%
\par
In the course verifying RH for the first \(75\)~million zeros, Brent~\hyperlink{brent}{[{\xreffont 27}]} found that \(Z(t)\) was unusually large, more than \(79.6\), at the \(70\,354\,406\)th Gram point. Furthermore, at that point \emph{the first \(72\) terms in the Riemann-Siegel formula~\hyperref[RiemannSiegel]{({\xreffont\ref{RiemannSiegel}})} were positive}. \hyperref[fig_brentZ]{Figure~{\xreffont\ref{fig_brentZ}}} shows a graph of \(Z(t)\) near that Gram point.%
\begin{figureptx}{Figure}{Plot of \(Z(g_{70 354 406}+t)\) for \(-10 \lt t \lt 10\), where \(g_{70 354 406}\approx 30694257.761\) is a Gram point.}{fig_brentZ}{}%
\begin{image}{0}{1}{0}{}%
\includegraphics[width=\linewidth]{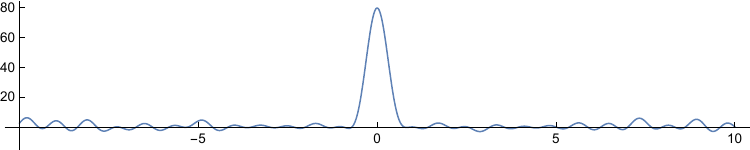}
\end{image}%
\tcblower
\end{figureptx}%
At a Gram point the first term in the Riemann-Siegel formula is \(1\) or \(-1\), and the other terms have no particular bias, so it is not surprising that~\hyperlink{T}{[{\xreffont 110}]}%
\begin{equation}
\sum_{n \le N,\ n\ \text{even}} Z(g_n) \sim N
\ \ \ \ \ \ \ \ \ \ 
\text{and}
\ \ \ \ \ \ \ \ \ \ 
\sum_{n \le N,\ n\ \text{odd}} Z(g_n) \sim -N\text{.}\label{gramaverage}
\end{equation}
Those averages are similar to the fact that the average value of \(\zeta(\tfrac12 + it)\) is~\(1\), see \hyperref[zeta1average]{({\xreffont\ref{zeta1average}})}. Such biases, combined with Selberg's CLT, are major contributors to:%
\begin{principle}{Principle}{}{}{bigOof1prin}%
For questions about the size of the \(\zeta\)-function, no numerical computation can give reliable evidence because the true nature of the \(\zeta\)-function reveals itself on the scale of \(\sqrt{\log\log T}\), which within the realm of computation is indistinguishable from a bias of order \(O(1)\).%
\end{principle}
\hyperref[bigOof1prin]{Principle~{\xreffont\ref{bigOof1prin}}}, and other sources for misinterpreting the data, are explored throughout this paper.%
\end{subsectionptx}
\typeout{************************************************}
\typeout{Subsection 4.2 Three unfounded inferences}
\typeout{************************************************}
\begin{subsectionptx}{Subsection}{Three unfounded inferences}{}{Three unfounded inferences}{}{}{misleading-4}
Brent's observations have re-appeared in subsequent numerical computations.  The result is that reinforcing those observations has been (at least partially) a goal of recent computational work, particularly as it relates to understanding the largest values of the \(\zeta\)-function. That is unfortunate because the largest values as they appear in computations are not representative of the largest values at greater height.  Thus, the impressions one has from those data are not helping to build intuition for the true nature of the \(\zeta\)-function. Quite the opposite: those numerical examples tempt one into mistaken notions.%
\par
The notions we refer to are:%
\begin{axiom}{Mistaken Notion}{}{}{mistakegap}%
The largest values of the \(\zeta\)-function occur when there is a particularly large gap between zeros.%
\end{axiom}
\begin{axiom}{Mistaken Notion}{}{}{mistakeRS}%
The largest values of the \(\zeta\)-function occur when a large number of initial terms in the Riemann-Siegel formula \hyperref[RiemannSiegel]{({\xreffont\ref{RiemannSiegel}})} have the same sign.%
\end{axiom}
\begin{axiom}{Mistaken Notion}{}{}{mistakeRH}%
Counterexamples to RH are more likely to occur near an unusually large gap between zeros.%
\end{axiom}
These Notions are part of the folklore of the \(\zeta\)-function and the author has observed conversations on these Notions at many conferences, often accompanied by refutations. Thus, it seemed prudent to address these issues explicitly.%
\par
We do not attribute those Notions as conjectures due to any specific person. Indeed, in many cases where those Notions appear in the literature (implicitly or explicitly), the writer acknowledges the Notion without necessarily endorsing it.  For example, concerning \hyperref[mistakeRS]{Mistaken Notion~{\xreffont\ref{mistakeRS}}}, Brent~\hyperlink{brent}{[{\xreffont 27}]} says ``This suggests that `interesting' regions might be predicted by finding values of \(t\) such that the first few terms in the Riemann-Siegel sum reinforce each other.'' The ``interesting'' behavior of an unusually large gap does typically occur when the initial terms reinforce, at least in the range accessible by computers. That Notion is not Mistaken in the limited context of computing the \(\zeta\)-function at moderate height.%
\par
The even more ``interesting'' behavior of a possible counterexample to RH requires a second step to get to \hyperref[mistakeRH]{Mistaken Notion~{\xreffont\ref{mistakeRH}}}. Odlyzko~\hyperlink{Odl}{[{\xreffont 88}]} explains it this way: ``One reason for the interest in large values of \(\zeta (1/2 + it)\) is that one could think of a large peak as `pushing aside' the zeros that would normally lie in that area, and if these zeros were pushed off of the critical line, one would find a counterexample to the RH.'' Odlyzko is not endorsing \hyperref[mistakeRH]{Mistaken Notion~{\xreffont\ref{mistakeRH}}}, merely noting the reasoning which might have led to its formulation.%
\par
Brent's and Odlyzko's use of ``scare quotes'' is a further indication that they are not endorsing those Notions as fundamental principles. Indeed, it may be that individually the vast majority of experts understand the limitations of what can be learned from existing computations. But the frequency with which those Notions have been repeated, not always with qualifiers such as ``it has been said that...'', indicates the need for clarification. In \hyperref[misleadingrevisited]{Section~{\xreffont\ref{misleadingrevisited}}} we revisit these Notions in the context of the Principles discussed in this paper, justifying the claim that the Notions are indeed ``Mistaken''.%
\end{subsectionptx}
\typeout{************************************************}
\typeout{Subsection 4.3 Gram's law}
\typeout{************************************************}
\begin{subsectionptx}{Subsection}{Gram's law}{}{Gram's law}{}{}{misleading-5}
If \((g_n)\) are the Gram points for the \(Z\)-function, then we say \terminology{Gram's law holds} for the \terminology{Gram interval \((g_m, g_{m+1})\)} if that interval contains exactly one zero of the \(Z\)-function. Gram noted that this law holds for the first \(15\) zeros, and Hutchinson~\hyperlink{hutch}{[{\xreffont 70}]} found that it holds for the first \(126\) zeros, and he also coined the term ``Gram's law''. Subsequent work found that it holds for more than 91\% of the first \(1000\) zeros, and most papers on zero calculations report statistics on Gram's law.  Such numerical evidence, and repeated attention, leads to the archetypal mistaken notion concerning zeros:%
\begin{axiom}{Mistaken Notion}{}{}{misleadinggramslaw}%
Gram points are special, and Gram's law is helpful for developing intuition about the location of zeros of the \(\zeta\)-function.%
\end{axiom}
It has been known for a long time that Gram's law fails infinitely often \hyperlink{T}{[{\xreffont 110}]} and in fact fails a positive proportion of the time (although there are good reasons to believe it is true more than 66\% of the time, see ~[\hyperlink{TT1}{{\xreffont 111}}, \hyperlink{HH}{{\xreffont 62}}]). Gram's law is not mistaken merely because it is not always true: it is mistaken because it paints a picture of the zeros which inhibits one from gaining a proper intuition about the behavior of the \(\zeta\)-function at large height.  Also, as we will explain in \hyperref[gramslawrevisited]{Subsection~{\xreffont\ref{gramslawrevisited}}}, if Gram's law is stated properly, it is true 0\% of the time.%
\end{subsectionptx}
\end{sectionptx}
\typeout{************************************************}
\typeout{Section 5 Separating out the local zero spacing}
\typeout{************************************************}
\begin{sectionptx}{Section}{Separating out the local zero spacing}{}{Separating out the local zero spacing}{}{}{waves}
\begin{introduction}{}%
\hyperref[fig_z5000]{Figure~{\xreffont\ref{fig_z5000}}} and \hyperref[fig_boberhiary]{Figure~{\xreffont\ref{fig_boberhiary}}} show that the local spacing of the zeros has an influence on the size of the \(\zeta\)-function: when there is a large gap the function is larger, and when there is a small gap the function stays small. Unfortunately, those observations require one to \emph{ignore the vertical scale}.  The specific spacing of nearby zeros strongly influences the \emph{relative} sizes of the \emph{nearby} maxima and minima.  But the local spacings have very little to do with whether or not the actual function values are particularly large or small, compared to what one would expect for the \(\zeta\)-function in that region.%
\par
Indeed, by the end of this section we will see that the local spacing of zeros is not the leading order contribution to the size of the \(\zeta\)-function.%
\par
In general, the size of the  \(\zeta\)-function is controlled by the \(\zeta\)-function's \terminology{carrier wave}.  The terminology is due to Hejhal~\hyperlink{Hej}{[{\xreffont 65}]}, with detailed discussion by Bombieri and Hejhal~\hyperlink{BomHej}{[{\xreffont 24}]}.  The idea of carrier waves is based on some unpublished speculations of H.L.~Montgomery. The main idea is that, on a logarithmic scale, the \(\zeta\)-function changes its size slowly.  If it is large, then outside a set with small measure, it usually stays large for a while. (The ``set with small measure'' is small neighborhoods of the zeros.) If it is small, it usually stays small for a while.  Here ``a while'' means ``on an interval containing many zeros''. In particular, the typical large values occur as clusters of large local maxima and minima, not as a single isolated large maximum. That is contrary to what we see in graphs of \(Z(t)\), whence \hyperref[mistakegap]{Mistaken Notion~{\xreffont\ref{mistakegap}}}. That is why for many people, this idea is in the ``I find that hard to believe'' category. Indeed, the phenomenon of carrier waves is not visible in any numeric computation of the \(\zeta\)-function, because the scale of the carrier waves, just like the scale of \(S(t)\), grows so slowly that it appears bounded within the range we can compute.  However, we can take a first-principles approach to defining what we mean by ``carrier wave'', and  then build intuition by looking at illustrative examples.%
\end{introduction}%
\typeout{************************************************}
\typeout{Subsection 5.1 The wave as a local constant factor}
\typeout{************************************************}
\begin{subsectionptx}{Subsection}{The wave as a local constant factor}{}{The wave as a local constant factor}{}{}{localconstant}
Suppose \(f(t)\) is a high degree polynomial, with zeros \(\gamma_1 \lt \ldots \lt \gamma_M\), and suppose we want to understand the graph of \(f\) near \(t_0\).  Further suppose that \(t_0\) is roughly near the middle of the~\(\gamma_j\).  We can write%
\begin{align}
f(t) =\mathstrut \amp a_0 \prod_{\gamma_j \text{ near } t_0} \left(1 - \frac{t}{\gamma_j}\right)
\prod_{\gamma_j \text{ far from } t_0} \left(1 - \frac{t}{\gamma_j}\right)\label{localconstant-2-8-1}\\
\approx\mathstrut \amp A_0 \prod_{\gamma_j \text{ near } t_0} \left(1 - \frac{t}{\gamma_j}\right)
\ \ \ \ \ \
\text{for } t \text{ very close to } t_0\text{.}\label{eqn_neart0}
\end{align}
The approximation in the second line above will be valid for \(t\) in a neighborhood of \(t_0\) if the zeros far from \(t_0\) are balanced on either side of \(t_0\), meaning that the product over ``\(\gamma_j\) far from \(t_0\)'' is approximately constant near \(t_0\).%
\par
\(L\)-functions are much like high degree polynomials, and the global spacing of their zeros is very regular (as can be seen from the small error term in the zero counting function \(N(T)\), see \hyperref[NT]{({\xreffont\ref{NT}})}). Thus we have:%
\begin{principle}{Principle}{}{}{prin_three_things}%
The behavior of \(Z(t)\) near \(t_0\) depends on three things: a global factor independent of \(t_0\), the arrangement of the zeros near \(t_0\), and a scale factor which depends on the zeros far from \(t_0\) and which does not change too quickly as a function of \(t_0\).%
\end{principle}
\hyperref[fig_sinezeta5000]{Figure~{\xreffont\ref{fig_sinezeta5000}}} illustrates some of the ideas in \hyperref[prin_three_things]{Principle~{\xreffont\ref{prin_three_things}}}.%
\par
\hyperref[prin_three_things]{Principle~{\xreffont\ref{prin_three_things}}} does not specify which of the two factors that depend on \(t_0\) are most relevant to the size of \(Z(t)\). Any graph of \(Z(t)\) in the range accessible by current computers, such as in \hyperref[fig_boberhiary]{Figure~{\xreffont\ref{fig_boberhiary}}}, makes it appear that the arrangement of the zeros is more important. We will see that those examples are misleading.%
\end{subsectionptx}
\typeout{************************************************}
\typeout{Subsection 5.2 Carrier waves}
\typeout{************************************************}
\begin{subsectionptx}{Subsection}{Carrier waves}{}{Carrier waves}{}{}{carrierwaves}
The scale factor \(A_0 = A_0(t_0)\) has been termed the \terminology{carrier wave} by Hejhal \hyperlink{Hej}{[{\xreffont 65}]} and Bombieri and Hejhal \hyperlink{BomHej}{[{\xreffont 24}]}, making rigorous a speculation of Montgomery.  Montgomery's preliminary calculations suggested that the carrier wave for the Riemann zeta-function should not vary too much over a window of width \(\exp(\delta_T {\log\log T})/\log T\), for some function \(\delta_T \to 0\). Bombieri and Hejhal \hyperlink{BomHej}{[{\xreffont 24}]} show that, for most \(T\), the carrier wave does not vary significantly over a window of width \(M/\log T\), for any fixed \(M \gt 0\), as \(T\to\infty\).  That is: across a span of \(M\) consecutive zeros, for any fixed \(M\), the size of \(\log\abs{\zeta(\frac12 + i t)}\) usually varies very little.  (The proof in \hyperlink{BomHej}{[{\xreffont 24}]} might actually show that one can take \(M=\log\log(T)^\kappa\) for any \(\kappa \lt \frac14\), with (6.21) in that paper providing the limiting constraint, but those details would need to be checked.) See \hyperref[how_wide]{Subsection~{\xreffont\ref{how_wide}}} for a discussion of how wide is the carrier wave.%
\par
Bombieri and Hejhal used carrier waves as the key idea toward their proof of the following surprising theorem: if \(L_1(s)\) and \(L_2(s)\) are \(L\)-functions which individually satisfy RH and have the same functional equation, and \(\alpha\in\R\), then \(L_1(s) + \alpha L_2(s)\) has 100\% of its zeros on the critical line (there are technical conditions which we have omitted).  The proof is: the \(\log(L_j(\frac12 + i t))\) are independently and normally distributed with a large variance, so most of the time one of the \(L_j\) is much larger than the other, and furthermore (because of the carrier wave) it stays larger across an arbitrarily many zeros.  Therefore in that region the zeros of the linear combination are very close to the zeros of the larger \(L\)-function, and that accounts for most of the zeros.  The large variance is an important ingredient: given two Gaussians with variance \(\sigma^2\), the probability that they differ by less than \(\sqrt{\sigma}\) is \(O(\sigma^{-\frac12})\).  The above argument requires \(\sigma \to \infty\).%
\par
Note that Selberg's result on the normal distribution of \(L(\frac12 + it)\) is not sufficient: one needs the additional fact that the carrier wave causes the larger \(L\)-function to stay large over a significant range.%
\par
Thus we have:%
\begin{principle}{Principle}{}{}{prin_carrier}%
The carrier wave is responsible for the bulk of the value distribution of an \(L\)-function, with the variation due to the local zero spacings playing a secondary role. In particular, it is the carrier wave which obeys Selberg's CLT.%
\end{principle}
\hyperref[prin_carrier]{Principle~{\xreffont\ref{prin_carrier}}} explains why, in the range accessible by current computers, the observed value distribution of \(\log\abs{Z(t)}\) departs significantly from normal. The local zero spacing contributes a lesser (typically, bounded) amount.  But a bounded amount is significant in the range where the carrier waves are very small.%
\end{subsectionptx}
\typeout{************************************************}
\typeout{Subsection 5.3 Measuring the wave}
\typeout{************************************************}
\begin{subsectionptx}{Subsection}{Measuring the wave}{}{Measuring the wave}{}{}{waves-5}
We now describe a way to measure and observe the carrier wave. The goal is to isolate the contribution of the nearby zeros, as suggested in \hyperref[eqn_neart0]{({\xreffont\ref{eqn_neart0}})}. The idea is to think of the zeros as parameters which can change: we can slide the zeros side-to-side, and this will cause a change in the graph of \(Z(t)\). (This is easier to picture if all the zeros are real, which we will assume for the purposes of this thought experiment.) If we slide the zeros apart, making a larger zero gap, then the function will acquire a larger local maximum.  If we slide the zeros so they become more equally spaced, then the maxima of the function will be approximately the same size.  In the extreme case of moving the zeros to be equally spaced, the result will be the (scaled and shifted) cosine function.%
\par
It is not possible to achieve a \emph{globally} equal spacing for the zeros, because the local average spacing of the zeros is not constant.  Instead we focus on ``nearby'' zeros, where the average spacing is close to constant. Suppose \(t_0 \in [\gamma_{M}, \gamma_{M + 1}]\). If \(K\) is large enough and we slide the zeros \(\gamma_{M-K},\ldots, \gamma_{M + K}\) to be equally spaced, then near \(t_0\) the resulting function will look like \(a \cos(b (t + c))\) for some \(a, b, c \in \R\).%
\par
We will run the above process in reverse.  That is, start with \(a \cos(b (t + c))\), pick a region of interest in the critical strip containing zeros \(\gamma_{M-K},\ldots,\gamma_{M+K}\), and then ``move'' the cosine zeros to the \(\zeta\)-zeros. This suggests that%
\begin{equation}
Z(t) \approx A_Z\, a \cos(b (t + c)) \prod_{j=M-K}^{M+K} \frac{t-\gamma_j}{t-g_j}\text{,}\label{eqn_sineapprox}
\end{equation}
 where the \(g_j\) are zeros of \(\cos(b (t + c))\), indexed so that \(g_j\) is close to \(\gamma_j\), the parameters \(a, b, c\) depend on \(t_0\) and \(K\), and \(A_Z\) is a normalization factor. The approximation should be good near \(t_0\) if \(K\) is large enough and \(t_0\) is near the middle of the interval \([\gamma_{M-K}, \gamma_{M+K+1}]\). This approximation works because locally the \(Z\)-function is determined by the nearby zeros and a scale factor.%
\par
A similar approximation was considered by Hiary and Odlyzko (\hyperlink{HiOd}{[{\xreffont 67}]}, Section~6), who refer to it as ``HP'', because essentially it uses a partial Hadamard Product for the approximation.  They find that the quality of the approximation improves linearly with the number of zeros in the product. A difference in the approach here is to recognize that the zeros far from \(t_0\) are basically contributing a multiplicative constant near~\(t_0\).%
\par
We apply \hyperref[eqn_sineapprox]{({\xreffont\ref{eqn_sineapprox}})} to \(Z(t)\) from \hyperref[fig_z5000]{Figure~{\xreffont\ref{fig_z5000}}}, with%
\begin{align}
t_0=\mathstrut \amp 5448\notag\\
M =\mathstrut \amp  5000\notag\\
\gamma_M \approx\mathstrut \amp 5447.8619\notag\\
K =\mathstrut \amp  10\notag\\
A_Z=\mathstrut \amp 2\notag\\
a=\mathstrut \amp 1.0991\notag\\
b=\mathstrut \amp 3.3825\label{bparameter}\\
c=\mathstrut \amp  0.17008 \text{.}\notag
\end{align}
The parameter \(A_Z = 2\) is explained after \hyperref[wavenotdefined]{Note~{\xreffont\ref{wavenotdefined}}}, as is the fact that \(b\) and \(c\) are actually simple functions of \(t_0\).  The parameter \(a\) is chosen so that the approximation is exact at \(t_0\). So, once \(t_0\) is selected, the only choice is \(K\), the number of zeros to be matched on either side of~\(t_0\). The result is shown in the top plot in \hyperref[fig_sinezeta5000]{Figure~{\xreffont\ref{fig_sinezeta5000}}}, where we superimpose \(Z(t)\) and its approximation based at \(t_0 = 5447.86\) using \(20\) matched zeros.. The bottom plot shows the same idea based at \(t_0 = \gamma_{7010} \approx 7273.70\), which has scale factor \(a=0.9130\).  In both cases one sees that the behavior near \(t_0\) is determined by the nearby zeros and a local scale factor (which depends on \(t_0\)).%
\begin{figureptx}{Figure}{The function \(Z(t)\) together with a local approximation as in \hyperref[eqn_sineapprox]{({\xreffont\ref{eqn_sineapprox}})} with \(K=10\). The top graph has \(t\) near \(5447.86\) and the local scale factor is \(1.0991\).  The bottom graph has \(t\) near \(7273.70\) and the scale factor is \(0.913\).}{fig_sinezeta5000}{}%
\begin{image}{0}{1}{0}{}%
\includegraphics[width=\linewidth]{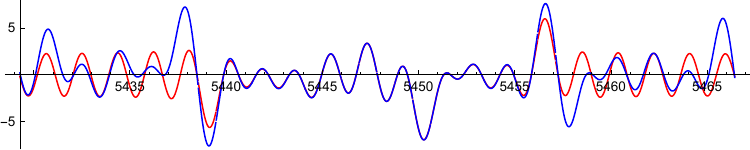}
\end{image}%
\begin{image}{0}{1}{0}{}%
\includegraphics[width=\linewidth]{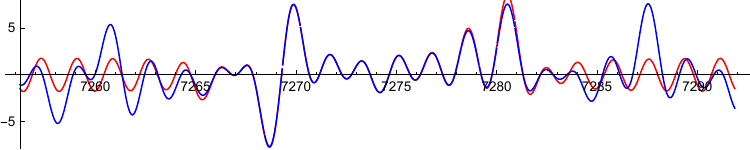}
\end{image}%
\tcblower
\end{figureptx}%
In the top plot in \hyperref[fig_sinezeta5000]{Figure~{\xreffont\ref{fig_sinezeta5000}}} we can interpret the scale factor \(a = 1.0991\) as the magnitude of the carrier wave at \(t_0 = 5447.86\), and similarly for the scale factor \(a = 0.9130\) at \(t_0 = 7273.70\). Those scale factors are different, which shows that the size of \(Z(t)\) is not solely due to the local zero spacing and a single global factor.%
\par
Here we are focusing on \hyperref[eqn_sineapprox]{({\xreffont\ref{eqn_sineapprox}})} because it explicitly separates the contribution of the local zeros from the contribution of the carrier wave. There are other ways to approximate the \(\zeta\)-function, see for example \hyperref[sec_hybrid]{Subsection~{\xreffont\ref{sec_hybrid}}}.%
\end{subsectionptx}
\typeout{************************************************}
\typeout{Subsection 5.4 Some caveats}
\typeout{************************************************}
\begin{subsectionptx}{Subsection}{Some caveats}{}{Some caveats}{}{}{waves-6}
The scale factors in \hyperref[fig_sinezeta5000]{Figure~{\xreffont\ref{fig_sinezeta5000}}} are not the carrier wave in the sense of Bombieri-Hejhal, because Montgomery's heuristic calculation suggests that at the small height of that example the carrier wave should only be approximately constant in a very narrow window. In \hyperref[fig_sinezeta5000]{Figure~{\xreffont\ref{fig_sinezeta5000}}} we have fit to the zeros in a much wider window. Nevertheless, those plots illustrate that the nearby zeros and a single local scale factor determine the local behavior of \(Z(t)\).%
\begin{note}{Note}{}{wavenotdefined}%
The scale factor at a point is not well-defined: it depends on \(K\), the number of zeros, or more precisely the width of the window, over which one has done the fit.  If the window covers an enormous number of zeros then the local scale factor will just be~\(1\), because of the long-range rigidity of the zeros; see \hyperref[sec_density]{Subsection~{\xreffont\ref{sec_density}}}. If the function changes scale in the window, i.e., the carrier wave is not approximately constant in the window, then the scale factor is not providing useful information. This is discussed further in \hyperref[densitynotunique]{Note~{\xreffont\ref{densitynotunique}}}.%
\end{note}
It remains to justify that the global scale factor (independent of \(t_0\)) is  \(A_Z = 2\). That comes from the overall factor of \(2\) in the Riemann-Siegel formula \hyperref[RiemannSiegel]{({\xreffont\ref{RiemannSiegel}})}.%
\par
The parameters \(b\) and \(c\) in \hyperref[eqn_sineapprox]{({\xreffont\ref{eqn_sineapprox}})} are also extraneous: instead of \(A_Z \cos(b (t + c))\) we could use the first term in the Riemann-Siegel formula:%
\begin{equation*}
2 \cos(\theta(t))\text{.}
\end{equation*}
Indeed \(\theta'(5447.86) \approx 3.38255\), which is the value for \(b\) in~\hyperref[bparameter]{({\xreffont\ref{bparameter}})}. Using \(2\cos(\theta(t))\) in \hyperref[eqn_sineapprox]{({\xreffont\ref{eqn_sineapprox}})} does not change the point illustrated by \hyperref[fig_sinezeta5000]{Figure~{\xreffont\ref{fig_sinezeta5000}}}.%
\par
In order to really ``see'' the carrier wave, one must go to enormously larger values of \(t\).  That is not computationally feasible in the \(L\)-function world, but it is easy in the random matrix world. In \hyperref[rmt_connections]{Section~{\xreffont\ref{rmt_connections}}} and \hyperref[RMThistory]{Section~{\xreffont\ref{RMThistory}}} we review the connections between the \(\zeta\)-function and the characteristic polynomials of random unitary matrices, returning to carrier waves in \hyperref[rmtwaves]{Section~{\xreffont\ref{rmtwaves}}}.%
\end{subsectionptx}
\typeout{************************************************}
\typeout{Subsection 5.5 The highest tone}
\typeout{************************************************}
\begin{subsectionptx}{Subsection}{The highest tone}{}{The highest tone}{}{}{sec_highest_note}
Many observations about the zeta-function can be traced to the fact that \(S(t)\), the error term in the zero counting function \(N(t)\), grows very slowly and is zero on average.  Here we collect some additional consequences.%
\begin{principle}{Principle}{}{}{zetamusic}%
The \(Z\)-function, interpreted as a sound wave made up of separate tones, has a highest frequency component which is the loudest tone and which is separated in frequency from the lower tones.  In particular, the \(Z\)-function is (approximately) a band-limited function with discrete support.%
\end{principle}
Specifically, at height \(t\) the highest frequency is \(\frac12 \log(t/2\pi)\), the factor of 1\slash{}2 coming from the fact that \(\sin\) or \(\cos\) have two zeros in each period.  Each subsequent frequency is lower by \(\frac12\log(2)\), \(\frac12\log(3)\), ..., and its contribution is smaller by a factor \(1/\sqrt{2}\), \(1/\sqrt{3}\),.... These observations come straight from the Riemann-Siegel formula \hyperref[RiemannSiegel]{({\xreffont\ref{RiemannSiegel}})}, and are illustrated in \hyperref[fig_zhatplot]{Figure~{\xreffont\ref{fig_zhatplot}}}, which shows the Fourier transform of \(Z(t)\) for \(t\approx 10^6\). Note that \(\log(10^6/2\pi) \approx 11.978\).%
\begin{figureptx}{Figure}{The Fourier transform of \(Z(t)\), computed numerically, for \(10^6 \le t \le 10^6 + 5000\).}{fig_zhatplot}{}%
\begin{image}{0}{1}{0}{}%
\includegraphics[width=\linewidth]{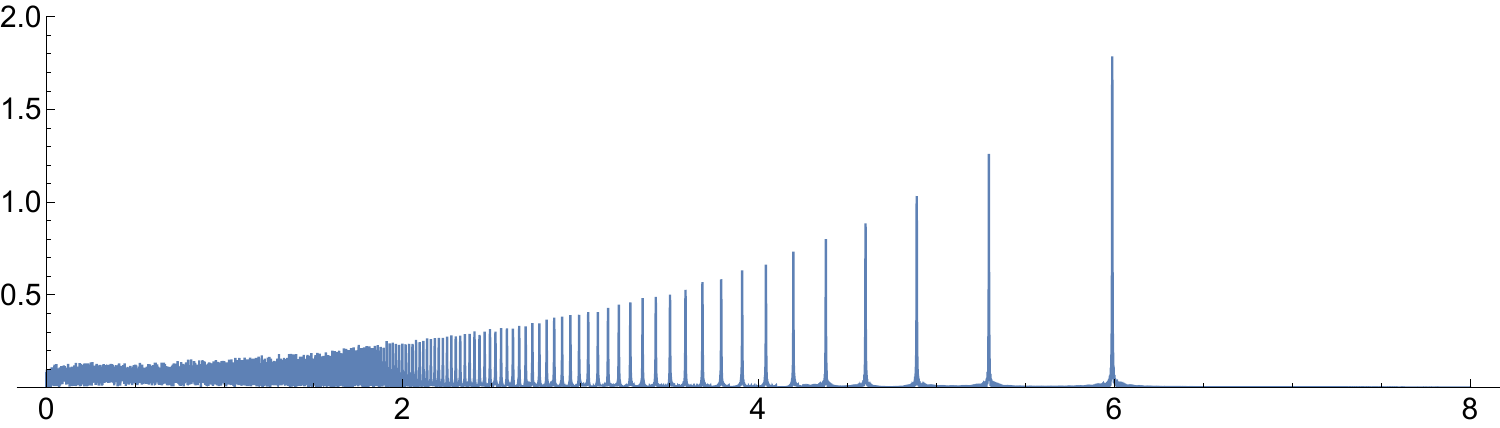}
\end{image}%
\tcblower
\end{figureptx}%
The \(Z\)-function is not actually band-limited, a fact which is relevant to superoscillations \textemdash{} the phenomenon where portions of a finite Fourier series can oscillate faster that its fastest component, see [\hyperlink{BerrySO}{{\xreffont 14}}, \hyperlink{BerryFTZ}{{\xreffont 15}}]. However, it is close enough to band-limited that band-limited interpolation has become an indespensible tool in large-scale computations of the zeros  [\hyperlink{OdlOld}{{\xreffont 86}}, \hyperlink{BobHia}{{\xreffont 21}}].%
\begin{principle}{Principle}{}{}{gammazeros}%
The highest tone of the \(Z\)-function comes from the \(\Gamma\)-factor in the functional equation, and the lower tones come from the terms in the Dirichlet series. For general \(L\)-functions, the \(\Gamma\)-factor and the sign of the functional equation place the zeros in well-spaced locations, and then the Dirichlet coefficients make adjustments.%
\end{principle}
To see that interpretation, view the Riemann-Siegel formula \hyperref[RiemannSiegel]{({\xreffont\ref{RiemannSiegel}})} as arising from the expression \(Z(t) = X(\frac12 + i t)^{-1/2} \zeta(\frac12 + it)\) for the \(Z\)-function in terms of the zeta-function, see \hyperref[eqn_Zdef]{({\xreffont\ref{eqn_Zdef}})}, combined with the \terminology{approximate functional equation}, (4.12.4) of \hyperlink{T}{[{\xreffont 110}]},%
\begin{equation}
\zeta(s) = \sum_{1\le n \le x} \frac{1}{n^s} + X(s)\sum_{1\le n \le y} \frac{1}{n^{1-s}} + \text{ correction terms},\label{sec_highest_note-8-7}
\end{equation}
where \(x y = t/2\pi\).  Setting \(x = y\), combining terms, and using Stirling's formula for the \(\Gamma\)-functions in \(X(\frac12 + it)\) yields the Riemann-Siegel formula. The \(n=1\) terms contribute \(2\, \Re X(\frac12 + it)^{-1/2}\), which is bounded by \(2\) and has regularly spaced zeros with the same counting function as the zeros of the zeta function (but with error term bounded by~\(1\)). Those zeros interlace the Gram points. We see that Gram's law is a consequence of a more general principle:%
\begin{principle}{Principle}{The first term dominates at low height.}{}{pre_grams_law}%
For \(\zeta(\frac12 + i t)\), the first term in the Riemann-Siegel formula has a strong influence on the location of the critical zeros when \(t\) is small. For general \(L\)-functions, if the analytic conductor is small then the first term in its Riemann-Siegel formula has a strong influence on the location of the critical zeros.%
\end{principle}
We used the term ``analytic conductor'' so that the principle applies more widely. For the \(\zeta\)-function at \(s=\frac12 + i t\), the analytic conductor is proportional to \(\log(2+\abs{t})\).  \hyperref[fig_first_term]{Figure~{\xreffont\ref{fig_first_term}}} illustrates the principle by graphing the \(\zeta\)-function and and the first term in its Riemann-Siegel formula, over the interval \([1000,1050]\). The Gram points are the locations of the local maxima and local minima of the red curve which oscillates between \(-2\) and~\(2\). Gram's law (see \hyperref[gramslawrevisited]{Subsection~{\xreffont\ref{gramslawrevisited}}} for further discussion) follows from the expectation that at low height, the zeros of the \(\zeta\)-function will be close to the zeros of the first term.%
\begin{figureptx}{Figure}{The functions \(Z(t)\) and \(2 \cos(\vartheta(t))\) on the interval \(1000 \le t \le 1050\).}{fig_first_term}{}%
\begin{image}{0}{1}{0}{}%
\includegraphics[width=\linewidth]{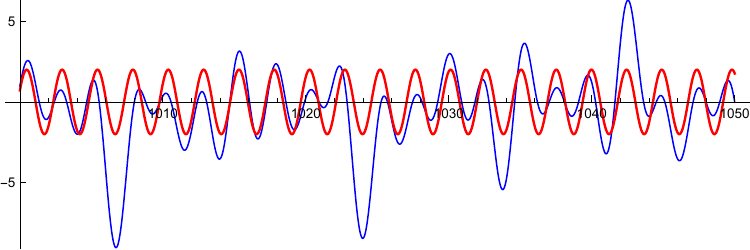}
\end{image}%
\tcblower
\end{figureptx}%
There is a Riemann-Siegel formula for any \(L\)-function of degree~\(d\), with the modification that \(x y\) is of size~\(t^d\), the numerators of the two halves of the approximate functional equation are \(a_n\) and \(\overline{a}_n\), respectively, and \(X(s)\) is replaced by \(\varepsilon X(s)\). For large \(t\) there is highest frequency term, with frequency asymptotically \(\frac{d}{2} \log t\), involving only the data in the functional equation. And just like in the classical case, initially the Dirichlet coefficients make small adjustments to the location of the zeros arising from the \(\Gamma\)-factors. The \(\Gamma\)-factors have a surprising influence on the initial zeros, particularly for nonarithmetic \(L\)-functions, whose trivial zeros have nonzero imaginary part. That influence can be captured by working directly with \(X(\frac12 + i t)^{-1/2}\) instead of using the asymptotics from Stirling's formula. See \hyperlink{FKLR}{[{\xreffont 53}]}.%
\end{subsectionptx}
\end{sectionptx}
\typeout{************************************************}
\typeout{Section 6 A random model for the \(\zeta\)-function}
\typeout{************************************************}
\begin{sectionptx}{Section}{A random model for the \(\zeta\)-function}{}{A random model for the \(\zeta\)-function}{}{}{rmt_connections}
\begin{introduction}{}%
A graph of \(Z(t)\) gives the impression of randomness: the function wiggles, and there is no apparent pattern to those wiggles. Obviously \(Z(t)\) is not random, because it is a specific function which has no randomness in its definition. But we can make a random function in the following way.  Fix an interval \(\Upsilon = [t_0,t_1] \subset \mathbb R\). If \(T\in \mathbb R\) is random, then \(Z_T(t) := Z(T+t)\) is a random function on~\(\Upsilon\).%
\par
If we had another set of random functions \(\mathcal Z_T\) on \(\Upsilon\), and it was possible to prove theorems about \(\mathcal Z_T\), and if furthermore we had reason to believe that \(Z_T\) and \(\mathcal Z_T\) had similar properties, then we could turn theorems about \(\mathcal Z_T\) into conjectures about~\(Z_T\).  That would be illuminating, particularly if precise conjectures about \(Z_T\) were in short supply. That is how random matrix theory made fundamental contributions to the study of \(L\)-functions.%
\end{introduction}%
\typeout{************************************************}
\typeout{Subsection 6.1 Self-reciprocal polynomials and the functional equation}
\typeout{************************************************}
\begin{subsectionptx}{Subsection}{Self-reciprocal polynomials and the functional equation}{}{Self-reciprocal polynomials and the functional equation}{}{}{rmt_connections-3}
Before getting into the details, let's consider a certain class of polynomials which have constant term \(1\):%
\begin{equation}
f(z) = 1 + a_1 z + a_2 z^2 + \cdots + a_{N-1} z^{N-1} +  a_N z^N .\label{srpoly}
\end{equation}
The polynomial \(f\) is \terminology{self-reciprocal} if%
\begin{equation}
a_j = a_N \overline{a_{N-j}}
\ \ \ \ \
\text{for}
\ \ \ \ \
0\le j \le N\text{,}\label{self_recip_coeffs}
\end{equation}
or equivalently%
\begin{equation}
f(z) = \mathcal{X}(z) \overline{f}(z^{-1})
\ \ \ \ \ \text{where} \ \ \ \ \ \mathcal{X}(z) = a_N z^N\text{,}\label{eqn_RMTfe}
\end{equation}
or equivalently%
\begin{equation}
\mathcal Z_f(\theta) = \mathcal X(e^{i\theta})^{-\frac12} f(e^{i\theta})
\ \ \ \ \text{is real if}\ \ \ \ \theta\in \R\text{.}\label{eqn_srZ}
\end{equation}
An important consequence of any of these conditions is that%
\begin{equation}
f(r e^{i\theta})=0  \ \ \ \ \
\text{if and only if}
\ \ \ \ \
f(r^{-1} e^{i\theta})=0\text{.}\label{self_recip_zeros}
\end{equation}
In \hyperref[eqn_RMTfe]{({\xreffont\ref{eqn_RMTfe}})}, \(\overline{f}\) is the \terminology{Schwarz reflection} of~\(f\): \(\overline{f}(z) := \overline{f(\overline{z})}\). The parallel with the functional equation for the \(\zeta\)-function is evident if one compares \hyperref[eqn_RMTfe]{({\xreffont\ref{eqn_RMTfe}})} to \hyperref[zeta_fe]{({\xreffont\ref{zeta_fe}})}, and \hyperref[eqn_srZ]{({\xreffont\ref{eqn_srZ}})} to \hyperref[eqn_Zdef]{({\xreffont\ref{eqn_Zdef}})}. At the risk of belaboring the point, in both cases there is an involution of the plane, the set of zeros is fixed by the involution, the \terminology{Riemann Hypothesis} is the assertion that every zero is fixed by the involution, and there is a Z-function which is real on the set of points which are fixed by the involution. The unit circle is the analogue of the critical line for self-reciprocal polynomials. Thus, if a collection of polynomials is claimed to be a random model for the \(\zeta\)-function, those polynomials must be self-reciprocal.%
\end{subsectionptx}
\typeout{************************************************}
\typeout{Subsection 6.2 Random polynomials and trigonometric polynomials}
\typeout{************************************************}
\begin{subsectionptx}{Subsection}{Random polynomials and trigonometric polynomials}{}{Random polynomials and trigonometric polynomials}{}{}{rmt_connections-4}
We seek random self-reciprocal polynomials which can serve as a model for \(\zeta(s)\) or equivalently~\(Z(t)\).  How to choose the randomness? There are three reasonable options.%
\par
The first option is to choose the coefficients of \(f(z)\) so that they satisfy \hyperref[self_recip_coeffs]{({\xreffont\ref{self_recip_coeffs}})}.  The second is to recognize that \(Z_f\) is a trigonometric polynomial (i.e.\@ a finite Fourier series), so its Fourier coefficients can be chosen randomly with only the requirement that they are real.  The third option is to choose the zeros so they satisfy \hyperref[self_recip_zeros]{({\xreffont\ref{self_recip_zeros}})}. We will argue that only the third option is reasonable.%
\par
\hyperref[fig_BH68]{Figure~{\xreffont\ref{fig_BH68}}} shows \(Z(t)\) at height \(1.5 \times 10^{30}\), in a region where nothing special is happening.  (Data courtesy of Bober and Hiary~\hyperlink{BobHia}{[{\xreffont 21}]}\hyperlink{Hia}{[{\xreffont 66}]}.)%
\begin{figureptx}{Figure}{\(Z(t)\) on an interval of width \(2 \pi\) near \(1.5 \times 10^{30}\)}{fig_BH68}{}%
\begin{image}{0}{1}{0}{}%
\includegraphics[width=\linewidth]{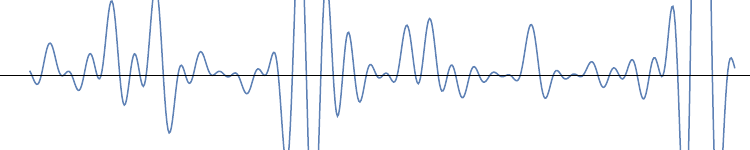}
\end{image}%
\tcblower
\end{figureptx}%
We wish create random polynomials which capture the type of randomness visible in~\hyperref[fig_BH68]{Figure~{\xreffont\ref{fig_BH68}}}. But what does that mean? The following features are worthy of attention.%
\begin{enumerate}[label={(\Alph*)}]
\item\hypertarget{RH68}{}Do all the zeros appear to be real?%
\item{}How often does the graph get big? Almost equivalently, how often are there large gaps between zeros?%
\item{}How often are pairs of zeros close together?  How often are there clusters of close zeros?%
\item\hypertarget{generallynice}{}Are the maxima and minima generally around the same size, with only occasional much larger or smaller local extrema?%
\end{enumerate}
\hyperlink{generallynice}{Item~{\xreffont D}}  is obviously a leading question. The answer is ``yes''.%
\par
\hyperlink{RH68}{Item~{\xreffont A}} refers to~RH.  We can't escape the fact that RH is true within the realm where we can experiment, so the random polynomials we seek must satisfy RH.  That is, we seek \terminology{unitary polynomials}, which is standard terminology for polynomials with all zeros on \(\abs{z}=1\).%
\par
The requirement of generating unitary polynomials effectively rules out the first two options for choosing random polynomials. If the coefficients are chosen randomly, then one can choose small coefficients to obtain a perturbation of \(z^N - 1\).  The zeros will be on the unit circle, but they will be close to regularly spaced and the function will not have randomness similar to~\hyperref[fig_BH68]{Figure~{\xreffont\ref{fig_BH68}}}.  If the coefficients are random and large, then the polynomial will usually not be unitary (although it may have a large proportion of its zeros on the unit circle). Choosing the coefficients randomly and then conditioning on having all zeros on the unit circle turns out to produce the wrong type of randomness~\hyperlink{FMS}{[{\xreffont 50}]}.  (It gives the \(\COE\), not the \(\CUE\). See \hyperref[CbetaE]{Subsection~{\xreffont\ref{CbetaE}}} for terminology.)%
\par
By process of elimination we conclude that the polynomials we seek must arise by having their zeros randomly generated, and on the unit circle.  Such polynomials have arisen in mathematical physics, which we describe next.%
\begin{note}{Note}{}{rmt_connections-4-10}%
It would be interesting to define a natural-looking ensemble of random self-reciprocal polynomials which are unitary most of time, but not always. If such an ensemble was consistent with RH being true within the realm of current computation, and it predicted a height at which RH would start to fail, that could throw doubt on~RH. Indeed, one of the great shortcomings of RH skepticism is the lack of conjectures for: the height at which RH fails, how often zeros would be off the line, and how their real parts would be distributed. \hyperref[snapshots]{Principle~{\xreffont\ref{snapshots}}} implies that even if RH fails, 100\% of the zeros are on the critical line.%
\end{note}
\end{subsectionptx}
\typeout{************************************************}
\typeout{Subsection 6.3 The circular \(\beta\)-ensembles}
\typeout{************************************************}
\begin{subsectionptx}{Subsection}{The circular \(\beta\)-ensembles}{}{The circular \(\beta\)-ensembles}{}{}{CbetaE}
The \terminology{circular \(\beta\)-ensemble} consists of \(N\) random points \(e^{i\theta_1},\ldots,e^{i\theta_N}\) on the unit circle,  equipped with the joint probability density function%
\begin{equation}
\frac{\Gamma(1 + \beta/2)^N}{\Gamma(1 + N\beta/2) (2\pi)^N} 
\prod_{1\le j \lt k \le N} |e^{i\theta_j} - e^{i\theta_k}|^\beta \text{.}\label{betameasure}
\end{equation}
This probability space is denoted~\(\CbetaE(N)\). We use \(\langle f(\theta_1,\dots,\theta_N) \rangle\) for the \terminology{expected value} of \(f\) averaged over \(\CbetaE(N)\).%
\par
The \(\CbetaE(N)\) is defined for any \(\beta \ge 0\) and any positive integer \(N\). We can obtain a random unitary (hence, self-reciprocal) polynomial by taking the points to be the zeros of a polynomial with constant term~\(1\), as in \hyperref[srpoly]{({\xreffont\ref{srpoly}})}.%
\par
If we write the \(\theta_j\) in increasing order, then it is clear that \(\langle\theta_{j+1} - \theta_j\rangle = 2\pi/N\). Almost equally clear is that the PDF of \(\theta_{j+1} - \theta_j\) vanishes to order \(\beta\) at~\(0\). For that reason \(\beta\) is referred to as the \terminology{order of repulsion} between the points.%
\par
As \(\beta\to 0\) the repulsion disappears and the points become independently uniformly distributed on the circle.  As \(\beta\to\infty\) the points approach equal spacing (known as the \terminology{picket fence} distribution), with the only remaining randomness being the angle of an initial point which determines the location of all the others. Thus, as \(\beta\to 0\) the particles behave like a gas, and as \(\beta\to\infty\) they crystallize.  For this reason, \(\beta\) is sometimes referred to as \terminology{inverse temperature}.%
\par
Between those temperature extremes there are three values which have been studied extensively in the mathematical physics literature: \(\beta=1\), \(2\), and~\(4\). Those are more commonly known as the \(\COE\), \(\CUE\), and \(\CSE\), where the O\slash{}U\slash{}S stand respectively for Orthogonal\slash{}Unitary\slash{}Symplectic. Examples of random polynomials~\(\mathcal Z(\theta)\) from each of those ensembles are shown in \hyperref[fig_beta124]{Figure~{\xreffont\ref{fig_beta124}}}. Note that the graphs all have the same vertical scale, ranging from \(-6\) to~\(6\). The graphs show the effect of \(\beta\) on the frequency of small zero gaps, large zero gaps, and large values. Those examples have \(N=68\), that choice ensuring the polynomials have the same number of zeros in a span of \(2\pi\) as the \(Z\)-function shown in \hyperref[fig_BH68]{Figure~{\xreffont\ref{fig_BH68}}}. Our next goal is to persuade that one of those ensembles provides a good model for the \(\zeta\)-function.%
\begin{figureptx}{Figure}{Example random polynomials from \(\CbetaE(68)\) with, reading top to bottom, \(\beta=1, 2\), and \(4\).}{fig_beta124}{}%
\begin{image}{0}{1}{0}{}%
\includegraphics[width=\linewidth]{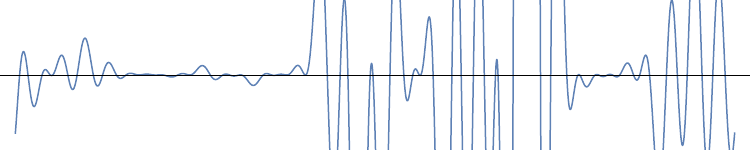}
\end{image}%
\begin{image}{0}{1}{0}{}%
\includegraphics[width=\linewidth]{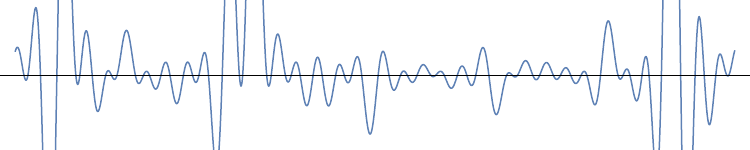}
\end{image}%
\begin{image}{0}{1}{0}{}%
\includegraphics[width=\linewidth]{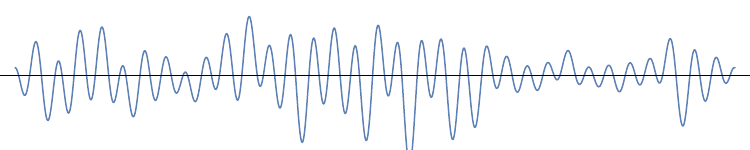}
\end{image}%
\tcblower
\end{figureptx}%
Hopefully one of the plots in \hyperref[fig_beta124]{Figure~{\xreffont\ref{fig_beta124}}} looks like it has similar randomness to the \(Z\)-function in \hyperref[fig_BH68]{Figure~{\xreffont\ref{fig_BH68}}}, because otherwise we are out of ideas. The bottom plot, \(\beta=4\), is ruled out in multiple ways. The 4th order repulsion causes the zeros to be too rigidly spaced, which in turn prevents the function from having sufficiently many large or small values.  Polynomials from the \(\CSE\) are pleasant but rather boring, lacking the pizzazz of the Riemann \(\zeta\)-function.%
\par
The top plot in \hyperref[fig_beta124]{Figure~{\xreffont\ref{fig_beta124}}} can be ruled out by carefully examining the features of the function.  The most obvious property is that the \(\COE\) example is large much more often: consider the width of the cutoffs at the large maxima. The \(\COE\) example also has more small zero gaps, both individually (small local extrema) and in clusters.%
\par
The \(\CSE\) is too rigid to serve as a model for the \(\zeta\)-function, and the \(\COE\) is too flexible. The Goldilocks zone is occupied by the \(\CUE\) which is conjectured to provide, with appropriate adjustments and caveats, a model for the zeros of the \(\zeta\)-function. We clarify the connection, and provide more terminology and historical details, in \hyperref[RMThistory]{Section~{\xreffont\ref{RMThistory}}}.%
\begin{warning}{Warning}{}{CbetaE-11}%
The \(\CUE\) has the same distribution as the eigenvalues of Haar-random matrices from the compact unitary group~\(U(N)\). But the \(\COE\) and \(\CSE\), and the \(\GOE\), \(\GUE\), and \(\GSE\) described in \hyperref[GUEera]{Subsection~{\xreffont\ref{GUEera}}}, do not correspond to Haar measure on classical matrix groups with similar names. The eigenvalues of Haar-random matrices from the unitary symplectic and unitary orthogonal groups all have quadratic repulsion for the bulk of their eigenvalues, with different behavior for the eigenvalues close to~\(1\).  These play a role in modeling the low-lying zeros of certain families of \(L\)-functions,  see~[\hyperlink{KS2}{{\xreffont 77}}, \hyperlink{Rub1}{{\xreffont 96}}].%
\end{warning}
\end{subsectionptx}
\typeout{************************************************}
\typeout{Subsection 6.4 Where does the \(\CUE\) come from?}
\typeout{************************************************}
\begin{subsectionptx}{Subsection}{Where does the \(\CUE\) come from?}{}{Where does the \(\CUE\) come from?}{}{}{rmt_connections-6}
To do calculations involving the \(\CUE\), such as computing an expectation \(\langle f \rangle\), all one needs is the joint law of the zeros: \hyperref[betameasure]{({\xreffont\ref{betameasure}})} with \(\beta=2\). That is, it does not matter how the points are generated. But to construct an example, such as the plots in \hyperref[fig_beta124]{Figure~{\xreffont\ref{fig_beta124}}}, one needs a way to produce actual sets of points with the appropriate distribution. Here are some ways:%
\begin{enumerate}
\item{}\lititle{Brownian motion.}\par%
 Have \(N\) points start at the origin and undergo non-intersecting Brownian motion until they hit the unit circle. Those points will be distributed according to~\hyperref[betameasure]{({\xreffont\ref{betameasure}})} with \(\beta=2\). (There are Brownian motion models for the \(\CbetaE\) for any \(\beta\). See \hyperlink{Dy}{[{\xreffont 44}]}.)%
\item{}\lititle{A product of tridiagonal matrices.}\par%
 Let \(\Xi_k\) be certain independent random \(2\times 2\) matrices, depending on a parameter \(\beta\), defined in~\hyperlink{KN}{[{\xreffont 79}]}. Let%
\begin{equation*}
L=\diag(\Xi_0,\Xi_2,\dots,\Xi_{\lfloor N/2 \rfloor})
\ \ \ \ \
\text{ and }
\ \ \ \ \ 
M=\diag(\Xi_{-1},\Xi_1,\dots,\Xi_{\lfloor N/2 \rfloor - 1}).
\end{equation*}
The eigenvalues of the tridiagonal matrix \(LM\), equivalently \(ML\), will be distributed according to the \(\CbetaE(N)\).%
\item{}\lititle{The unitary group \(U(N)\).}\par%
 Let \(A\in U(N)\) be chosen randomly with respect to Haar measure. The eigenvalues of \(A\) will be distributed according to~\hyperref[betameasure]{({\xreffont\ref{betameasure}})} with \(\beta=2\).%
\end{enumerate}
The final option is the easiest to implement, but see Mezzadri's paper~\hyperlink{Mez}{[{\xreffont 83}]} which describes possible pitfalls.%
\par
To summarize:%
\begin{principle}{Principle}{}{}{notreallyaboutmatrices}%
The connection between random matrices and the local statistics of zeros of the \(\zeta\)-function is not actually about random matrices:  all that really matters is the \(\beta=2\)  repulsion \hyperref[betameasure]{({\xreffont\ref{betameasure}})} on the zeros\slash{}eigenvalues. Random matrices, either Haar-random from \(U(N)\) or products of tridiagonal matrices~\hyperlink{KN}{[{\xreffont 79}]}, just happen to be convenient ways to generate such distributions.%
\end{principle}
However, for historical reasons, and for lack of a good alternate term, it is common to refer to the \(N\) points in the \(\CbetaE(N)\) as ``eigenvalues'', even if they do not arise from a matrix.%
\par
The description we have given is incomplete and our claim of a connection between the \(\CUE\) and the \(\zeta\)-function is perhaps not persuasive, so in \hyperref[RMThistory]{Section~{\xreffont\ref{RMThistory}}} we provide more details about random matrix eigenvalues, random unitary polynomials, and the connection with \(L\)-functions. Then in \hyperref[rmtwaves]{Section~{\xreffont\ref{rmtwaves}}} we return to our main theme and explore carrier waves in random polynomials from the \(\CUE\).%
\end{subsectionptx}
\end{sectionptx}
\typeout{************************************************}
\typeout{Section 7 RMT and \(L\)-functions: terminology and history}
\typeout{************************************************}
\begin{sectionptx}{Section}{RMT and \(L\)-functions: terminology and history}{}{RMT and \(L\)-functions: terminology and history}{}{}{RMThistory}
\begin{introduction}{}%
The \(\CbetaE\) were not how Random Matrix Theory (RMT) first arose in mathematical physics. RMT came about in the 1950s as a way to understand complicated quantum mechanical systems. For example: the energy levels of a large atomic nucleus. The justification is that the Hamiltonian of the system is a large and complicated matrix with certain symmetries, and so the eigenvalues of the Hamiltonian should have similar statistical properties to a large random matrix with the same symmetries.%
\par
The connection with \(L\)-functions arose in three phases.%
\begin{enumerate}
\item{}\lititle{The \(\GUE\) era.}\par%
 The original formulation in the early 1970's concerned the limiting local statistics of eigenvalues\slash{}zeros. By \emph{limiting} we refer to the rescaled eigenvalues in the large matrix limit.  The matrices are Hermitian so the eigenvalues are real. See \hyperref[GUEera]{Subsection~{\xreffont\ref{GUEera}}} for more details.%
\item{}\lititle{The classical compact group era.}\par%
 In the late 1990's it was realized that certain compact unitary groups of matrices could be associated with families of \(L\)-functions.  The size of the matrices is a natural function of the conductor of the \(L\)-functions (so one could model at finite height), the particular matrix group is determined from properties of the \(L\)-functions, and the characteristic polynomial can be used to model the values of the \(L\)-functions. See \hyperref[CUEera]{Subsection~{\xreffont\ref{CUEera}}} for more details.%
\item{}\lititle{The recipe era.}\par%
 In the early 2000's a heuristic, inspired by RMT but not actually making use of random matrices, was developed to conjecture the full main term of averages of shifted products and ratios of \(L\)-functions. Those expressions, ratios in particular, are sufficient to produce all the correlation functions of the zeros, and so can reproduce virtually every conjecture arising from RMT. See~\hyperlink{CS}{[{\xreffont 38}]}. Furthermore, the heuristic produces the complete main term and not just the leading order behavior, and there is no need to insert an ``arithmetic factor'' in an ad-hoc way.  The heuristic goes by the name ``the recipe''.  See the end of \hyperref[recipeera]{Subsection~{\xreffont\ref{recipeera}}} for a brief mention.%
\end{enumerate}
\par
Despite the fact that the recipe has largely replaced RMT as a tool for making conjectures about \(L\)-functions, random matrices are essential for our discussion of carrier waves, because characteristic polynomials provide concrete examples which we can visualize.%
\end{introduction}%
\typeout{************************************************}
\typeout{Subsection 7.1 The \(\GUE\) era}
\typeout{************************************************}
\begin{subsectionptx}{Subsection}{The \(\GUE\) era}{}{The \(\GUE\) era}{}{}{GUEera}
The matrices in this setting are Hermitian, so the eigenvalues are real. The random entries are real, complex, or quaternionic depending on the type of physical system being modeled, and the randomness is Gaussian. As in the circular ensembles, the three Gaussian ensembles are named \(\GOE\), \(\GUE\), and \(\GSE\), where again O\slash{}U\slash{}S stands for Orthogonal\slash{}Unitary\slash{}Symplectic.  Again there is a single parameter \(\beta\) which describes the eigenvalue repulsion, with \(\beta=1\), \(2\), or \(4\) for each of those cases, respectively. And again the way the eigenvalues (or energy levels) are generated could be Brownian motion~\hyperlink{Dy}{[{\xreffont 44}]}, or tridiagonal matrices~\hyperlink{DE}{[{\xreffont 45}]}, or the original formulation of random Hermitian matrices with a given symmetry.  The only concern for our purposes is that one has \(N\) points on the real line with joint probability density function:%
\begin{equation}
\frac{\Gamma(1+\beta/2)^N}{(2\pi)^{N/2}\prod_{j=1}^N \Gamma(1+j\beta/2)}
\exp\bigl(-\tfrac{1}{2} \sum_{j=1}^N \lambda_j^2\bigr)
\prod_{1\le j \lt k \le N} \abs{\lambda_j - \lambda_k}^\beta\text{.}\label{GbetaEmeasure}
\end{equation}
\par
A similarity with the \(\CbetaE\) is the order \(\beta\) repulsion between points.  A difference is that the points are on a line, with the factor \(\exp(-\frac12\lambda_j^2)\) keeping more of the points near the origin and preventing the points from wandering too far away.  Most of the points satisfy \(\abs{\lambda_j} \lt \sqrt{2\beta N}\) and the rescaled points \(\lambda_j/ \sqrt{\beta N}\) are distributed according to the \terminology{Wigner semicircle law} \(\pi^{-1}\sqrt{2 - x^2}\).%
\par
The similarities outweigh the differences.  The \(\GbetaE\) and \(\CbetaE\) belong to the same \terminology{universality class} which is characterized only by order \(\beta\) repulsion between points. Let's be specific about which properties are universal. Define the rescaled eigenvalues \(\tilde{\lambda}\) so that \(\langle \tilde{\lambda}_{j+k} - \tilde{\lambda}_{j}\rangle = k\) for all fixed~\(k\). Consider any \terminology{local statistic}, that is, a statistic only involving a finite number of differences of rescaled eigenvalues.  Examples of local statistics are the normalized nearest neighbor spacing%
\begin{equation}
p_2(x) = \lim_{\Delta x \to 0}
\frac{1}{\Delta x}
\left\langle \frac{\# \{\tilde{\lambda}_{j+1} - \tilde{\lambda}_{j} \in [x, x + \Delta x] \}}{N}
\right\rangle\text{,}\label{nndef}
\end{equation}
the pair correlation function%
\begin{equation}
R_2(x) = \lim_{\Delta x \to 0}
\frac{1}{\Delta x}
\left\langle \frac{\#\{\tilde{\lambda}' - \tilde{\lambda} \in [x, x + \Delta x]\}}{N}
\right\rangle\text{,}\label{paircorrdef}
\end{equation}
or more generally the joint distribution of \((\tilde{\lambda}_{j_1} - \tilde{\lambda}_{j_0}, \dots, \tilde{\lambda}_{j_k} - \tilde{\lambda}_{j_0})\) for any \(j_0\) and \(j_1 \lt \dots \lt j_k\). Note that the focus is on the normalized distances between an eigenvalue and its neighbors, not on their absolute locations.%
\par
The quantities above depend on \(N\), but we suppress that notation because the leading order behavior is independent of~\(N\).%
\begin{principle}{Principle}{Universality.}{}{betauniversality}%
For each \(\beta\), all \(\beta\)-ensembles have the same limiting local eigenvalue statistics as \(N\to\infty\).%
\end{principle}
\hyperref[betauniversality]{Principle~{\xreffont\ref{betauniversality}}} comes from the fact that the \(\beta\)-repulsion term is the most important feature of the measure. Here is a less hand-wavy explanation in the case of the \(\CUE(N)\) and the \(\GUE(N)\). Both ensembles are examples of a \terminology{determinental point process}. This means there is a \terminology{kernel function} \(K_N(x,y)\) such that the \(n\)-correlation function of the eigenvalues can be expressed as the determinant of an \(n\times n\) matrix with entries involving \(K_N(x,y)\). For example, the pair correlation function is given by%
\begin{equation}
R_{2,N}(x-y) = \begin{vmatrix} 1 \amp K_N(x,y) \\
K_N(y,x) \amp 1
\end{vmatrix}\text{.}\label{parcorrdet}
\end{equation}
Thus, the kernel function determines all the correlation functions, so it also determines all the local statistics (because those can be expressed in terms of the correlation functions, possibly involving a complicated inclusion-exclusion argument).%
\par
For \(\GUE(N)\) the kernel function is%
\begin{equation}
K^{\GUE(N)}(x, y) = \sqrt{N}\frac{\psi_N(x) \psi_{N-1}(y)
-
\psi_{N-1}(x) \psi_{N}(y) }
{x-y}\text{,}\label{GUEN}
\end{equation}
where \(x, y\in \R\). Here \(\psi_n(x) = He_n(x) e^{-x^2/4}/\sqrt{\sqrt{2\pi} n!}\), where \(He_n(x)\) is a Hermite polynomial. Change variables \((x,y) \mapsto (4x/\sqrt{N},4y/\sqrt{N})\) for the normalized neighbor spacing.%
\par
For \(\CUE(N)\) the kernel function is%
\begin{equation}
K^{\CUE(N)}(\theta, \phi) = \frac{1}{N}\frac{\sin(N (\theta-\phi)/2)}{\sin( (\theta-\phi)/2)}\text{,}\label{GUEera-9-2}
\end{equation}
where \(\theta, \phi \in \R \mod 2\pi\). Change variables \((\theta,\phi) \mapsto (2\pi\theta/N, 2\pi\phi/N)\) for the normalized neighbor spacing. To leading order those kernels are the same, because by properties of Hermite polynomials, both (normalized) large \(N\) limits are \(\sin(\pi x)/(\pi x) \).%
\par
The difference between the kernel functions of the \(\GUE(N)\) and \(\CUE(N)\) is \(O(1/N)\), so the same holds for all correlation functions. \hyperref[guecuedifference]{Figure~{\xreffont\ref{guecuedifference}}} shows the pair correlation for each ensemble for \(N=50\), and their difference (with the \(\GUE(N)\) pair correlation being larger in a neighborhood of \(0\)).%
\begin{figureptx}{Figure}{Pair correlation for \(\GUE(50)\) and \(\CUE(50)\), and their difference.}{guecuedifference}{}%
\begin{image}{0.05}{0.9}{0.05}{}%
\includegraphics[width=\linewidth]{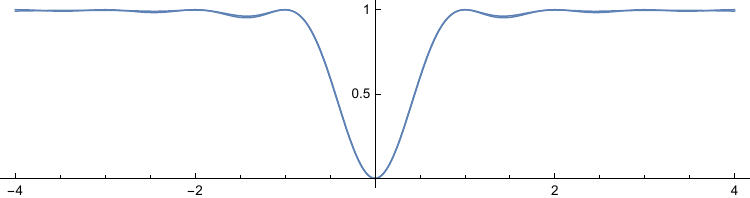}
\end{image}%
\begin{image}{0.05}{0.9}{0.05}{}%
\includegraphics[width=\linewidth]{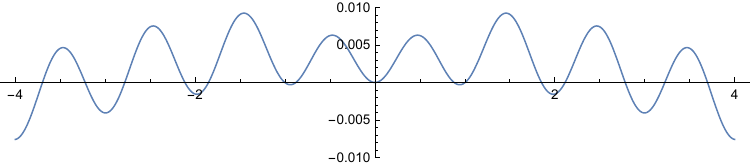}
\end{image}%
\tcblower
\end{figureptx}%
Note that by \hyperref[parcorrdet]{({\xreffont\ref{parcorrdet}})} and the limiting values of the kernel functions for the \(\GUE\) and \(\CUE\), we find that both have the limiting normalized pair correlation function:%
\begin{equation}
R_2(x) = 1 - \frac{\sin^2(\pi x)}{(\pi x)^2}\text{.}\label{cuepaircorrelation}
\end{equation}
\par
That pair correlation function was the key to discovering the connection between random matrix eigenvalues and the zeros of the \(\zeta\)-function.%
\end{subsectionptx}
\typeout{************************************************}
\typeout{Subsection 7.2 Zero statistics for \(L\)-functions}
\typeout{************************************************}
\begin{subsectionptx}{Subsection}{Zero statistics for \(L\)-functions}{}{Zero statistics for \(L\)-functions}{}{}{zero_stats}
The RMT revolution in number theory began when Montgomery~\hyperlink{Mon}{[{\xreffont 84}]} determined partial information about the pair correlation of the zeros of the \(\zeta\)-function.  The theorem he proved, combined with some heuristics about the prime numbers, led him to conjecture that asymptotically the pair correlation function was%
\begin{equation}
R_{2,\zeta}(x) = 1 - \frac{\sin^2(\pi x)}{(\pi x)^2}\text{.}\label{zetapc}
\end{equation}
Here%
\begin{equation}
R_{2,\zeta}(x) :=
\lim_{\Delta x \to 0}
\frac{1}{\Delta x}
\lim_{T\to\infty}
\frac{1}{N(T)}
\sum_{0 \lt \gamma,\gamma' \leq T} \#\{\tilde{\gamma}' - \tilde{\gamma} \in [x, x+\Delta x] \}\label{zero_stats-2-5}
\end{equation}
is the \(\zeta\)-function analogue of~\hyperref[paircorrdef]{({\xreffont\ref{paircorrdef}})}. Montgomery's conjecture became more significant when he met Freeman Dyson, who informed him that \hyperref[zetapc]{({\xreffont\ref{zetapc}})} was the limiting pair correlation of eigenvalues of the \(\GUE\), as shown in~\hyperref[cuepaircorrelation]{({\xreffont\ref{cuepaircorrelation}})}. Combined with extensive computations by Odlyzko~\hyperlink{Odl}{[{\xreffont 88}]}, this established:%
\begin{principle}{Principle}{The \(\GUE\) Hypothesis, aka the Montgomery-Odlyzko law.}{}{guehypothesis}%
The limiting local statistics of the zeros of the \(\zeta\)-function are the same as the limiting local statistics of the \(\GUE\).%
\end{principle}
The ``limiting statistics'' above refer to \(T\to\infty\) and \(N\to\infty\).%
\begin{paragraphs}{Consequences of the \(\GUE\) Hypothesis.}{zero_stats-5}%
The \(\GUE\) Hypothesis has many useful consequences:%
\begin{enumerate}
\item{}Montgomery's pair correlation conjecture~\hyperref[zetapc]{({\xreffont\ref{zetapc}})}.%
\item{}100\% of the zeros of the \(\zeta\)-function are simple.%
\item{}The PDF \(p_2(x)\), see \hyperref[nndef]{({\xreffont\ref{nndef}})}, of the nearest neighbor spacing \(\tilde{\gamma}_{j+1} - \tilde{\gamma}_j\) is given by an explicit expression, which is well-approximated  in the bulk by the \terminology{Wigner surmise}%
\begin{equation}
\tfrac{32}{\pi^2} x^2 e^{-\frac{4}{\pi} x^2}\text{,}\label{zero_stats-5-3-3-1-5}
\end{equation}
which is the exact expression for \(\GUE(2)\). The limiting behavior of \(p_2(x)\) for \(\GUE(n)\) as \(n\to\infty\) is \(\pi^2 x^2/3\) as \(x\to 0\) and \(\exp(-\pi^2 x^2/8)\) as \(x\to \infty\).%
\par
The next two items are consequences of those limiting behaviors.%
\item{}The smallest gaps between consecutive zeros satisfy%
\begin{equation}
\tilde{\gamma}_{j+1} - \tilde{\gamma}_j \sim  \gamma_j^{-\frac13 + o(1)}\text{,}\label{mingap}
\end{equation}
 that is,%
\begin{equation}
\liminf_{j\to\infty} \frac{\log(\gamma_{j+1} - \gamma_j)}{\log \gamma_j} = -\frac13\text{.}\label{zero_stats-5-3-4-2}
\end{equation}
It is possible to make a more precise statement. See \hyperref[extremegaps]{Subsection~{\xreffont\ref{extremegaps}}}.%
\item{}The largest gaps satisfy%
\begin{equation}
\tilde{\gamma}_{j+1} - \tilde{\gamma}_j \sim
c_{large} \sqrt{\log \gamma_j}\text{,}\label{maxgap}
\end{equation}
where perhaps \(c_{large} = 1/\sqrt{32}\)~\hyperlink{BA_B}{[{\xreffont 12}]}. Or perhaps \(c_{large}\) is a bit smaller.  See \hyperref[extremegaps]{Subsection~{\xreffont\ref{extremegaps}}}.%
\end{enumerate}
Information about small gaps between zeros has implications for the class number problem~\hyperlink{CI}{[{\xreffont 37}]}, which was Montgomery's original motivation for studying the pair correlation function~\hyperlink{Mon}{[{\xreffont 84}]}.%
\end{paragraphs}%
\begin{paragraphs}{Shortcomings of the \(\GUE\) Hypothesis.}{zero_stats-6}%
The predictions of the \(\GUE\) Hypothesis are spectacular, but it was clear from the beginning that only part of the story was being revealed. Odlyzko's computations showed a discrepancy with the predictions.  The discrepancy decreased at larger heights, as would be expected since the \(\GUE\) Hypothesis only refers to the limiting behavior, but clearly there was a need for a more precise model.%
\par
Another shortcoming is that the prime numbers are nowhere to be seen in the \(\GUE\) Hypothesis.  How can it be that the truth about the \(\zeta\)-function has nothing to do with the prime numbers?%
\par
Finally, and most relevant to this present work, the \(\GUE\) Hypothesis says little or nothing about the actual values of the \(\zeta\) function. This appeared in Odlyzko's calculations (which computed the values of the \(\zeta\)-function and not just the zeros), which found a large discrepancy between the computed value distribution and the limiting Gaussian distribution from Selberg's CLT. The convergence to Gaussian is slow, see \hyperref[fig_nongaussian]{Figure~{\xreffont\ref{fig_nongaussian}}},  so it was not surprising to see a discrepancy.  But the nature of the discrepancy is not explained by the \(\GUE\) Hypothesis. It took until the year 2000 to address that shortcoming.%
\end{paragraphs}%
\end{subsectionptx}
\typeout{************************************************}
\typeout{Subsection 7.3 After 30 years of \(\GUE\): classical compact groups and the Keating-Snaith Law}
\typeout{************************************************}
\begin{subsectionptx}{Subsection}{After 30 years of \(\GUE\): classical compact groups and the Keating-Snaith Law}{}{After 30 years of \(\GUE\): classical compact groups and the Keating-Snaith Law}{}{}{CUEera}
In the late 1990's two innovations increased the influence of RMT on number theory.  The first was due to Katz and Sarnak~\hyperlink{KaSa}{[{\xreffont 75}]} who found that naturally arising collections of \(L\)-functions have a \terminology{symmetry type} which is given by a classical compact matrix group.  The possible symmetry types are Unitary, Symplectic, and Orthogonal, with Orthogonal splitting into two cases depending on whether the sign of the functional equation is always \(+1\) or equally likely \(+1\) or \(-1\). The random matrices in this context are elements of the compact unitary groups \(U(N)\), \(Sp(N)\), and \(O(N)\), where the randomness is uniform with respect to Haar measure.  Note that \(U(N)\) with Haar measure is the same as the \(\CUE(N)\), but the others are completely different than their mathematical physics counterparts having similar names.  In fact, all of those classical compact groups have the same bulk eigenvalue statistics: the only difference arises with the eigenvalues near~\(1\).  Numerical evidence~\hyperlink{Rub1}{[{\xreffont 96}]}, and much subsequent work, supports the conjecture that the symmetry type of the family of \(L\)-functions determines the distribution of zeros near the critical point.%
\par
Keating and Snaith~\hyperlink{KS1}{[{\xreffont 76}]}\hyperlink{KS2}{[{\xreffont 77}]} made the second key leap when they recognized that the characteristic polynomial could be used to model the \(L\)-function itself. Here the \terminology{characteristic polynomial} of the \(N\times N\) matrix \(A\) is written in the slightly nonstandard form%
\begin{align}
\Lambda_A(z) =\mathstrut& \det( I - z A)\label{CUEera-3-8-1}\\
=\mathstrut& \prod_{1\le j \le N}(1 - z e^{i\theta_j})\text{,}\notag
\end{align}
which is in keeping with our normalization \hyperref[srpoly]{({\xreffont\ref{srpoly}})} for self-reciprocal polynomials. Here the \(e^{i \theta_j}\) are the eigenvalues of \(A\). The analogue of the \(Z\)-function is%
\begin{align}
\mathcal Z_A(\theta) =\mathstrut & ((-1)^N \det A)^{-\frac12} e^{-iN\theta/2} \det(I - e^{i\theta} A)\notag\\
=\mathstrut & ((-1)^N \det A\,  e^{i N\theta})^{-\frac12} \Lambda_A(e^{i\theta})\text{,}\label{RMTZ}
\end{align}
which is real for \(\theta \in \R\). We write \(\mathcal S_A\) for the error term in the zero counting function:%
\begin{equation*}
\#\{\theta_j \in [0, X] \} = \frac{1}{2\pi} X + C_A + \mathcal S_A(X)
\end{equation*}
where \(C_A\) is chosen so that \(\mathcal S_A\) is \(0\) on average.  The number \(C_A\) is analogous to the constant \(7/8\) in the zero counting function of the \(\zeta\)-function, see \hyperref[NT]{({\xreffont\ref{NT}})}.%
\par
We confine our discussion to the case of the \(\zeta\)-function, which constitutes a unitary family, meaning that \(\zeta(\frac12 + i t)\) is modeled by the characteristic polynomial of Haar-random matrices from \(U(N)\).%
\begin{principle}{Principle}{The Keating-Snaith Law.}{}{KSlaw}%
Model \(\zeta(\frac12 + i t)\) for \(t\approx T\) by \(\Lambda_A(e^{i\theta})\) for Haar-random \(A\in U(N)\), where \(N = \log(T/2\pi)\).%
\end{principle}
A key component of \hyperref[KSlaw]{Principle~{\xreffont\ref{KSlaw}}} is the identification \(N=\log(T/2\pi)\), where the equality means ``is an integer close to''. This is justified by ``equating the density of zeros''. When calculating asymptotics, where the only realistic hope is to make leading-order predictions, one can take \(N=\log T\).  But for numerical comparison involving small numbers, the more precise choice of \(N\) is helpful. See \hyperref[kappa1]{Subsection~{\xreffont\ref{kappa1}}} for further discussion of \(N=\log T\).%
\par
The first great success of \hyperref[KSlaw]{Principle~{\xreffont\ref{KSlaw}}} was explaining the apparent discrepancy between numerical data and Selberg's CLT. Keating and Snaith proved a central limit theorem for \(\log\Lambda_A(e^{i\theta})\), exactly analogous to Selberg's result. Odlyzko's data, taken at height \(1.52 \times 10^{19}\), does not look Gaussian, and is not even symmetric. But analogous data for \(\Lambda_A(e^{i\theta})\) for Haar-random \(A\in U(42)\) gives a close fit with the \(\zeta\)-function data. See Figure~1 in~\hyperlink{KS1}{[{\xreffont 76}]}. The ``42'' comes from \(\log(1.52 \times 10^{19}/(2 \pi))\approx 42.33\). Analogous to \hyperref[prin_carrier]{Principle~{\xreffont\ref{prin_carrier}}}, it is the carrier wave which obeys the Keating-Snaith central limit theorem.%
\begin{figureptx}{Figure}{On the left is the PDF of \(\log\abs{\Lambda_A(e^{i\theta})}\) for random \(A\in U(N)\) for \(N=25, 50, 100, 200, 400\), and \(800\), each normalized to have unit variance.  On the right is the negative logarithm of those PDFs.  In both cases the dashed line is the limiting Gaussian distribution.%
}{fig_nongaussian}{}%
\begin{sidebyside}{2}{0}{0.02}{0}%
\begin{sbspanel}{0.50}[bottom]%
\noindent\includegraphics[width=\linewidth]{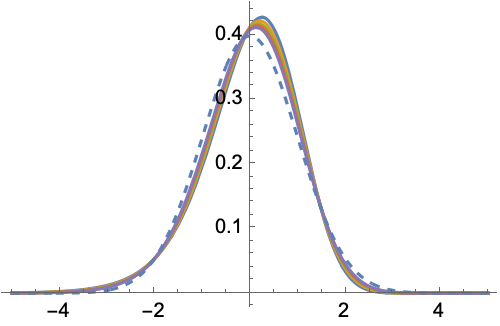}
\end{sbspanel}%
\begin{sbspanel}{0.48}[bottom]%
\noindent\includegraphics[width=\linewidth]{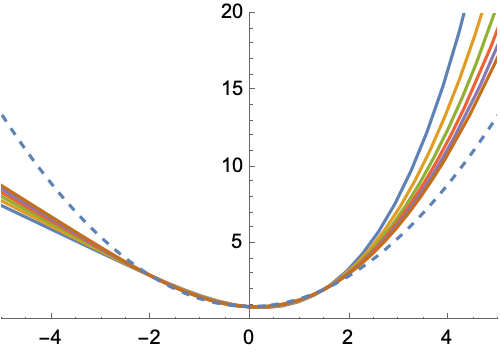}
\end{sbspanel}%
\end{sidebyside}%
\tcblower
\end{figureptx}%
The second great success came from predicting the mysterious factor \(g_k\) in the conjectured moments \hyperref[leadingtwokthmoment]{({\xreffont\ref{leadingtwokthmoment}})} of the \(\zeta\)-function. Previously it was known that \(g_1=1\) and \(g_2=2\). It had been conjectured [\hyperlink{CG2}{{\xreffont 34}}, \hyperlink{CoGo}{{\xreffont 36}}] that \(g_3=42\) and \(g_4=24024\). What Keating and Snaith found  was (they compute an exact expression, but only the asymptotic is relevant for this discussion):%
\begin{equation}
\left\langle |\Lambda_A(e^{i\theta})|^{2k}\right\rangle_{U(N)}
\sim g_{U,k} N^{k^2}\text{,}\label{CUE2k}
\end{equation}
as \(N\to\infty\), where%
\begin{equation}
g_{U,k} = k^2! \prod_{j=0}^{k-1} \frac{j!}{(j+k)!}\text{.}\label{CUEera-9-11}
\end{equation}
They found the beautiful equality \(g_{U,k} = g_k\) for the known and conjectured values for \(k=1,2,3,4\).%
\par
We make explicit some of the assumptions in \hyperref[KSlaw]{Principle~{\xreffont\ref{KSlaw}}}.%
\begin{principle}{Principle}{The Keating-Snaith Law, part 2.}{}{eNprin}%
To model \(\zeta(\frac12 + it)\) throughout the interval \([T, 2T]\), choose \(e^N\) Haar-random characteristic polynomials from \(U(N)\), where \(N\approx \log T/2\pi\).%
\end{principle}
In other words, the \(\zeta\)-function is treated as being independent on each separate interval of length~\(2\pi\).  One could argue that \(e^N/2\pi\) random matrices is a better choice, but that makes no difference in practice.  \hyperref[eNprin]{Principle~{\xreffont\ref{eNprin}}} is one of the many ways to obtain the conjectured maximum values \hyperref[FGHconjecture]{({\xreffont\ref{FGHconjecture}})} and \hyperref[FGHconjectureS]{({\xreffont\ref{FGHconjectureS}})} of \(Z(t)\) and \(S(t)\)~\hyperlink{FGH}{[{\xreffont 49}]}.%
\par
The success of the \hyperref[KSlaw]{Keating-Snaith Law} is our justification to use characteristic polynomials from \(U(N)\), equivalently \(\CUE(N)\), to explore the behavior of \(Z(t)\) beyond the realm accessible by direct computation.  We address some natural questions concerning the connection between RMT and the \(\zeta\)-function before examining carrier waves in characteristic polynomials in~\hyperref[rmtwaves]{Section~{\xreffont\ref{rmtwaves}}}.%
\begin{paragraphs}{Limitations of the Keating-Snaith law.}{CUEera-14}%
The GUE Hypothesis is a statement about limiting behavior, and the Keating-Snaith law can be interpreted as a first order correction. Here we discuss the prospects of making a more precise version of the Keating-Snaith law.%
\par
It is obvious that random matrices cannot capture the subtle behavior of the \(\zeta\)-function because the matrices do not know anything about the prime numbers. For example, when characteristic polynomials were first used to conjecture moments of the \(\zeta\)-function, see \hyperref[twokthmoment]{({\xreffont\ref{twokthmoment}})}, the arithmetic factor \(a_k\) was inserted in an ad-hoc manner.  A way to make the moment conjecture more natural is the ``hybrid model'' \hyperlink{hybrid}{[{\xreffont 59}]}, which expresses the \(\zeta\)-function as a product with two components: one involving the zeros and one involving the primes. See \hyperref[sec_hybrid]{Subsection~{\xreffont\ref{sec_hybrid}}}. Assuming a statistical independence of those two components, the leading term in conjectured moments arises as the product \(g_k a_k\).%
\par
It is natural to wonder if a more precise hybrid model could produce lower order terms in the conjectured moments? Since the zero statistics of the \(\zeta\)-function and of characteristic polynomials agree only to leading order, a more precise model would either have to produce eigenvalue statistics which incorporate arithmetic effects, or the arithmetic factor would somehow have to compensate for the incorrect lower order terms in the factor involving the zeros.  Neither option seems plausible. The difficulty is further compounded by the fact that the two pieces cannot be treated as independent \textemdash{} because by inspection the main terms in the conjectured moments do not factor as a product. We say a bit more about this in \hyperref[sec_hybrid]{Subsection~{\xreffont\ref{sec_hybrid}}}.%
\par
A more detailed hybrid model incorporating the interactions between the primes and the zeros would be interesting, but it seems unlikely that such a construction exists. And it is largely unnecessary because the recipe already provides precise information about the moments and the zeros of the \(\zeta\)-function. In particular, we know a lot about the lower order corrections to the Keating-Snaith law.  In every case we find that the lower order corrections have a dampening effect.  For example, the arithmetic factor in the conjectured moments, \(a_k\), is very small. Another example is that the spacing of zeros of the \(\zeta\)-function is more rigid than the eigenvalues of random matrices, see \hyperref[spectralrigidity]{Subsection~{\xreffont\ref{spectralrigidity}}}.  This has the effect of making extreme gaps, or dense clusters of zeros, slightly less likely.  To summarize:%
\begin{principle}{Principle}{}{}{KSoverestimate}%
Lower-order corrections to the Keating-Snaith law have a dampening effect: the statistical behavior of the \(\zeta\)-function and its zeros tend to be slightly less extreme than the analogous quantities for random unitary matrices.%
\end{principle}
\end{paragraphs}%
\end{subsectionptx}
\typeout{************************************************}
\typeout{Subsection 7.4 Modeling \(Z(t)\) vs. modeling \(\zeta(s)\)}
\typeout{************************************************}
\begin{subsectionptx}{Subsection}{Modeling \(Z(t)\) vs. modeling \(\zeta(s)\)}{}{Modeling \(Z(t)\) vs. modeling \(\zeta(s)\)}{}{}{RMThistory-6}
Our original motivation was to model \(Z(t)\), but our description of the Keating-Snaith law referred to modeling \(\zeta(s)\).  To clarify:%
\begin{principle}{Principle}{}{}{prin_Zandzeta}%
Model \(Z(t)\) for \(t\in \mathbb R\) by \(\mathcal Z_A(\theta)\). Model \(\zeta(\frac12 + a + i t)\) for \(a\in \mathbb C\) by \(\Lambda(e^{i\theta - a})\).%
\end{principle}
Some consequences of this model arise from the symmetries in the \(\CbetaE\) measure \hyperref[betameasure]{({\xreffont\ref{betameasure}})}:%
\begin{lemma}{Lemma}{}{}{cbetaesymmetries}%
The \(\CbetaE(N)\) is invariant under the operations%
\begin{gather}
\{ \theta_j \} \leftrightarrow \{ \theta_j + \theta \}\label{rotationinvariance}\\
\{ \theta_j \} \leftrightarrow \{ - \theta_j \}\text{.}\label{conjugationinvariance}
\end{gather}
\end{lemma}
Note that the \(\GbetaE\), \hyperref[GbetaEmeasure]{({\xreffont\ref{GbetaEmeasure}})}, has a symmetry analogous to \hyperref[conjugationinvariance]{({\xreffont\ref{conjugationinvariance}})}, but no translation invariance as in \hyperref[rotationinvariance]{({\xreffont\ref{rotationinvariance}})}.%
\par
By \hyperref[rotationinvariance]{({\xreffont\ref{rotationinvariance}})}, the expected value of expressions like \(\mathcal Z_A(\theta) \mathcal Z_A(\theta + \alpha)\) or \(\abs{\Lambda_A(e^{i\theta})}^k\) are independent of \(\theta\), so one typically sets \(\theta=0\). See for example \hyperref[farmerconjecture]{({\xreffont\ref{farmerconjecture}})} and \hyperref[farmeranalogue]{({\xreffont\ref{farmeranalogue}})}. Also, since \(\mathcal Z_A(\theta + 2\pi/N) = - \mathcal Z_{A'}(\theta)\) with \(A\) and \(A'\) equally likely, the distribution of \(\mathcal Z_A(\theta)\) is symmetric (but the distribution of \(\Lambda_A(\theta)\) is not: it has average value~\(1\).) Thus, \hyperref[hardyZpm]{Principle~{\xreffont\ref{hardyZpm}}} is a theorem for characteristic polynomials.%
\par
By \hyperref[conjugationinvariance]{({\xreffont\ref{conjugationinvariance}})}, any (finite) sequence of gaps between eigenvalues is as likely as that sequence in reverse. Another way to say it is that \(\mathcal Z_{A}(-\theta) = \mathcal Z_{A'}(\theta)\) with \(A\) and \(A'\) equally likely. This explains some conjectures from~\hyperlink{shanker}{[{\xreffont 105}]}, see \hyperref[grampointsrevisited]{Subsection~{\xreffont\ref{grampointsrevisited}}}.%
\par
Modeling \(Z(t)\) by \(\mathcal Z_A(\theta)\) should seem natural:  both are real-valued functions on \(\mathbb R\), and (with \(N \approx \log t/2\pi\)) have the same average spacing between their zeros.%
\par
Modeling \(\zeta(s)\) by \(\Lambda_A(z)\) is slightly more subtle:  one is a function which (conjecturally) has its zeros on the line \(\sigma = \frac12\), and the other has its zeros on the circle \(\abs{z} = 1\). How to translate between those worlds? The answer is to identify the half-plane \(\sigma \gt \frac 12\) with the interior of the unit disc \(\abs{z} \lt 1\). That choice makes sense for several reasons; two of the author's favorites are: \(\lim_{s\to +\infty} \zeta(s) = 1\) and \(\lim_{z\to 0} \Lambda_A(z) = 1\); and all zeros of \(\zeta(s)\) being simple and lying on \(\sigma = \frac12\) implies~\hyperlink{speiser}{[{\xreffont 107}]} all zeros of \(\zeta'(s)\) lie in \(\sigma \gt \frac12\), while all zeros of \(\Lambda_A(z)\) being simple and lying on \(\abs{z}=1\) implies (by Gauss-Lucas) all zeros of \(\Lambda_A'(z)\) lie in \(\abs{z} \lt 1\).%
\par
The most natural mapping from \(\sigma \gt \frac 12\) to \(\abs{z} \lt 1\) is \(s \mapsto e^{\frac12 - s}\), as in \hyperref[prin_Zandzeta]{Principle~{\xreffont\ref{prin_Zandzeta}}}.%
\end{subsectionptx}
\typeout{************************************************}
\typeout{Subsection 7.5 Analogies, and lack thereof}
\typeout{************************************************}
\begin{subsectionptx}{Subsection}{Analogies, and lack thereof}{}{Analogies, and lack thereof}{}{}{recipeera}
We have seen that, conjecturally, to leading order the zeros of the \(\zeta\)-function have the same limiting local statistics as the \(\GUE\) or \(\CUE\).  When modeling the \(\zeta\)-function at height~\(T\) we use Haar-random characteristic polynomials from \(U(N)\) under the matching \(N \leftrightarrow \log T/2\pi\). One of three things can happen.%
\begin{paragraphs}{Case 1. Agreement to leading order.}{recipeera-3}%
For example, the author conjectured~\hyperlink{Fratios}{[{\xreffont 46}]}%
\begin{equation}
\frac{1}{T}\int_0^T \frac{\zeta(\frac12 + a + it)\zeta(\frac12 + b - it)}
{\zeta(\frac12 + c + it)\zeta(\frac12 + d - it)} \, dt
\sim
\frac{(a+d)(b+c)}{(a + b)(c+d)} - T^{-a-b} \frac{(c-a)(d-b)}{(a + b)(c+d)}\text{,}\label{farmerconjecture}
\end{equation}
for \(a, b, c, d \ll 1/\log T\), as \(T\to\infty\).  For characteristic polynomials we have the theorem~\hyperlink{HPZ}{[{\xreffont 68}]}\hyperlink{CFZ}{[{\xreffont 32}]}%
\begin{equation}
\left\langle \frac{\Lambda(e^{-a})\overline{\Lambda(e^{-b})}}
{\Lambda(e^{-c})\overline{\Lambda(e^{-d})}}\right\rangle_{U(N)}
\sim
\frac{(a+d)(b+c)}{(a + b)(c+d)} - e^{-N(a+b)} \frac{(c-a)(d-b)}{(a + b)(c+d)}\text{,}\label{farmeranalogue}
\end{equation}
as \(N\to\infty\). Those expressions fit the analogy in \hyperref[prin_Zandzeta]{Principle~{\xreffont\ref{prin_Zandzeta}}} and are identical under \(N \leftrightarrow \log T\). Both expressions can be made more precise; they have different lower order terms, with the primes playing an important role for the \(\zeta\)-function.  (See~\hyperlink{CS}{[{\xreffont 38}]}, Section~4.)%
\par
In \hyperref[waves]{Section~{\xreffont\ref{waves}}} we described the result of Bombieri and Hejhal~\hyperlink{BomHej}{[{\xreffont 24}]} that if a set of Z-functions satisfy RH, then a linear combination has 100\% of its zeros on the critical line.  The same holds for characteristic polynomials~\hyperlink{BCNN}{[{\xreffont 11}]}: if \(a_1,\dots,a_n\) are nonzero real numbers and \(A_1,\dots,A_n\) are chosen independently and Haar-random from \(U(N)\), then the expected number of zeros of \(\sum a_n \mathcal Z_{A_n}\) on the unit circle is \(N-o(N)\).%
\end{paragraphs}%
\begin{paragraphs}{Case 2. Agreement to leading order, after inserting an arithmetic factor.}{recipeera-4}%
This situation appears in the moments of \(L\)-functions, as illustrated in \hyperref[leadingtwokthmoment]{({\xreffont\ref{leadingtwokthmoment}})} and \hyperref[CUE2k]{({\xreffont\ref{CUE2k}})}%
\end{paragraphs}%
\begin{paragraphs}{Case 3. Failure of the analogy.}{recipeera-5}%
For \(L\)-functions, RH implies the Lindelöf  Hypothesis (LH), see \hyperref[conjLH]{Conjecture~{\xreffont\ref{conjLH}}}. But for unitary polynomials the analogue of RH is true but the analogue of LH can fail.  For example, the unitary polynomial \((z+1)^N\) achieves the value \(2^N\) on the unit circle, which is not \(\ll e^{\varepsilon N}\) for all \(\varepsilon > 0\).%
\par
A less extreme example is a degree \(N\) polynomial with all zeros in half of the unit circle: it takes values as large as \(\sqrt{2}^N\). Somewhat less obviously:  LH fails for a degree \(N\) unitary polynomial which has no zeros in an arc of a fixed length. The relationship between large zero gaps and the size of the function is central to the main theme of this paper, and is the topic of \hyperref[extreme]{Section~{\xreffont\ref{extreme}}}. A related situation is key to one of our refutations, see \hyperref[blanc]{Subsection~{\xreffont\ref{blanc}}}.%
\end{paragraphs}%
\begin{paragraphs}{How to have confidence in the predictions?}{recipeera-6}%
Given the range of possibilities above, how can one to give credence to any specific prediction?%
\par
The 3rd case, complete failure, may seem alarming but in fact it is of little concern.  The goal is to model statistically. By \hyperref[KSlaw]{Principle~{\xreffont\ref{KSlaw}}}, at height \(T\) the \(\zeta\)-function is modeled by Haar-random characteristic polynomials from~\(U(N)\) with \(N=\log(T/2\pi)\). More specifically, by \hyperref[eNprin]{Principle~{\xreffont\ref{eNprin}}}, \(e^N\) random matrices are chosen from~\(U(N)\). That may seem like a large number, but one would never expect to encounter a polynomial with all its zeros in half of the circle. Examining the \(\CUE(N)\) measure, \hyperref[betameasure]{({\xreffont\ref{betameasure}})} with \(\beta=2\), we see that a polynomial with all zeros in half of the circle would have many of the terms in that product be less than half the average.  Therefore such a polynomial is \(e^{-c N^2}\) times less likely than typical, so it has little chance of appearing in a sample of size~\(e^N\) if \(N\) is large. The same argument applies in the case of polynomials with no zeros in an arc of fixed width.%
\par
Thus, combining both parts of the Keating-Snaith law finds that the Lindelöf Hypothesis is predicted by random matrix theory.%
\par
Cases 1 and 2 are actually the same: the arithmetic factor in the first case happens to be~\(1\). Those arithmetic factors are well-understood, having arisen long before RMT entered the picture.  In any specific case one can identify the arithmetic factor.%
\par
Note that this discussion is mostly moot, because tools like \terminology{the recipe}~\hyperlink{CFKRS}{[{\xreffont 31}]} and \terminology{the ratios conjecture}~\hyperlink{CFZ}{[{\xreffont 32}]} directly produce expressions which include any arithmetic factors and furthermore include all main terms. In particular, Case~1 and Case~2 above are handled by the ratios conjecture and the recipe, respectively. However, RMT remains useful for the purpose of this paper because it enables us to generate examples for building intuition about \(Z(t)\) for large~\(t\).%
\end{paragraphs}%
\end{subsectionptx}
\end{sectionptx}
\typeout{************************************************}
\typeout{Section 8 Waves in a unitary polynomial}
\typeout{************************************************}
\begin{sectionptx}{Section}{Waves in a unitary polynomial}{}{Waves in a unitary polynomial}{}{}{rmtwaves}
\begin{introduction}{}%
As defined in \hyperref[RMTZ]{({\xreffont\ref{RMTZ}})},  the rotated characteristic polynomial \(\mathcal{Z}_A(\theta)\)  for Haar-random \(A \in U(N)\) is a model for \(Z(t)\) for \(t\) of size \(e^N\). If we let \(N=1000\), we obtain a model for the zeta function at height \(e^{1000} \approx 10^{434}\). That height is sufficient to ``see'' the carrier waves in the characteristic polynomial.  The analogy between \(L\)-functions and characteristic polynomials suggests that carrier waves in the \(\zeta\)-function at that height should be comparable.%
\par
We will illustrate carrier waves by picking \emph{one} large random matrix and exploring its characteristic polynomial in detail. The goal is provide intuition for carrier waves \textemdash{} these examples will not show extreme behavior. We will see that carrier waves occur within each individual random unitary polynomial.%
\end{introduction}%
\typeout{************************************************}
\typeout{Subsection 8.1 An example of a carrier wave}
\typeout{************************************************}
\begin{subsectionptx}{Subsection}{An example of a carrier wave}{}{An example of a carrier wave}{}{}{rmtwaves-3}
Suppose \(B \in U(1000)\) is a specific (randomly chosen) matrix, which is fixed and will be used in all the illustrations below. The function \(\mathcal{Z}_B(\theta)\) is real-valued, periodic with period \(2\pi\), and in principle it can be graphed. However, that graph would convey little useful information because 1000 wiggles across the width of a piece of paper is beyond the resolution of what can be printed or seen.  Instead, we offer two other ways to graph that function and see the carrier wave.%
\par
\hyperref[fig_two1000examples]{Figure~{\xreffont\ref{fig_two1000examples}}} shows graphs of \(\mathcal{Z}_B(\theta)\) over two intervals of width \(2\pi \times 30/1000\). Note the vertical scales in the graphs.%
\begin{figureptx}{Figure}{Graphs of \(\mathcal{Z}_B(\theta)\) over two intervals of width \(2\pi \times 0.03\).}{fig_two1000examples}{}%
\begin{image}{0}{1}{0}{}%
\includegraphics[width=\linewidth]{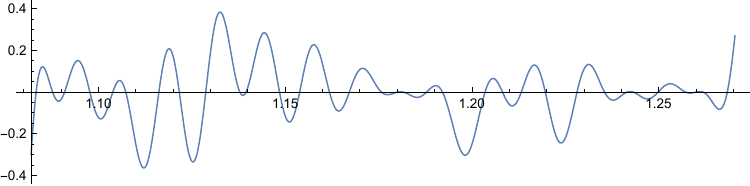}
\end{image}%
\begin{image}{0}{1}{0}{}%
\includegraphics[width=\linewidth]{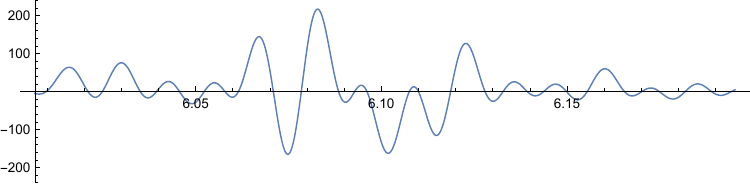}
\end{image}%
\tcblower
\end{figureptx}%
\hyperref[fig_two1000examples]{Figure~{\xreffont\ref{fig_two1000examples}}} illustrates the effect of the carrier wave. Both graphs are part of the \emph{same} polynomial, so the global scale factors are equal.  The variation in the local zero spacing has a similar effect in both regions, yet there is a factor of \(500\) difference in the graphs. That factor of \(500\) is due to the carrier wave.%
\par
We can use the idea of \hyperref[eqn_sineapprox]{({\xreffont\ref{eqn_sineapprox}})} to measure the carrier wave in \(\mathcal{Z}_B(\theta)\), with the obvious modification of replacing \(\cos(x)\) by a function with equally spaced zeros on the unit circle: \(c(\theta) = z^{N/2} - z^{-N/2}\), where \(z=e^{i\theta}\). At a given point \(e^{i\theta_0}\), we move the nearby zeros of \(c(\theta)\) to match the zeros of the polynomial, and then choose a scale factor so the functions are equal at \(e^{i\theta_0}\).  That scale factor is the carrier wave at \(e^{i \theta_0}\).  The resulting carrier wave for our example function \(\mathcal{Z}_B(\theta)\), sampled at a large number of points around the circle, is shown in \hyperref[fig_wave1000]{Figure~{\xreffont\ref{fig_wave1000}}}.  In the notation of \hyperref[eqn_sineapprox]{({\xreffont\ref{eqn_sineapprox}})}, we set \(K=10\), meaning that \(10\) zeros on either side of a given point were used to calculate the carrier wave. Note: the vertical axis is on a logarithmic scale. %
\begin{figureptx}{Figure}{The carrier wave for \(\mathcal{Z}_B(\theta)\), on a logarithmic scale, calculated using \(K=10\) in \hyperref[eqn_sineapprox]{({\xreffont\ref{eqn_sineapprox}})}.}{fig_wave1000}{}%
\begin{image}{0}{1}{0}{}%
\includegraphics[width=\linewidth]{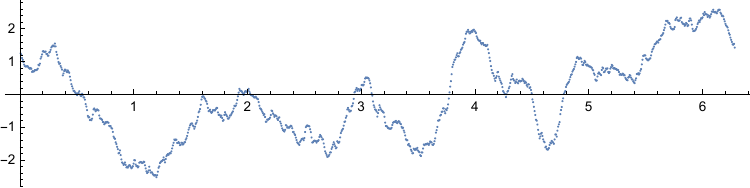}
\end{image}%
\tcblower
\end{figureptx}%
\hyperref[fig_wave1000]{Figure~{\xreffont\ref{fig_wave1000}}} indicates that near \(e^{1.2 i }\) the polynomial is wiggling with small amplitude, and near \(e^{6.1 i }\) it is wiggling with large amplitude.  That is how the examples were chosen for \hyperref[fig_two1000examples]{Figure~{\xreffont\ref{fig_two1000examples}}}.%
\begin{paragraphs}{Lower order effects from the zeros.}{rmtwaves-3-9}%
The carrier wave is the primary factor in determining the large values, but the local spacing will still have an effect.  If one focuses on the maximum value for a specific characteristic polynomial, the variation in the local spacing will cause that maximum to occur within a larger than average gap. This can be seen in the lower plot in \hyperref[fig_two1000examples]{Figure~{\xreffont\ref{fig_two1000examples}}}, where the maximum near \(6.083\) occurs in a gap of \(1.6\) times the average.  See Figure~5 in \hyperlink{FGK}{[{\xreffont 57}]} for a detailed numerical study of this phenomenon.%
\par
At \(N=1000\) the carrier wave is more important than the local variation. The  polynomial in \hyperref[fig_two1000examples]{Figure~{\xreffont\ref{fig_two1000examples}}} has 20 gaps which are more than twice the average, the largest of which is more than \(2.45\) times the average. But the local maximum in those gaps is smaller than in the \(1.6\) gap in the lower graph in \hyperref[fig_two1000examples]{Figure~{\xreffont\ref{fig_two1000examples}}} because the carrier wave in those other regions is smaller. In the random matrix world it is easy to compute in realms where \hyperref[mistakegap]{Mistaken Notion~{\xreffont\ref{mistakegap}}} is seen to be mistaken.%
\par
The zeros also have an influence on departures from the CLT at large values: the distribution has the shape of a Gaussian but is smaller by a multiplicative factor. See [\hyperlink{ArBa1}{{\xreffont 1}}, \hyperlink{ArBa2}{{\xreffont 2}}, \hyperlink{FMN1}{{\xreffont 55}}, \hyperlink{Rad1}{{\xreffont 92}}].%
\end{paragraphs}%
\end{subsectionptx}
\typeout{************************************************}
\typeout{Subsection 8.2 The density of zeros is more important than individual gaps}
\typeout{************************************************}
\begin{subsectionptx}{Subsection}{The density of zeros is more important than individual gaps}{}{The density of zeros is more important than individual gaps}{}{}{sec_density}
A close examination of the two plots in \hyperref[fig_two1000examples]{Figure~{\xreffont\ref{fig_two1000examples}}} reveals another key fact.  The two plots cover equal widths, namely 30 times the average gap between zeros.  But the first plot has 33 zeros, while the second plot has only 29 zeros.  What we have observed is:%
\begin{principle}{Principle}{}{}{prin_density}%
The main contribution to the carrier wave is the difference between the density of nearby zeros and the average density: the local density is negatively correlated with the size of the function.%
\end{principle}
Thus, the gaps between zeros determine the size of the function, but it is collections of nearby gaps which are the dominant factor: not individual gaps.%
\par
\hyperref[prin_density]{Principle~{\xreffont\ref{prin_density}}} would be more useful if we had a definition of ``local density'' and a way to measure it.  We will now develop an expression based on the gaps between widely spaced zeros.%
\par
In the top graph of \hyperref[fig_two1000examples]{Figure~{\xreffont\ref{fig_two1000examples}}} the function is small over a wide region. \hyperref[localconstant]{Subsection~{\xreffont\ref{localconstant}}} described the scale of the graph as arising from the zeros outside that region.  Since the zeros in that region are closer together than average, one can view this as being caused by the zeros on either side pushing in towards that region, compressing those zeros. Equivalently, one expects the distance between a zero to the left of that region and a zero to the right of that region to be smaller than expected (based on the number of intervening zeros and the average gap size). If \(\gamma, \gamma'\) are zeros, and%
\begin{equation}
\tilde{\gamma}' - \tilde{\gamma} = j + \delta(\gamma,\gamma')\label{sec_density-6-4}
\end{equation}
where \(j\) is the expected normalized gap between \(\gamma\) and \(\gamma'\), then we will call \(\delta(\gamma,\gamma')\) the \terminology{\(j\)th neighbor discrepancy}. Thus, when the zeros are more dense in a region, the \(j\)th neighbor discrepancy for pairs of zeros either side of that region, should tend to be negative.%
\par
To turn the neighbor discrepancy into a measure of the density, we need to determine the relative weights to give to all the neighbor gaps spanning the region.  Consider a function \(F(t)\) in a region centered on~\(0\), and renumber the zeros of \(F\) as \(\cdots \le \gamma_{-2} \le \gamma_{-1} \lt
0 \lt \gamma_1 \le \gamma_2 \le \cdots \).%
\begin{principle}{Principle}{}{}{measuredensity}%
The logarithm of the carrier wave is given to leading order by the \terminology{density wave}, expressed as a weighted sum of the \(j\)th neighbor discrepancies by%
\begin{equation}
\sum_{j\ge J} \frac{\delta(\gamma_{-j},\gamma_j)}{j}\text{.}\label{densityformula}
\end{equation}
\end{principle}
The sum \hyperref[densityformula]{({\xreffont\ref{densityformula}})} converges, by Dirichlet's test.%
\par
\hyperref[fig_density1000]{Figure~{\xreffont\ref{fig_density1000}}} compares the density wave of zeros for the same degree 1000 polynomial used in the above examples, to the logarithm of the carrier wave from \hyperref[fig_wave1000]{Figure~{\xreffont\ref{fig_wave1000}}}. Here the density wave is measured by \hyperref[densityformula]{({\xreffont\ref{densityformula}})} with \(J=11\). Note that the thin blue graph was calculated using \(20\) zeros centered around a given point, while the thicker red graph was calculated using the complementary set of 980 zeros. The agreement between those two functions, calculated using different sets of zeros, is remarkable.%
\begin{figureptx}{Figure}{The thicker red graph is the local density of zeros (the density wave) measured by \hyperref[densityformula]{({\xreffont\ref{densityformula}})} with \(J=11\), for the degree 1000 polynomial \(\mathcal Z_B(\theta)\). The thinner blue graph is the (logarithm of the) carrier wave for the same polynomial, calculated with \(K=10\).  The similarity of those functions illustrates \hyperref[prin_density]{Principle~{\xreffont\ref{prin_density}}} and \hyperref[measuredensity]{Principle~{\xreffont\ref{measuredensity}}}.}{fig_density1000}{}%
\begin{image}{0}{1}{0}{}%
\includegraphics[width=\linewidth]{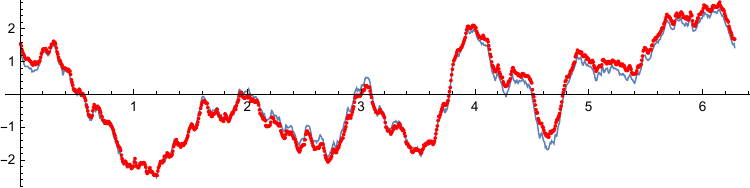}
\end{image}%
\tcblower
\end{figureptx}%
\begin{note}{Note}{}{densitynotunique}%
The formula in \hyperref[measuredensity]{Principle~{\xreffont\ref{measuredensity}}} depends on \(J\), so the density of the zeros at a point is not uniquely defined; the same shortcoming applies to the local scale factor (i.e., the carrier wave) which depends on the parameter \(K\), as observed in \hyperref[wavenotdefined]{Note~{\xreffont\ref{wavenotdefined}}}.  However, those are the same shortcomings: an interval of zeros is used to determine the value of the carrier wave. If the zeros used to calculate the pressure wave are the complement of the zeros chosen to calculate the carrier wave, as in \hyperref[fig_density1000]{Figure~{\xreffont\ref{fig_density1000}}}, then it is perhaps not surprising in retrospect that those neighbor gaps provide an accurate measure of the density of zeros they contain, and so are highly correlated with the carrier wave.%
\par
An unresolved issue is how to identify the appropriate window. In \hyperref[how_wide]{Subsection~{\xreffont\ref{how_wide}}} we suggest that the carrier wave for degree \(N\) polynomials covers around \(\log N\) zeros, which is around \(7\) for \(N=1000\), so the size window we are using in these examples is not unreasonable. Higher degree polynomials will require a wider window.%
\par
If \(N\) is sufficiently large then the number of zeros in the interval used to calculate the wave will itself be large, suggesting some type of fractal waves-within-waves phenomenon. Directly exhibiting such behavior, even in the random matrix world, seems out of reach, although some of the ideas in \hyperref[multiscale]{Section~{\xreffont\ref{multiscale}}} could be relevant.%
\end{note}
\begin{paragraphs}{Justification of Principle~{\xreffont\ref*{measuredensity}}.}{sec_density-13}%
We now justify \hyperref[densityformula]{({\xreffont\ref{densityformula}})}. Recall that we are considering a function~\(F\) with zeros \(\cdots \le \gamma_{-2} \le \gamma_{-1} \lt
0 \lt \gamma_1 \le \gamma_2 \le \cdots \). Let \(\kappa\) be the expected density of zeros near \(0\). For small \(t\) we have%
\begin{equation}
F(t/\kappa) = \prod_{j=1}^\infty
\biggl(1-\frac{t}{\tilde{\gamma}_j}\biggr)
\biggl(1-\frac{t}{\tilde{\gamma}_{-j}}\biggr)\text{.}\label{fproduct}
\end{equation}
\par
In \hyperref[fproduct]{({\xreffont\ref{fproduct}})} write \(\tilde{\gamma}_j = j \mp \frac12 + \delta_j\), where \(\mp \frac12\) has the opposite sign of \(j\). That choice assures that \(\delta_j\) is \(0\) on average. We have%
\begin{align}
F(t/\kappa) =& \prod_{j=1}^\infty
\biggl(1-\frac{t}{j - \tfrac12 + \delta_j}\biggr)
\biggl(1-\frac{t}{-j+\tfrac12 + \delta_{-j}}\biggr)\label{fproductrearranged-1}\\
=& \prod_{j=1}^\infty
\biggl(1-\frac{t}{j-\frac12}+ \frac{t\delta_j}{(j-\frac12)^2}-\frac{t \delta_j^2}{(j-\frac12)^3} + \cdots\biggr)\label{fproductrearranged-2}\\
& \phantom{xxxx}\times
\biggl(1+\frac{t}{j-\frac12}+ \frac{t\delta_{-j}}{(j-\frac12)^2}+\frac{t \delta_{-j}^2}{(j-\frac12)^3} + \cdots\biggr)\label{fproductrearranged-3}\\
=& \prod_{j=1}^\infty
\biggl(1-\frac{t^2 + t (\delta_j + \delta_{-j})}{(j-\frac12)^2}+
\frac{t^2(\delta_j-\delta_{-j}) - t(\delta_j^2 - \delta_{-j}^2)}{(j-\frac12)^3}+\cdots\biggr)\label{fproductrearranged-4}\\
=& \prod_{j=1}^\infty
\biggl(1-\frac{t^2 }{(j-\frac12)^2}+
\frac{t^2(\delta_j-\delta_{-j}) }{(j-\frac12)^3}+\cdots\biggr)\label{fproductrearranged-5}\\
=& \prod_{j=1}^\infty
\biggl(1-\frac{t^2 }{(j-\frac12)^2}\Bigl( 1+
\frac{\delta_j-\delta_{-j} }{ j-\frac12 }\Bigr)+\cdots\biggr)\text{.}\label{fproductrearranged-6}
\end{align}
On the next-to-last line we moved terms involving \(\delta_j + \delta_{-j}\) and \(\delta_j^2 - \delta_{-j}^2\) into the lower order terms because those expressions do not detect departures from the average zero spacing. Now recognize \(\prod_{j\ge 1}\bigl(1-t^2/(j-\frac12)^2\bigr)\) as \(\cos(\pi t)\) and factor it out of the above expression. The terms with \(j\ge J\) give the carrier wave. Taking the logarithm, expanding, and replacing \(j-\frac12\) by \(j\), because \(j\ge J\) is large, we have \hyperref[densityformula]{({\xreffont\ref{densityformula}})} because \(\delta_j-\delta_{-j}=\delta(\gamma_{-j},\gamma_j)\).%
\end{paragraphs}%
\end{subsectionptx}
\typeout{************************************************}
\typeout{Subsection 8.3 The carrier wave and \(S(t)\)}
\typeout{************************************************}
\begin{subsectionptx}{Subsection}{The carrier wave and \(S(t)\)}{}{The carrier wave and \(S(t)\)}{}{}{rmtwaves-5}
The relationship between zero density and the size of the carrier wave can also be expressed in terms of \(S(t)\).  Since \(S(t)\) is the error term in the zero counting function, we see that \(S(t)\) will be generally increasing in a region where the zeros are more dense than the local average, and vice-versa. By ``generally'' we mean when averaged across a range of several zeros, because of course \(S(t)\) is decreasing wherever it is continuous, and jumps by 1 unit at each critical zero.  Thus we have:%
\begin{principle}{Principle}{}{}{prin_Z_S}%
In regions where \(Z(t)\) is particularly large, \(S(t)\) will on average be decreasing.  In regions where \(Z(t)\) is particularly small, \(S(t)\) will on average be increasing.%
\end{principle}
\hyperref[fig_S1000wave]{Figure~{\xreffont\ref{fig_S1000wave}}} shows \(\mathcal S_B(\theta)\),  superimposed on the logarithm of the carrier wave, for \(\mathcal{Z}_B(\theta)\).%
\begin{figureptx}{Figure}{The thinner blue graph is the carrier wave of \(\mathcal{Z}_B(\theta)\) from \hyperref[fig_wave1000]{Figure~{\xreffont\ref{fig_wave1000}}} and the thicker green graph is the \(\mathcal S_B(\theta)\) for \(\mathcal{Z}_B(\theta)\).  The graphs illustrate \hyperref[prin_Z_S]{Principle~{\xreffont\ref{prin_Z_S}}}: \(\mathcal S_B(\theta)\) is increasing\slash{}decreasing in regions where \(\mathcal{Z}_B(\theta)\) is particularly small\slash{}large.}{fig_S1000wave}{}%
\begin{image}{0}{1}{0}{}%
\includegraphics[width=\linewidth]{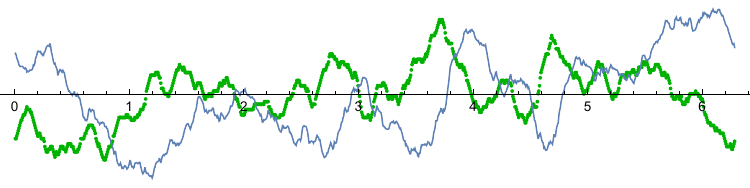}
\end{image}%
\tcblower
\end{figureptx}%
\hyperref[prin_Z_S]{Principle~{\xreffont\ref{prin_Z_S}}} explains why carrier waves are not observed in plots of \(Z(t)\): \(S(t)\) remains small throughout the realm accessible by computers, so there is no opportunity for it to increase or decrease enough to create a visible carrier wave.%
\par
Note that the relationship described in \hyperref[prin_Z_S]{Principle~{\xreffont\ref{prin_Z_S}}} does not contradict Selberg's result that \(\log\abs{Z(t)}\) and \(S(t)\) are statistically independent. However, it does suggest (negative) correlations in lower order terms and in the tails.%
\par
The failure to observe carrier waves computationally is also explained by \hyperref[prin_density]{Principle~{\xreffont\ref{prin_density}}}. We cloak this in fancy language in the next subsection.%
\end{subsectionptx}
\typeout{************************************************}
\typeout{Subsection 8.4 Carrier waves are an emergent phenomenon}
\typeout{************************************************}
\begin{subsectionptx}{Subsection}{Carrier waves are an emergent phenomenon}{}{Carrier waves are an emergent phenomenon}{}{}{emergent}
The conclusion to all the discussion in this section is that carrier waves are an \terminology{emergent phenomenon}.  That is defined as an observable property of a system which arises from the interaction of many components of the system, but is not a property of those individual components. In the case at hand, the components are the zeros and the interaction is the repulsion between the zeros. The emergent property is variations in the density of zeros.%
\par
An emergent property cannot be understood merely by examining the components and their interactions. The study of microscopic water droplets interacting with gaseous nitrogen and oxygen does not reveal the different types of clouds, and cannot tell you that a particular cumulus cloud is shaped like a rabbit.%
\par
Emergent phenomena can only exist in a large system with many interacting components. Points made earlier in this section explain why carrier waves can only occur in high degree polynomials. By \hyperref[prin_density]{Principle~{\xreffont\ref{prin_density}}}, the carrier wave arises from fluctuations in the density of the zeros. The concept of density does not make sense unless one has more than just a few points. Once one has enough points to meaningfully talk about their density, those points must exist in a larger system in order to allow the local density to be significantly different than average. And that system must be yet larger in order for the repulsion between points to allow enough flexibility that significant fluctuations in density are likely. Thus, carrier waves can only exist within a large system.%
\par
The explanation above describes why carrier waves do not appear in unitary polynomials until the degree is large. But what about the \(\zeta\)-function, which has infinitely many zeros? Isn't that a large system, even if one restricts to the range accessible by computers?%
\par
The explanation lies in the rigidity of the zero spacing, i.e.\@, the slow growth of \(S(t)\), interpreted as the error term in the zero counting function.  Since the zeros must stay close to their expected location, there is limited opportunity to create a region where there are significantly more, or significantly fewer zeros. Not only does \(S(t)\) grow slowly, it also changes sign frequently. It is as if the \(\zeta\)-function not only resets itself after every interval of length \(2\pi\), there is an anti-correlation between extreme values on one interval and the next.  Thus, a sequence of independent characteristic polynomials slightly overestimates the frequency of large values. In other words, the arithmetic factors used to adjust random matrix predictions, such as \(a_k\) in the \(2k\)th moment of the \(\zeta\)-function, see \hyperref[leadingtwokthmoment]{({\xreffont\ref{leadingtwokthmoment}})}, tend to be very small. This is an aspect of \hyperref[KSoverestimate]{Principle~{\xreffont\ref{KSoverestimate}}}.%
\end{subsectionptx}
\typeout{************************************************}
\typeout{Subsection 8.5 A hidden choice in the Keating-Snaith law}
\typeout{************************************************}
\begin{subsectionptx}{Subsection}{A hidden choice in the Keating-Snaith law}{}{A hidden choice in the Keating-Snaith law}{}{}{kappa1}
The matching \(N=\log(T/2\pi)\) is designed to allow the random characteristic polynomial \(\mathcal Z_A(\theta)\) with \(A\in U(N)\) to model \(Z(t)\) for \(t\approx T\). The justification is ``matching the density of zeros'', but that justification contains a hidden choice. If instead we chose some real number \(\kappa\), let \(A\in U(\kappa N)\), and considered \(\mathcal Z_A(\theta/\kappa)\), we again would have random functions with the same average zero gap as~\(Z(t)\). Maybe there is a better choice than \(\kappa=1\)?  It is tempting to think that a larger value of \(\kappa\) would be better, because the standard choice imposes the absurdity \(Z(t) = Z(t + 2\pi)\) \textemdash{} although \(Z(t) = Z(t + 2\pi\kappa)\) for some \(\kappa \gt 1\) is not particularly less absurd.%
\par
Choosing a larger value for \(\kappa\) does not work, and carrier waves explain why. A rescaled characteristic polynomial \(\mathcal Z_A(\theta/10)\) for \(A\in U(10 N)\)  does not look like the concatenation of \(10\) different \(\mathcal Z_A(\theta)\) for \(A\in U(N)\). The rescaled polynomial will have larger carrier waves, because the system is larger so there is more scope for the zeros to slosh back-and-forth and create regions of higher and lower density. A side-effect will be that moments of the polynomials will not accurately predict moments of the \(\zeta\)-function:  the maximum from a \(10\times\) larger system will typically be much larger than the maximum among \(10\) different smaller systems.%
\par
Thus, in some ways, such as the lack of large carrier waves at low height, the \(\zeta\)-function behaves like a sequence of small systems, not a single large system.  But actually it is specific properties of the large \(\zeta\)-system which is the underlying reason the carrier waves start out small and are confined to short intervals.  This is explained in the next section.%
\end{subsectionptx}
\typeout{************************************************}
\typeout{Subsection 8.6 Spectral rigidity}
\typeout{************************************************}
\begin{subsectionptx}{Subsection}{Spectral rigidity}{}{Spectral rigidity}{}{}{spectralrigidity}
Another way to see that the \(\zeta\)-function cannot be modeled by a small number of large characteristic polynomials involves the analogue of what is called \terminology{spectral rigidity} \hyperlink{berryA}{[{\xreffont 16}]} in physics. For the \(\zeta\)-function, that terminology is just another way of saying that \(S(T)\), the error term in the zero counting function \(N(T)\), grows very slowly. We examine this in the context of the neighbor spacing between zeros, in particular the normalized \(j\)th nearest neighbor spacing \(\tilde\gamma_{n+j} - \tilde\gamma_n\). \hyperref[nn3]{Figure~{\xreffont\ref{nn3}}} shows the distribution of the 1st, 3rd, and 10th normalized neighbor spacings for \(10000\) consecutive zeros at height \(10^{12}\).  Data taken from \hyperlink{odl12}{[{\xreffont 87}]}, which is the same data used in Berry's original study \hyperlink{berryA}{[{\xreffont 16}]} of spectral rigidity in the context of the \(\zeta\)-function.%
\begin{figureptx}{Figure}{The normalized 1st, 3rd, and 10th neighbor spacing for \(\gamma_n\) for \(10^{12} \le n \lt 10^{12} + 10000\), along with the variance of each sample.}{nn3}{}%
\begin{sidebyside}{3}{0}{0}{0.05}%
\begin{sbspanel}{0.3}[bottom]%
\noindent\includegraphics[width=\linewidth]{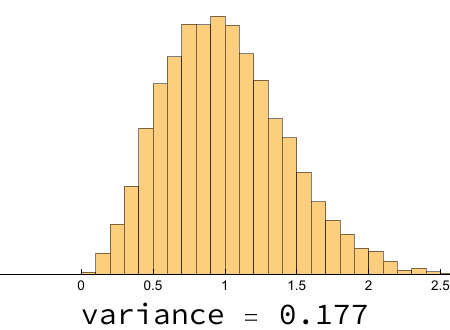}
\end{sbspanel}%
\begin{sbspanel}{0.3}[bottom]%
\noindent\includegraphics[width=\linewidth]{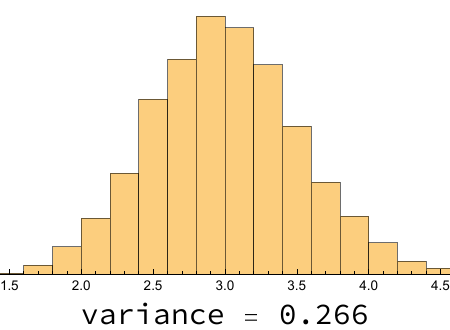}
\end{sbspanel}%
\begin{sbspanel}{0.3}[bottom]%
\noindent\includegraphics[width=\linewidth]{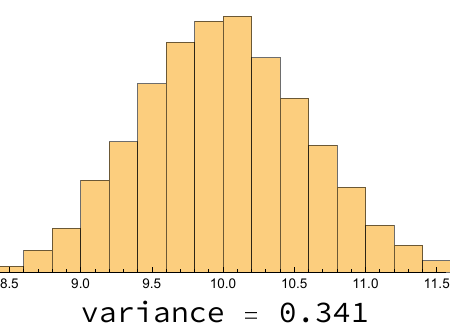}
\end{sbspanel}%
\end{sidebyside}%
\tcblower
\end{figureptx}%
The histograms in \hyperref[nn3]{Figure~{\xreffont\ref{nn3}}} have the same horizontal scales, so one can see (even without the calculated variances) that the variance of the \(n\)th neighbor spacing is growing for \(n\le 10\).  However, it is clear that the variance cannot continue to grow at that rate, because the variance of the neighbor spacing is bounded by a small multiple of the variance  \(\langle S(T)^2\rangle = \log\log T\). \hyperref[nnvar]{Figure~{\xreffont\ref{nnvar}}} shows the variance of the \(j\)th nearest-neighbor spacing for \(1\le j \le 200\) for the same set of zeros. Note that the phenomenon illustrated in \hyperref[nnvar]{Figure~{\xreffont\ref{nnvar}}} is usually expressed in terms of the \terminology{number variance} [\hyperlink{berryA}{{\xreffont 16}}, \hyperlink{berrykeatingA}{{\xreffont 17}}],%
\begin{equation}
V_T(L) = \mathrm{Var}(N(T+L) - N(T)) = \mathrm{Var}(S(T+L) - S(T))\text{,}\label{spectralrigidity-4-11}
\end{equation}
which is easier to handle theoretically but perhaps more difficult to grasp intuitively.%
\begin{figureptx}{Figure}{The blue dots show the variance of the \(j\)th nearest-neighbor spacing, \(0\le j \le 200\), for zeros \(\gamma_n\) with \(10^{12} \le n \lt 10^{12} + 10000\). The height of the horizontal red dashed line is twice the average variance of the samples, which would be the variance of the neighbor gaps if the zeros were distributed independently. The vertical dotted lines are at \(c\, \gamma_k\), with \(c\approx 3.894\) and \(\gamma_j\) the height of the \(j\)th nontrivial zero of the \(\zeta\)-function. The green dots are the variance of the \(j\)th nearest-neighbor spacing, \(0 \le j \le 24\), for eigenvalues of random matrices in~\(U(24)\).}{nnvar}{}%
\begin{image}{0}{1}{0}{}%
\includegraphics[width=\linewidth]{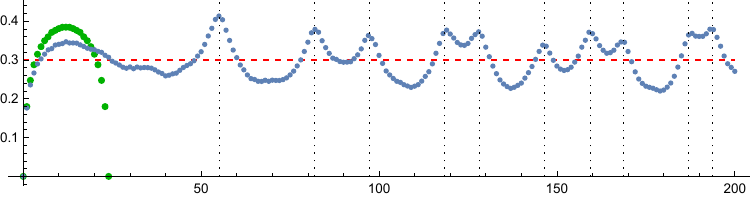}
\end{image}%
\tcblower
\end{figureptx}%
Note that the average value of \(\tilde\gamma_n - n\) is \(\frac12\), and in the sample used in the above figures, has variance \(0.15005\). If the \(j\)th nearest neighbors were distributed independently, then the \(j\)th nearest neighbor gaps would have variance \(2\times 0.15005\), which is the horizontal line in \hyperref[nnvar]{Figure~{\xreffont\ref{nnvar}}}.  Thus, the relative displacement of neighboring zeros (relative to their expected location) is sometimes positively correlated, and sometimes negatively correlated, across a span of many zeros.%
\par
The peaks in \hyperref[nnvar]{Figure~{\xreffont\ref{nnvar}}} occur near the rescaled zeros of the \(\zeta\)-function, where the scale factor is the density of the zeros at the height of the sample (in this example, \(t \approx 2.67 \times 10^{11}\)). This \terminology{resurgence} of the first few zeros occurs in many zero statistics [\hyperlink{berrykeatingA}{{\xreffont 17}}, \hyperlink{CS}{{\xreffont 38}}], typically manifested via the influence of the zeros on the size of~\(\zeta(1 + it)\).%
\par
The universality of \hyperref[betauniversality]{Principle~{\xreffont\ref{betauniversality}}} is illustrated in \hyperref[nnvar]{Figure~{\xreffont\ref{nnvar}}} by the approximate agreement, for very small gaps, between the \(\zeta\)-zeros and eigenvalues variances.  This is referred to as the \terminology{universal regime}. In the universal regime, no details about the \(\zeta\)-function are relevant to the leading-order behavior. This idea appears again in \hyperref[smallgaps]{Principle~{\xreffont\ref{smallgaps}}}. In the non-universal regime, the prime numbers (equivalently, the low-lying zeros) have a strong effect.  Only recently has this effect been treated rigorously~\hyperlink{LMQ-H}{[{\xreffont 81}]}.%
\par
The rigidity in the zeros spacing, equivalently the slow growth of \(S(t)\), prevents the zeros from moving far from their expected location, which prevents large regions with a high or low density of zeros, which prevents large carrier waves at low height.%
\end{subsectionptx}
\typeout{************************************************}
\typeout{Subsection 8.7 Extreme gaps}
\typeout{************************************************}
\begin{subsectionptx}{Subsection}{Extreme gaps}{}{Extreme gaps}{}{}{extremegaps}
In \hyperref[zero_stats]{Subsection~{\xreffont\ref{zero_stats}}} we described the GUE predictions for the largest and smallest gaps between zeros.  We now elaborate on how those predictions are made, and indicate some issues which impact those predictions.%
\begin{paragraphs}{Small gaps.}{extremegaps-3}%
\hyperref[GUEera]{Subsection~{\xreffont\ref{GUEera}}} mentioned the universality of the \(\beta\)-ensembles: the local statistics of the eigenvalues only depend (to leading order) on the degree of repulsion.  This has implications for the neighbor spacing. From the \(\CUE\) measure, \hyperref[betameasure]{({\xreffont\ref{betameasure}})} with \(\beta=2\), it follows immediately that the nearest-neighbor spacing vanishes to order~\(2\), and (by integrating \(\theta_j\) from  \(\theta_{j-1}\) to \(\theta_{j+1}\)) the next-nearest-neighbor spacing vanishes to order~\(7\). Thus, to leading order the pair correlation for the \(\CUE\), \hyperref[zetapc]{({\xreffont\ref{zetapc}})}, describes the distribution of small gaps between eigenvalues, and that prediction is not sensitive to any details about the particular system being discussed.  That is,%
\begin{principle}{Principle}{}{}{smallgaps}%
The PDF \(p_2(x)\) of the normalized gaps between zeros of the \(\zeta\)-function, or any other \(L\)-function, decays like \((\pi^2/3) x^2\) as \(x\to 0\).  That prediction is very likely to be accurate, with the prime numbers and other details particular to any specific \(L\)-function having little influence on the smallest gaps.%
\end{principle}
Any skepticism about \hyperref[smallgaps]{Principle~{\xreffont\ref{smallgaps}}} should be alleviated by \hyperref[nnPDF]{Figure~{\xreffont\ref{nnPDF}}}, which shows the distribution of the smallest normalized gaps among the first \(10^{13}\) zeros of the \(\zeta\)-function, compared to the \((\pi^2/3)x^2\) prediction. Note that \hyperref[nnPDF]{Figure~{\xreffont\ref{nnPDF}}} concerns the universal regime, so there is no contradiction with our repeated assertion that numerical calculations of the \(\zeta\)-function can give a misleading impression.%
\begin{figureptx}{Figure}{Histogram of the \(9333\) smallest normalized neighbor gaps among the first \(10^{13}\) zeros of the \(\zeta\)-function, compared to the predicted density \(\frac{\pi^2}{3}x^2\). Data from  Xavier Gourdon \hyperlink{XG}{[{\xreffont 60}]}.}{nnPDF}{}%
\begin{image}{0.1}{0.8}{0.1}{}%
\includegraphics[width=\linewidth]{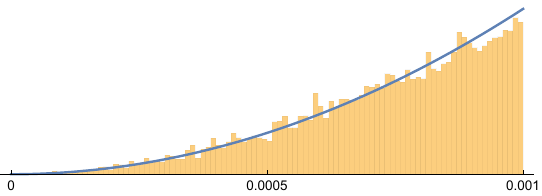}
\end{image}%
\tcblower
\end{figureptx}%
To estimate the size of the normalized smallest gaps among a sample of \(X\) gaps, calculate the gap size which has a 50\% chance of appearing:%
\begin{equation}
\text{solve }\ \ \ \ X \int_0^{\xi_{min}} A x^2\, dx = \frac12
\ \ \ \ \
\text{to get}
\ \ \ \ \
\xi_{min} = \Bigl(\frac{A X}{3}\Bigr)^{-\frac13}\text{,}\label{extremegaps-3-6-2}
\end{equation}
where \(A=\frac{\pi^2}{3}\). So, among the gaps up to height \(T\), one expects the smallest normalized gaps to be of size \(T^{-\frac13 + o(1)}\).  One can try to be more precise, but there probably is no meaning in the specific constant or power of \(\log T\). There is meaning, however, to the exponent \(-\frac13\). A small gap between two zeros forces a small value for the derivative of the \(\zeta\)-function at both of those zeros. An implication \hyperlink{HKO}{[{\xreffont 69}]} is that the predicted discrete mean values%
\begin{equation}
\frac{1}{N(T)} \sum_{0 \lt \gamma_j \le T} \abs{\zeta'(\tfrac12 + i\gamma_j)}^{2k} \sim c_k T^{2k(k+2)},\label{extremegaps-3-6-10}
\end{equation}
conjectured to hold for \(k\gt -\frac32\), cannot hold for \(k \lt -\frac32\) because (due to the smallest neighbor gaps) there are individual terms which are very large.%
\end{paragraphs}%
\begin{paragraphs}{Large gaps.}{extremegaps-4}%
The situation for large gaps is likely to be significantly more complicated.%
\par
For large random unitary matrices we have \(p_2(x) \sim e^{-\pi^2 x^2/8}\) as \(x\to \infty\).  For an individual matrix, the neighbor gaps are not independent because adjacent neighbor gaps are negatively correlated, and also they average to exactly~\(1\).  But, if one pretends those gaps are independent, then using the tail of \(p_2\) the expected maximum among \(M\) normalized gaps would be \(\sqrt{8 \log M}/\pi\).  Setting \(M=N\) gives the correct answer for the expected maximum among the gaps of an individual random matrix~\hyperlink{BA_B}{[{\xreffont 12}]}. Thus, there is some justification to using the tail of the \(\CUE\) neighbor spacing to conjecture the maximum gap over longer ranges, and Figure~1 in \hyperlink{BA_B}{[{\xreffont 12}]} shows that such a prediction is supported by data involving \(2\times 10^9\) zeros of the \(\zeta\)-function.%
\par
As illustrated in \hyperref[nnvar]{Figure~{\xreffont\ref{nnvar}}}, for more widely spaced zeros, where ``widely'' might only mean a few times the average zero spacing, the distribution of \(n\)th neighbor gaps for the \(\zeta\)-function departs significantly from that of random unitary matrices:  the \(\zeta\) gaps are more constrained. Therefore, it would not be surprising if the largest normalized gaps between zeros of the \(\zeta\)-function were smaller than the analogous normalized gaps between eigenvalues of random unitary matrices. It particular, it would not be surprising if the largest normalized neighbor gaps of the \(\zeta\)-function (either as conjectured in  \hyperlink{BA_B}{[{\xreffont 12}]}, or from a naive interpretation of the tail of \(p_2(x)\)) were actually smaller than \(\sqrt{8 \log T}/\pi\). See \hyperref[KSoverestimate]{Principle~{\xreffont\ref{KSoverestimate}}}.%
\par
Resolving such an issue with data is difficult because the distribution of large gaps decays rapidly. In this paper we only make use of the fact that the largest normalized gaps are likely to be \(O(\sqrt{\log T})\).%
\end{paragraphs}%
\end{subsectionptx}
\end{sectionptx}
\typeout{************************************************}
\typeout{Section 9 Extreme values}
\typeout{************************************************}
\begin{sectionptx}{Section}{Extreme values}{}{Extreme values}{}{}{extreme}
\begin{introduction}{}%
So far we have considered typical large and small values. That is, values within the bulk of the distribution of \(\log{\abs{Z(t)}}\). The conclusion was that those large values are due to the carrier wave.  Since the carrier wave arises from variations in the local density of zeros, once the wave is large it must stay large over a span of several zeros. Thus, the graph of the \(Z\)-function when it is particularly large looks the same as it does anywhere else \textemdash{} if you ignore the scales on the axes.%
\par
However, we have not yet ruled out the possibility that even larger values might arise from very rare events, such as an extremely large gap between zeros.  The effect might look similar to what appears in the data of Bober and Hiary \hyperlink{BobHia}{[{\xreffont 21}]}, as shown in \hyperref[fig_boberhiary]{Figure~{\xreffont\ref{fig_boberhiary}}}. We will explain why the largest neighbor gaps are not responsible for the function reaching the largest values. First a trivial observation, following immediately from \hyperref[prin_large_S_gaps]{Principle~{\xreffont\ref{prin_large_S_gaps}}} and \hyperref[maxgap]{({\xreffont\ref{maxgap}})}:%
\begin{principle}{Principle}{}{}{prin_large_S_gap_values}%
The largest gaps between zeros contribute \(O(\sqrt{\log t})\) to \(S(t)\).%
\end{principle}
It takes more effort to establish the analogous result for \(Z(t)\):%
\begin{principle}{Principle}{}{}{prin_large_Z_gaps}%
The largest gaps between zeros contribute \(O(\sqrt{\log t})\) to \(\log\abs{Z(t)}\).%
\end{principle}
Since the largest values of \(\log\abs{Z(t)}\) are conjectured to be \(\gg \sqrt{\log(t)\log\log(t)}\), \hyperref[prin_large_Z_gaps]{Principle~{\xreffont\ref{prin_large_Z_gaps}}} tells us that the largest neighbor gaps are not responsible for the extreme values.%
\par
\hyperref[prin_large_Z_gaps]{Principle~{\xreffont\ref{prin_large_Z_gaps}}} follows immediately from \hyperref[maxgap]{({\xreffont\ref{maxgap}})}, see also \hyperref[extremegaps]{Subsection~{\xreffont\ref{extremegaps}}}, combined with:%
\begin{principle}{Principle}{}{}{prin_large_gaps}%
The size of \(Z(t)\) at a point comes from two independent sources: the carrier wave, and the local arrangement of zeros. An individual large gap of \(g\) times the  average zero spacing contributes \(O(g)\) to the local maximum of \(\log\abs{Z(t)}\) within that large gap.%
\end{principle}
\hyperref[nearlarge]{Subsection~{\xreffont\ref{nearlarge}}} is devoted to justifying \hyperref[prin_large_gaps]{Principle~{\xreffont\ref{prin_large_gaps}}}.%
\end{introduction}%
\typeout{************************************************}
\typeout{Subsection 9.1 Zeros near a large gap}
\typeout{************************************************}
\begin{subsectionptx}{Subsection}{Zeros near a large gap}{}{Zeros near a large gap}{}{}{nearlarge}
To understand the effect of a large gap on the size of the function, we need information about the gaps adjacent to the large gap. From Haar measure \hyperref[betameasure]{({\xreffont\ref{betameasure}})} on \(U(N)\) we see that the expected size of the immediately adjacent gaps is smaller than average.  Bober and Hiary \hyperlink{BobHia}{[{\xreffont 21}]} expressed it as: \(S(t)\) tends to be increasing immediately before and immediately after a large gap. To isolate the effect of only the single large gap, we propose to focus on the case where \emph{the nearby zeros are in their most likely configuration}. We will do this for zeros on the unit circle.%
\par
Determining the most likely configuration is an easy computer experiment. Begin by fixing two zeros with a chosen gap and distributing the other \(N-2\) zeros on the circle anywhere outside that gap.  Equally spaced is a perfectly good starting configuration. Then repeatedly perturb the \(N-2\) zeros to increase the measure \hyperref[betameasure]{({\xreffont\ref{betameasure}})}, until the process stabilizes. The result will be a good approximation to the most likely configuration.%
\par
With \(N=74\) and a gap of 6 times the average, which is the random matrix analogue of \(Z(t)\) in \hyperref[fig_boberhiary]{Figure~{\xreffont\ref{fig_boberhiary}}}, the most likely configuration is shown in \hyperref[fig_RMT74]{Figure~{\xreffont\ref{fig_RMT74}}}.%
\begin{figureptx}{Figure}{Graphs of \(\mathcal Z(\theta)\) and \(\mathcal S(\theta)\) for a degree \(74\) polynomial with a zero gap of 6 times the average and all other zeros in their most likely configuration with respect to Haar measure~\hyperref[betameasure]{({\xreffont\ref{betameasure}})}.}{fig_RMT74}{}%
\begin{image}{0.08}{0.84}{0.08}{}%
\includegraphics[width=\linewidth]{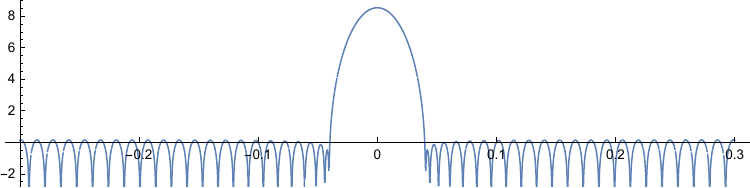}
\end{image}%
\begin{image}{0.08}{0.84}{0.08}{}%
\includegraphics[width=\linewidth]{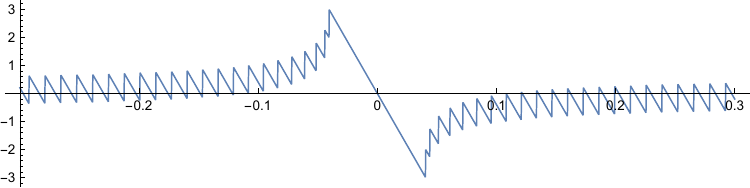}
\end{image}%
\tcblower
\end{figureptx}%
A similarity with \hyperref[fig_boberhiary]{Figure~{\xreffont\ref{fig_boberhiary}}} is the general shape of \(S(t)\) near the large gap.  A difference is that the maximum in \hyperref[fig_RMT74]{Figure~{\xreffont\ref{fig_RMT74}}} is smaller, which is due to the fact that the immediate neighbor gaps in \hyperref[fig_boberhiary]{Figure~{\xreffont\ref{fig_boberhiary}}} are wider than the most likely configuration. The wider neighboring gaps in \hyperref[fig_boberhiary]{Figure~{\xreffont\ref{fig_boberhiary}}} are not surprising, because the large gap in \hyperref[fig_RMT74]{Figure~{\xreffont\ref{fig_RMT74}}} is not as large as possible in that range. That means there is room for more variation in the nearby zeros, so less likely configurations can occur.%
\par
We see that a large gap causes nearby gaps to be smaller.  The key issue for justifying \hyperref[prin_large_gaps]{Principle~{\xreffont\ref{prin_large_gaps}}} is:  how far does that influence extend?  That is, at what distance from the large gap do the neighbor gaps (approximately) return to their average?  The answer is easy to express if we measure on the scale of the average gap between zeros.%
\begin{principle}{Principle}{}{}{prin_gap_distance}%
The normalized range of influence of a large gap is proportional to the normalized width of the gap.%
\end{principle}
So, doubling the width of a large gap will double the range of its influence on nearby neighbor gaps.%
\par
To see why \hyperref[prin_gap_distance]{Principle~{\xreffont\ref{prin_gap_distance}}} is true, consider the following thought experiment. Fix a large gap, and then distribute the other gaps in their most likely configuration.  It is helpful to think of the zeros as particles with equal charge, obeying a force law such that the particles are in equilibrium when they are in their most likely configuration. Now double all the charges.  The particles will still be in equilibrium. Then split every doubled charge into two equal charges, leave one in place, and move the other halfway toward the neighboring charge, with the movement away from the large gap.  The new configuration will not be in equilibrium, but it will be nearly so.  As the system settles into equilibrium (holding fixed the two particles forming the large gap), the particles will move very little, because they were already close to equilibrium. (The reader is invited to persuade themself that these equilibrium configurations are stable.  Also, as the system settles back into equilibrium, the slight movement of the particles brings them closer to the large gap.) The result is a new system where the large gap is the same on an absolute scale, as is the range of influence of that large gap.  But in the new system, the large gap and its influence are twice as large on the scale of the average spacing. Thus, \hyperref[prin_gap_distance]{Principle~{\xreffont\ref{prin_gap_distance}}}.%
\par
\hyperref[fig_gap230]{Figure~{\xreffont\ref{fig_gap230}}} provides another way to understand \hyperref[prin_gap_distance]{Principle~{\xreffont\ref{prin_gap_distance}}}.  For 230 zeros on the circle, which corresponds to \(Z(t)\) near \(10^{100}\), we used the procedure described at the start of this section to determine the most likely configuration with a gap of 5, 10, or 15 times the average.  The first plot in \hyperref[fig_gap230]{Figure~{\xreffont\ref{fig_gap230}}} shows the successive normalized neighbor gaps in each of the three cases. As expected, the immediate neighbor gap is small, and is smaller when the large gap is bigger. When the large gap is 5 times the average, the immediate neighbor gap is approximately \(0.32\) times the average. When the large gap is 15 times the average, the immediate neighbor gap is approximately \(0.1\) times the average. The gaps grow, approaching 1, the normalized average gap, when farther away from the large gap.  And as expected, when the large gap is larger, it takes longer for the neighbor gaps to approach 1.%
\par
The second plot in \hyperref[fig_gap230]{Figure~{\xreffont\ref{fig_gap230}}} is made from the same data, but this time the horizontal axis is scaled by the size of the large gap.  We see that the three graphs are virtually indistinguishable. In other words, measured on the scale of the large gap, the decay of the influence of that gap is independent of the size of the gap. By 3 times the width of the large gap, there is very little influence on the neighbor gaps.%
\begin{figureptx}{Figure}{The typical normalized gap between successive zeros adjacent to a large gap of 5, 10, or 15 times the average spacing.   In the plot on the top the horizontal scale is the index of the gap. On the bottom the horizontal scale is the index divided by the size of the large gap.}{fig_gap230}{}%
\begin{image}{0.1}{0.8}{0.1}{}%
\includegraphics[width=\linewidth]{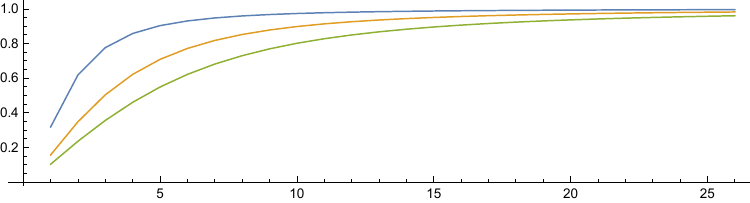}
\end{image}%
\begin{image}{0.1}{0.8}{0.1}{}%
\includegraphics[width=\linewidth]{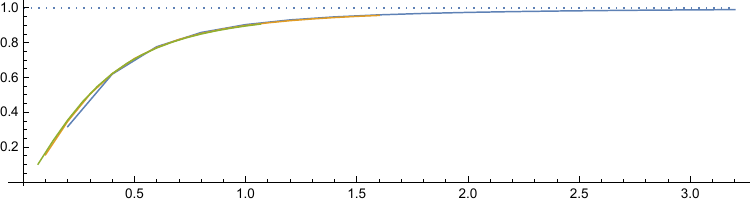}
\end{image}%
\tcblower
\end{figureptx}%
We use \hyperref[prin_gap_distance]{Principle~{\xreffont\ref{prin_gap_distance}}} to justify \hyperref[prin_large_gaps]{Principle~{\xreffont\ref{prin_large_gaps}}}.  The idea is similar to \hyperref[eqn_sineapprox]{({\xreffont\ref{eqn_sineapprox}})}. Start with a function having equally spaced zeros, and then move some of those zeros to a new location.  In this case it is easiest to start with \(f_0(z) = z^N + 1\) and then move the zeros with \(-\frac{\pi}{N} g \lt \theta_j \lt \frac{\pi}{N} g\) to create a normalized gap of size \(g\), placing those zeros in their most likely location. Denote the new value of \(\theta_{j}\) by \(\hat\theta_{j}\), and call the resulting function \(f_g(z)\). We have%
\begin{equation*}
f_g(z) = (z^N + 1) \prod_{-\pi g \le j \le \pi g}
\frac
{1-e(\hat\theta_j)}
{1-e(\theta_j)}\text{.}
\end{equation*}
By \hyperref[prin_gap_distance]{Principle~{\xreffont\ref{prin_gap_distance}}}, the zeros which were moved will have \(\abs{\hat\theta_j} \le K g/N\) where \(K\) is some absolute constant, independent of \(g\) and \(N\).  Therefore%
\begin{align*}
f_g(1) \ll \mathstrut \amp  \prod_{-\pi g \le j \le \pi g} \frac{
\frac{2\pi K g}{N}
}{
\frac{2\pi (2j - 1)}{2N}
}\\
\ll \mathstrut \amp \frac{K^{2g} g^{2\pi g}}{(\pi g)!^2}\\
= \mathstrut \amp  e^{O(g)}\text{,}
\end{align*}
the last step following from Stirling's formula. That establishes \hyperref[prin_large_gaps]{Principle~{\xreffont\ref{prin_large_gaps}}}.%
\par
We have seen that an isolated large gap has essentially the same effect on the sizes if \(S(t)\) and \(\log\abs{Z(t)}\). Yet, it was almost trivial to handle \(S(t)\), and fairly complicated to handle \(\log\abs{Z(t)}\). Why is it simple to understand the effect of large gaps on the imaginary part of \(\log \zeta(\frac12 + i t)\) but not on the real part? Is there a simple way to deal with both cases? Perhaps not: the discussion at the end of Section~1.2 of~\hyperlink{LMQ-H}{[{\xreffont 81}]} describes some essential differences between \(S(t)\) and \(\log\abs{\zeta(\frac12 + it)}\), the most important being that \(\log\abs{\zeta(\frac12 + it)}\) is unbounded in a neighborhood of its discontinuities, which they note ``cause technical difficulties''.%
\end{subsectionptx}
\typeout{************************************************}
\typeout{Subsection 9.2 More intuition for the need for carrier waves}
\typeout{************************************************}
\begin{subsectionptx}{Subsection}{More intuition for the need for carrier waves}{}{More intuition for the need for carrier waves}{}{}{wave_intuition}
We describe yet another way to see that carrier waves are primarily responsible for the size of~\(Z(t)\).%
\par
Suppose the typical large value of \(Z(t)\) were primarily due to large gaps, such as in \hyperref[fig_boberhiary]{Figure~{\xreffont\ref{fig_boberhiary}}}. Combining Selberg's CLT on the distribution of \(\log\abs{Z(t)}\) with \hyperref[prin_large_gaps]{Principle~{\xreffont\ref{prin_large_gaps}}} would imply that a positive proportion of \(t\) occur in a gap of relative size \(\sqrt{\log\log(t)}\).  In other words, the probability that a gap is larger than \(\sqrt{\log\log(t)}\) times the average, must be \(\gg 1/\sqrt{\log\log(t)}\).  This would correspond to the distribution of nearest-neighbor spacings decreasing much more slowly than conjectured. That is contrary to our observations and not consistent with the conjectures on neighbor spacing arising from RMT.%
\par
The situation becomes more clear when we consider small values. Obviously the zeros on the critical line cause very small (i.e.\@, large negative) values for \(\log\abs{Z(t)}\). It is not difficult to see (consider \(\log\abs{\sin x}\), for example) that the neighborhoods of the zeros cannot account for \(\log\abs{Z(t)}\) being \(\ll -\sqrt{\log\log(t)}\) for a positive proportion of \(t\).  Therefore the typical small values must arise from the function staying small between neighboring zeros.%
\par
Suppose two neighbor zeros are separated by \(g\) times the average spacing, with \(g\) very small. By considering \(x(g-x)\)  or \(\sin(x) - 1 + g^2\), one expects the local maximum between those zeros to scale as \(g^2\).  Thus, if the local zero spacing is the main cause of those small values, then a positive proportion of \(t\) must occur within a gap of size \(o(1)\) times the local average.  That is absurd because it requires more than 100\% of the gaps to be very small! Thus, the large proportion of small values required by Selberg's CLT must come from a conspiracy involving many gaps, with most of those gaps not particularly small.%
\end{subsectionptx}
\typeout{************************************************}
\typeout{Subsection 9.3 The Lindelöf Hypothesis and self-reciprocal polynomials}
\typeout{************************************************}
\begin{subsectionptx}{Subsection}{The Lindelöf Hypothesis and self-reciprocal polynomials}{}{The Lindelöf Hypothesis and self-reciprocal polynomials}{}{}{ssecLindel}
As mentioned in \hyperref[recipeera]{Subsection~{\xreffont\ref{recipeera}}}, an apparent failure in the analogy between the \(\zeta\)-function and random unitary polynomials is the Lindelöf Hypothesis:%
\begin{conjecture}{Conjecture}{The Lindelöf Hypothesis (LH).}{}{conjLH}%
If \(\varepsilon > 0\) then%
\begin{equation}
Z(t) = O_\varepsilon(t^\varepsilon)\label{conjLH-2-1-2}
\end{equation}
as \(t\to\infty\), or in other notation,%
\begin{equation}
Z(t) = e^{o(\log t)}\text{.}\label{conjLH-2-1-4}
\end{equation}
\end{conjecture}
The analogue for a unitary matrix \(A \in U(N)\) is%
\begin{equation}
\max_{|z|=1} {\mathcal Z}_A(z) = e^{o(N)}\text{.}\label{ssecLindel-4-2}
\end{equation}
That estimate need not hold: if all of the eigenvalues of \(A\) are in \(\Re z \ge 0\) then \(\abs{{\mathcal Z}_A(-1)} \ge \sqrt{2}^N\).%
\par
We suggest that when one is discussing LH (therefore implicitly not assuming RH because RH implies LH), the comparison should be with self-reciprocal polynomials \textemdash{} which are not necessarily unitary.  With that in mind, consider Backlund's \hyperlink{Bac}{[{\xreffont 7}]} equivalence (see \hyperlink{T}{[{\xreffont 110}]} Section 13.5). Let%
\begin{equation}
N(\sigma, T) = \#\{\rho=\beta + i \gamma\ :\ \zeta(\rho)=0,\ \beta \gt \sigma,\ 0 \lt t \le T\}\text{.}\label{ssecLindel-5-4}
\end{equation}
\begin{theorem}{Theorem}{Backlund [{\xreffont 7}].}{}{thmBac}%
The Lindelöf Hypothesis true if and only if%
\begin{equation}
N(\sigma, T+1) - N(\sigma, T) = o_\sigma(\log T)\label{NsigmaTdiff}
\end{equation}
for all \(\sigma > \frac12\).%
\end{theorem}
Backlund's equivalence \hyperref[NsigmaTdiff]{({\xreffont\ref{NsigmaTdiff}})} says that for any fixed \(A \gt 0\) and any fixed-width strip around the critical line, a negligible proportion of the zeros with \(T \le \gamma \le T+A\) lie outside that strip.%
\par
One of the implications in \hyperref[thmBac]{Theorem~{\xreffont\ref{thmBac}}} holds for self-reciprocal polynomials.  If \(P\) is a polynomial, let%
\begin{equation}
N(\delta, P) = \#\{ \rho = r e^{i \theta}\ : \ P(\rho) = 0,\ r \lt e^{-\delta} \}\text{.}\label{ssecLindel-8-3}
\end{equation}
\begin{proposition}{Proposition}{}{}{propBac}%
Suppose \(P\) is a self-reciprocal polynomial with%
\begin{equation}
\max_{\abs{z}=1} \abs{P(z)} = e^{o(\deg P)}\text{.}\label{LHpoly}
\end{equation}
Then%
\begin{equation}
N(\delta, P) = o_\delta(\deg P)\label{propBac-1-1-3}
\end{equation}
for all \(\delta > 0\).%
In other words, a negligible proportion of the zeros lie outside any fixed-width annulus around \(\abs{z}=1\).%
\end{proposition}
\begin{proof}{Proof}{}{propBac-3}
Suppose \(\delta > 0\) and factor \(P\) as%
\begin{equation}
P(z) = Q_\delta(z) Q^\delta(z)\label{Pjfactor}
\end{equation}
where the zeros of \(Q_\delta\) are the zeros of \(P\) in the annulus \(e^{-\delta} \lt \abs{z} \lt e^\delta\). Suppose the degree of \(Q^\delta\) is~\(2M\). The maximum of \(Q^\delta(z)\) on \(\abs{z}=1\) is decreased if the zeros of \(Q^\delta\) are moved to have equally spaced arguments and all have absolute value \(e^{-\delta}\) or \(e^{\delta}\).  Thus, it is sufficient to consider the case%
\begin{equation}
Q^\delta(z) = (z^M - e^{-\delta M}) (z^M - e^{\delta M})
= z^{2M} - (e^{-\delta M} + e^{\delta M}) z^M + 1\text{.}\label{Qdelta}
\end{equation}
By \hyperref[Qdelta]{({\xreffont\ref{Qdelta}})} we have  \(\abs{Q^\delta(z)} \ge e^{\delta M} - 3\) on all of \(\abs{z}=1\). Since \(Q_\delta(z)\) cannot be small on all of \(\abs{z}=1\) (because it is \(1\) on average), by \hyperref[LHpoly]{({\xreffont\ref{LHpoly}})} we have \(M = o(\deg P)\), as claimed.%
\end{proof}
The converse of \hyperref[propBac]{Proposition~{\xreffont\ref{propBac}}} is not true, and adapting the above proof fails because \(Q_\delta(z)\) in \hyperref[Pjfactor]{({\xreffont\ref{Pjfactor}})} could be large on \(\abs{z}=1\). Thus:%
\begin{principle}{Principle}{}{}{lindelof_two_ways}%
There are two different ways a self-reciprocal polynomial can be \emph{very large} on \(\abs{z} = 1\): there may be a large number of zeros far from the unit circle, or there may be a large gap (or other irregularity) in the arguments of the zeros.%
\par
For the \(\zeta\)-function, the second option does not occur, because the slow growth of \(S(t)\), equivalently the small error term in the zero counting function \(N(t)\), prevents extreme irregularities in the distribution of zeros.%
\end{principle}
By ``very large'' we mean a violation of (the polynomial analogue of)~LH.%
\par
In \hyperref[blanc]{Subsection~{\xreffont\ref{blanc}}} we will suggest that one of the claimed arguments against RH is actually an argument \emph{for} the possibility of improving bounds on the size of extreme gaps between zeros of the \(\zeta\)-function, based on bounds for the size of \(Z(t)\) and its derivatives.%
\end{subsectionptx}
\typeout{************************************************}
\typeout{Subsection 9.4 Carrier waves and mollifiers}
\typeout{************************************************}
\begin{subsectionptx}{Subsection}{Carrier waves and mollifiers}{}{Carrier waves and mollifiers}{}{}{ssecMollifiers}
In certain applications, variations in the size of the \(\zeta\)-function leads to inefficiencies.  This can be addresed with a \emph{mollifier}, which is an auxiliary function designed to dampen the wild behavior of the \(\zeta\)-function.  We use carrier waves to help understand why certain functions are effective as mollifiers.%
\par
A \terminology{mollifier} is a Dirichlet polynomial \(M(s)\) which is, in a sense we will make precise, an approximation to \(1/\zeta(s)\). In many applications, such as Levinson's method [\hyperlink{lev}{{\xreffont 80}}, \hyperlink{con25}{{\xreffont 29}}], one considers mean values like%
\begin{equation}
I(M, T) = \int_0^T |\zeta(\tfrac12 + i t) M(\tfrac12 + it)|^2 dt\label{ssecMollifiers-3-5}
\end{equation}
where \(M\) is a Dirichlet polynomial of length \(T^\theta\). One would like \(I(M, T) \sim c T\) for some \(c \gt 0\). If \(M = 1\) then \(I(1, T) \sim c T \log T\), so we see that a mollifier must be small where the \(\zeta\)-function is large. That is, \(M(s)\) ``mollifies'' the wild behavior of the \(\zeta\)-function.%
\par
It is clear that an approximation to \(1/\zeta(s)\) is a good candidate for a mollifier.  We will use the idea of carrier waves to motivate the use of more complicated mollifiers. These will be based on (approximations to the reciprocal of) linear combinations of the \(\zeta\)-function and its derivatives.%
\par
We rephrase the main idea of carrier waves, see \hyperref[localconstant]{Subsection~{\xreffont\ref{localconstant}}}. For most \(X\), if \(t\) is small we have%
\begin{equation}
\zeta(X + t) \approx e^{A(X)} P_X\Bigl(t \,\frac{\log X}{2\pi}\Bigr)\text{.}\label{ssecMollifiers-5-4}
\end{equation}
Here \(A(X)\) is normally distributed and changes slowly, \(P_X(x)\) is a polynomial with zeros that are spaced 1~apart on average, and the distribution of \(P_X\) changes slowly with~\(X\). In particular, \(A(X)\) is responsible for the size of the \(\zeta\)-function, and it is the reason \(I(1, T)\) is of size \(T \log T\).  Thus, the mollifier must counteract the \(e^{A(X)}\) factor, with \(P_X\) and the zeros from \(P_X\) playing a secondary role.%
\par
Because \(A(X)\) changes slowly and the distribution of \(P_X\) also changes slowly, we have that the \(k\)th derivative of the \(\zeta\)-function is the same size as the \(\zeta\)-function, times a factor of \(\log^{k} X\):%
\begin{equation}
\zeta^{(k)}(X + t) \approx e^{A(X)}
\Bigl(\frac{\log X}{2\pi} \Bigr)^k
P^{(k)}_X\Bigl(t \,\frac{\log X}{2\pi}\Bigr)\text{.}\label{ssecMollifiers-6-7}
\end{equation}
Thus, for any real numbers \(a_0,\ldots,a_K\), and with \(L=\log T / 2\pi \), a Dirichlet polynomial approximation to the reciprocal of%
\begin{equation}
\sum_{k=0}^K a_k \frac{1}{L^k} \zeta^{(k)}(s)\label{lincombzertdir}
\end{equation}
is a viable mollifier for the \(\zeta\)-function, and it is also a viable mollifier for any expression of the form~\hyperref[lincombzertdir]{({\xreffont\ref{lincombzertdir}})}. This idea lies behind various improvements to Levinson's method [\hyperlink{feng}{{\xreffont 54}}, \hyperlink{prattetal}{{\xreffont 90}}].%
\end{subsectionptx}
\end{sectionptx}
\typeout{************************************************}
\typeout{Section 10 The primes}
\typeout{************************************************}
\begin{sectionptx}{Section}{The primes}{}{The primes}{}{}{primes}
\begin{introduction}{}%
When describing the relationship between different objects, it is common to use anthropomorphic language, or to speak in terms of cause and effect.  For example, one of the principles below is ``the carrier wave arises from the prime numbers''. The intent is not to describe a causal relationship, but rather to provide intuition \textemdash{} intuition which possibly may eventually lead one to new insight or a new proof.%
\par
A previous principle asserted that the carrier wave arises from fluctuations in the density of zeros. Have we just contradicted ourself, and if so, which assertion is ``correct''?%
\par
Both principles are valid because both help provide intuition for a mental model of the behavior of the zeta-function. Equalities like%
\begin{equation}
\zeta(s) = \prod_p (1 - p^{-s})^{-1}
= \frac{e^{(\log(2 \pi) - 1 - \frac{\gamma}{2})s}}{2(s-1)\Gamma(1+\frac{s}{2})}
\prod_\rho (1 - s/\rho) e^{s/\rho}\text{,}\label{primes-2-3-1}
\end{equation}
where the first product is over the primes and the second is over the zeros, suggest that it is not meaningful to ask whether the primes cause the zeros or the zeros cause the primes. But if one has knowledge, or intuition, about one of those objects, then it is reasonable to ask what that suggests about the other.  This is where analogies between different areas help further the subject.  In the previous parts of this paper we primarily used knowledge and conjectures about the zeros as a starting point to understand the \(\zeta\)-function. In this section and in \hyperref[multiscale]{Section~{\xreffont\ref{multiscale}}} we take the primes as the starting point.%
\end{introduction}%
\typeout{************************************************}
\typeout{Subsection 10.1 Where are the primes?}
\typeout{************************************************}
\begin{subsectionptx}{Subsection}{Where are the primes?}{}{Where are the primes?}{}{}{primes-3}
The Riemann \(\zeta\)-function and its zeros were introduced because of their connection to the prime numbers, yet the previous several sections of this paper had little mention of the primes.%
\par
So, where are the primes?%
\begin{principle}{Principle}{}{}{prin_carrier_primes}%
The carrier wave arises from the prime numbers.%
\end{principle}
That principle does not contradict \hyperref[prin_density]{Principle~{\xreffont\ref{prin_density}}}, which says that the carrier wave arises from the relative density of the zeros.  Both principles have explanatory power. Which one is more illuminating will depend on the particular question at hand.%
\par
We partially justify \hyperref[prin_carrier_primes]{Principle~{\xreffont\ref{prin_carrier_primes}}} by Selberg's CLT \hyperref[selberg_gaussian]{({\xreffont\ref{selberg_gaussian}})} that \(\log \abs{\zeta(\frac12 + it)}\) has a Gaussian distribution. The starting point of the proof is to write%
\begin{equation}
\log \abs{\zeta(s)} = \sum_{\text{small primes } p} \Re\frac{1}{p^s}
+ \sum_{\text{zeros } \rho \text{  near } s} \log \abs{s-\rho} + \text{error terms}\text{.}\label{eqn_logzeta}
\end{equation}
The hard work is showing that, most of the time, the sum over zeros and the error terms are of size \(O(1)\). Thus, only the primes contribute to the size of the \(\zeta\)-function. One should expect a central limit theorem to hold for the sum over primes, because the \(\log p\) are linearly independent over the rationals. Thus, the variance when \(\Re(s) = \frac12\) should be \(\sum_{p\le x} p^{-1} \sim \log\log x\), which usually is larger than the size of the error terms.%
\par
Selberg's CLT is not an adequate justification for \hyperref[prin_carrier_primes]{Principle~{\xreffont\ref{prin_carrier_primes}}}, because by definition the carrier wave stays large over the span of several zeros. That does not immediately follow from the proof sketched above. In \hyperref[multiscale]{Section~{\xreffont\ref{multiscale}}} we describe more recent work which extracts additional information about the randomness of the \(\zeta\)-function from the primes.%
\par
\hyperref[prin_three_things]{Principle~{\xreffont\ref{prin_three_things}}} described the behavior of \(Z(t)\) as arising from a combination of local and global factors. We rephrase that in light of \hyperref[prin_carrier_primes]{Principle~{\xreffont\ref{prin_carrier_primes}}} (also see \hyperref[spectralrigidity]{Subsection~{\xreffont\ref{spectralrigidity}}}).%
\begin{principle}{Principle}{}{}{short-long}%
The short range behavior of the zeros is universal and follows random matrix statistics.  The long range behavior of the zeros is dictated by the prime numbers.%
\end{principle}
In the next subsection we quantify what is meant by ``short range'' and ``long range'' above.%
\end{subsectionptx}
\typeout{************************************************}
\typeout{Subsection 10.2 The hybrid model}
\typeout{************************************************}
\begin{subsectionptx}{Subsection}{The hybrid model}{}{The hybrid model}{}{}{sec_hybrid}
The \(\zeta\)-function can be expressed in terms of only zeros, as in the Hadamard product:%
\begin{equation}
\zeta(s) = \frac{e^{(\log 2\pi - \frac{\gamma}{2}-1)s}}{2 (s-1)\Gamma(1+\frac{s}{2})}
\prod_\rho \biggl(1-\frac{s}{\rho}\biggr) e^{\frac{s}{\rho}}\text{.}\label{sec_hybrid-2-2}
\end{equation}
And it can be expressed in terms of only primes:%
\begin{equation}
\zeta(s) = \exp\left(\sum_n \frac{\Lambda(n)}{n^s\log n}\right),
\ \ \ \ \ \ \sigma \gt 1\text{,}\label{sec_hybrid-2-3}
\end{equation}
where \(\Lambda(n)\) is the von-Mangoldt function, which is supported on prime powers.%
\par
Here we describe the \terminology{hybrid model}, which involves both zeros and primes, where one can adjust the relative contribution of each. Gonek, Hughes, and Keating~\hyperlink{hybrid}{[{\xreffont 59}]} proved that if \(s=\sigma+ i t\), with \(0\leq\sigma\leq 1\) and \(|t|\geq 2\), then for \(X>2\) and \(K\) any positive integer,%
\begin{equation}
\zeta(s) = P_X(s) Z_X(s) \left(1+O\left(\frac{X^{2-\sigma+K}}{(|t|
\log X)^{K}}\right)+O(X^{-\sigma}\log X)\right)\text{,}\label{eq_zetaPZ}
\end{equation}
where%
\begin{equation}
P_X(s) := \exp\left(\sum_{n\leq X}\frac{\Lambda(n)}{n^{s}\log n}
\right)\label{eq_P_X}
\end{equation}
and%
\begin{equation}
Z_X(s):= \exp\left(-\sum_{\rho}U((s-\rho)\log X)\right)\text{.}\label{eq_Z_X}
\end{equation}
Here the \(\rho\) are non-trivial zeros of \(\zeta(s)\) and \(U(z) = \int_0^\infty u(x) E_1(z\log x)\; d x\), where \(E_1(z) = \int_{z}^{\infty} \frac{e^{-w}}{w}\;d w\) is the exponential integral and \(u\) is any smooth function supported in \([e^{1-1/X}, e]\).%
\par
The parameter \(X\) controls the relative influence of the primes and the zeros. If \(X\) is large, there are many primes in \(P_X(s)\), and only the zeros very close to \(s\) are relevant to \(Z_X(s)\). If \(X\) is small, the zeros further away from \(s\) also make a contribution to \(Z_X(s)\), but the number of primes in \(P_X(s)\) is diminished.%
\par
One might expect  \(Z_X\) and \(P_X\) to behave somewhat independently, and there are two bits of evidence to support that idea. First, Gonek, Hughes, and Keating~\hyperlink{hybrid}{[{\xreffont 59}]} propose the ``splitting conjecture'': if \(X\to\infty\) with \(X = O(\log^{2-\epsilon} T)\) then the \(2k\)th moment of \(\zeta(\frac12 + i t)\) splits as the product of the \(2k\)th moment of \(P_X(\frac12 + i t)\) and \(Z_X(\frac12 + i t)\) separately. They rigorously compute the moments of \(P_X(\frac12 + i t)\), obtaining the expected arithmetic factor.%
\par
Second:  Farmer,  Gonek, and Hughes \hyperlink{FGH}{[{\xreffont 49}]} calculate the (conjectural) extreme value of the product \(P_X(\frac12 + i t)Z_X(\frac12 + i t)\) for \(T \le t \le 2T\), using a random model for each piece separately. They find that, for a wide range of \(X\), the expected maximum of the product is independent of~\(X\), even though the maximum of each piece separately depends on both \(X\) and~\(T\).%
\par
Both of the above results concern the leading order behavior, and there seems to be no reason to believe that any type of random model for the \(\zeta\)-function can accurately predict lower order behavior. However, recent work of Sawin~\hyperlink{saw}{[{\xreffont 102}]} concerns a random matrix model in the function field case, in which the lower order terms of the moments have a structure similar to the conjectured moments of the \(\zeta\)-function.%
\end{subsectionptx}
\end{sectionptx}
\typeout{************************************************}
\typeout{Section 11 Randomness not involving the zeros}
\typeout{************************************************}
\begin{sectionptx}{Section}{Randomness not involving the zeros}{}{Randomness not involving the zeros}{}{}{multiscale}
\begin{introduction}{}%
In this section we briefly describe how expressions like%
\begin{equation}
\log \zeta(\tfrac12 + i t) \ \ \ \leftrightarrow \ \ \ \sum_p \frac{1}{\sqrt{\mathstrut p}} \, p^{-i t}\label{primesum}
\end{equation}
can be used to understand the statistical behavior of the \(\zeta\)-function on the critical line. To a traditional analytic number theorist, quantities like \hyperref[primesum]{({\xreffont\ref{primesum}})} hurt the eyes and churn the stomach. But to a physicist or a probabilist, such expressions are a meaningful starting point for establishing interesting results.%
\par
It has been noted that characteristic polynomials do not know about the primes.  However, in the characteristic polynomial world there is an analogous expression \hyperlink{HKO}{[{\xreffont 69}]} to \hyperref[primesum]{({\xreffont\ref{primesum}})} for the characteristic polynomial of a matrix \(A\in U(N)\):%
\begin{equation}
\log \Lambda_A(e^{i \theta}) = \sum_n \frac{1}{\sqrt{\mathstrut n}} \, \frac{\Tr A^n}{\sqrt{\mathstrut n}} e^{i n\theta}\text{.}\label{multiscale-2-2-4}
\end{equation}
It is a theorem \hyperlink{DiaSha}{[{\xreffont 40}]} that for \(1\le n \le k\), with \(k\) fixed, as \(N\to\infty\) the \(\Tr A^n/ \sqrt{\mathstrut n}\) are i.i.d Gaussian.%
\par
In both worlds we have a weighted sum of i.i.d. random variables, so we would like to know their variance.  For the \(L\)-functions, the variance involves \(\sum_p p^{-1}\), which diverges like \(\log\log T\). For the characteristic polynomials we have \(\sum_n n^{-1}\), which diverges like \(\log N\). \hyperref[KSlaw]{The Keating-Snaith Law} \(N=\log T\) is a clue that we should expect similar types of randomness in the two worlds.%
\par
The perspective we take here is a metric space~\(X\), a circle or line our context, with each \(x\in X\) having an associated random (generalized) function \(F_x\). In other words, a random field. Of particular concern will be the relationship between the random functions at different points. If the covariance satisfies%
\begin{equation}
\langle F_x, F_{x'} \rangle = - \log\abs{x - x'} + g(x,x')\label{multiscale-2-4-5}
\end{equation}
where \(g\) is smooth, we say that the field is \terminology{log-correlated}.%
\par
To get into the spirit of this section, we encourage the reader to use \hyperref[primesum]{({\xreffont\ref{primesum}})} to heuristically manipulate \(\log \zeta(\frac12 + i t) \log\zeta(\frac12 - i s)\). Keep only the diagonal terms, and recognize the resulting sum as \(\log \zeta(1 + i (t-s))\).  Thus, in some average sense, \(\log \zeta(\frac12 + i t) \log\zeta(\frac12 - i s) \approx -\log\abs{t-s}\), so we conclude that the \(\zeta\)-function is log-correlated.  That heuristic was adapted from the Appendix in~\hyperlink{FyKe}{[{\xreffont 56}]}.%
\begin{paragraphs}{Log-correlation and physics.}{multiscale-2-6}%
The discussions throughout this section concern mathematical theorems, but just as important are intuitions from physics concerning log-correlated fields. The author thanks an anonymous referee for sharing this perspective.%
\par
First is that in a log-correlated field, maxima tend to occur in clusters. See \hyperlink{RhVa}{[{\xreffont 93}]}, in particular Figure~1, for an illustration. This is relevant to the size of carrier waves: the size of the clusters should be around \(\log\log N\), a conclusion we will reach for carrier waves by round-about means in \hyperref[how_wide]{Subsection~{\xreffont\ref{how_wide}}}. It is also relevant to the most extreme values, which are somewhat smaller than one might guess from a naive treatment of the value distribution. Since the values are highly correlated, it takes more energy to create large values because nearby values must also be large.%
\par
Second, in a log-correlated field the clusters form ``cascades'', which we will describe as ``valleys within valleys'' in \hyperref[smallwave]{Subsection~{\xreffont\ref{smallwave}}}.%
\end{paragraphs}%
\par\medskip
For details about the topics below, see [\hyperlink{BaKe}{{\xreffont 10}}, \hyperlink{SaksWebb}{{\xreffont 100}}, \hyperlink{SaksWebbSurv}{{\xreffont 101}}, \hyperlink{FyKe}{{\xreffont 56}}, \hyperlink{RhVa}{{\xreffont 93}}].%
\end{introduction}%
\typeout{************************************************}
\typeout{Subsection 11.1 Log-correlation and short-range maximum values}
\typeout{************************************************}
\begin{subsectionptx}{Subsection}{Log-correlation and short-range maximum values}{}{Log-correlation and short-range maximum values}{}{}{long_range}
Fyodorov, Hiary, and Keating \hyperlink{FHK}{[{\xreffont 58}]}, and Fyodorov and Keating \hyperlink{FyKe}{[{\xreffont 56}]}, considered the maximum values of the \(\zeta\)-function, not the global maxima, but over intervals of bounded length:%
\begin{equation}
\zeta_{\max} (L; T) = \max_{T\le t \le T+L} \log \abs{\zeta(\tfrac12 + it)}\text{.}\label{long_range-2-4}
\end{equation}
For comparison with characteristic polynomials, the interesting case is \(L = 2\pi\). What they found is that, as expected from the log-correlation of \(\zeta(\tfrac12 + it)\), the maximum values are slightly smaller than one might expect. The FHK conjecture \hyperlink{FHK}{[{\xreffont 58}]} is%
\begin{equation}
\zeta_{\max} (2\pi; T) \sim \log\log (T/2\pi) - C \log\log\log(T/2\pi) + X_T\text{,}\label{FHKconjecture}
\end{equation}
where%
\begin{itemize}[label=\textbullet]
\item{}\(C = \frac34\),%
\item{}\(X_T\) is a random variable which has a limiting distribution, \(X\), as \(T\to\infty\), and%
\item{}The PDF of \(X\) decays like \(x e^{-x}\) as \(x\to \infty\).%
\end{itemize}
Upper and lower bounds of size \hyperref[FHKconjecture]{({\xreffont\ref{FHKconjecture}})} are established in~[\hyperlink{ABR1}{{\xreffont 3}}, \hyperlink{Harper1}{{\xreffont 64}}] and~\hyperlink{ABR2}{[{\xreffont 4}]}, respectively.%
\par
Note that \hyperref[FHKconjecture]{({\xreffont\ref{FHKconjecture}})} is more than just a conjecture for the expected maximum value on an interval: it is a conjecture for the distribution of the maxima on such intervals.%
\par
The FHK conjecture relates to several areas of mathematical physics, and has an analogous statement for the \(\CUE(N)\) under the usual identification \(N=\log(T/2\pi)\). See Section~8 of \hyperlink{soundICM}{[{\xreffont 106}]} and Section~2 of \hyperlink{BaKe}{[{\xreffont 10}]} and references therein. We limit our discussion to the relationship between the FHK conjecture \hyperref[FHKconjecture]{({\xreffont\ref{FHKconjecture}})} and carrier waves.%
\begin{paragraphs}{Why the FHK conjecture is surprising.}{long_range-5}%
Much of the discussion below follows Section~8 of Soundararajan~\hyperlink{soundICM}{[{\xreffont 106}]}.%
\par
Consider the random matrix analogue of \(\zeta_{\max} (2\pi; T)\):%
\begin{equation}
\mathcal Z_{\max}(U) = \max_{\theta\in [0,2\pi]} \log \abs{\mathcal Z_U(e^{i\theta})}\text{.}\label{long_range-5-3-2}
\end{equation}
We will view \(\mathcal Z_{\max}(U)\) as arising from \(J\) samples of the random variable \(\log \abs{\mathcal Z_U(e^{i\theta})}\), for some sample size \(J\). If those samples were independent, and since \(\log\abs{Z_U(e^{i\theta})}\) has a Gaussian distribution with variance \(\sqrt{\frac12 \log N}\) (which is Keating and Snaith's analogue of Selberg's CLT), then \(\mathcal Z_{\max}(U)\) should typically be around%
\begin{equation}
\sqrt{\tfrac12 \log N} \, \bigl(\sqrt{2\log J} - \log \sqrt{4\pi \log J}/\sqrt{2\log J}\bigr)\text{.}\label{Jsamples}
\end{equation}
It remains to decide on the number of samples~\(J\). Clearly we cannot choose \(J \gt N\) because \(\mathcal Z_U\) is determined by the \(N\) eigenvalues of~\(U\).  If we choose \(J=N\), a seemingly reasonable choice because \(\abs{\mathcal Z_U(e^{i\theta})}\) has \(N\) local maxima and we want to know the largest, we find that \hyperref[Jsamples]{({\xreffont\ref{Jsamples}})} equals%
\begin{equation}
\log N - \tfrac14 \log\log N\text{.}\label{long_range-5-3-22}
\end{equation}
Setting \(N=\log T/2\pi\) we contradict the \(C=\frac34\) part of conjecture \hyperref[FHKconjecture]{({\xreffont\ref{FHKconjecture}})}.  That is the first way in which the FHK conjecture is surprising. In particular, their conjecture says that the short range maximum values tend to be slightly smaller than one would expect from a naive view of Selberg's CLT.%
\par
The second surprising aspect is the size of the tail of the random variable~\(X\). The maximum among independent choices from a Gaussian follow the \terminology{Gumbel distribution}, which decays like \(e^{-x}\), whereas the FHK conjecture posits decay like \(xe^{-x}\).%
\par
The conclusion is that \(N\) choices from \(\abs{Z_U(e^{i\theta})}\) are not independent. That is another aspect of the carrier wave, which we try to make more explicit in \hyperref[how_wide]{Subsection~{\xreffont\ref{how_wide}}}.%
\end{paragraphs}%
\begin{paragraphs}{Non-implications for extreme values.}{long_range-6}%
It is tempting to use the FHK conjecture \hyperref[FHKconjecture]{({\xreffont\ref{FHKconjecture}})} as a step toward a conjecture for the extreme values, by invoking \hyperref[eNprin]{Principle~{\xreffont\ref{eNprin}}} and using the conjectured tail \(x e^{-x}\) for \(X\). That would predict much larger values for the \(\zeta\)-function than suggested in \hyperlink{FGH}{[{\xreffont 49}]}, but that reasoning is flawed because of an implicit switching of limits. The limiting distribution of \(X_T\) as \(T\to\infty\) does decrease like \(x e^{-x}\), but the distribution of \(X_T\) has an additional factor \(e^{-x^2/\log\log T}\).  See~\hyperlink{ADH}{[{\xreffont 5}]}. That factor is significant when the sample size is large.%
\end{paragraphs}%
\end{subsectionptx}
\typeout{************************************************}
\typeout{Subsection 11.2 How wide is the carrier wave?}
\typeout{************************************************}
\begin{subsectionptx}{Subsection}{How wide is the carrier wave?}{}{How wide is the carrier wave?}{}{}{how_wide}
Our discussion of carrier waves has not indicated their characteristic scale. That is, over what span of zeros does the logarithm of the function typically change very little? Four sources of information about that question have been mentioned. We will phrase things in terms of \(M/\log T\), where \(M\) is the typical number of zeros over which the size of \(\log\abs{\zeta(\frac12 + i t)}\) does not change.%
\par
Montgomery's original speculations suggested that \(M\) might be as large as \(\exp(\delta_T {\log\log T})\), for some function \(\delta_T \to 0\). The material below suggests that \(\delta_T\) might need to decrease at least as quickly as \(\log\log\log T/\log\log T\).%
\par
Bombieri and Hejhal \hyperlink{BomHej}{[{\xreffont 24}]} show that \(M\) can be taken to be any fixed number.  That was sufficient for their application. It may be that a detailed examination of their proof allows \(M\) as large as \(\log\log(T)^\kappa\) for any \(\kappa \lt \frac14\).%
\par
We now show that the FHK conjecture \hyperref[FHKconjecture]{({\xreffont\ref{FHKconjecture}})} suggests \(M=\log\log T\).  We will invoke \hyperref[Jsamples]{({\xreffont\ref{Jsamples}})}. Suppose \(Y\) is the span over which the (logarithm of) the function typically changes very little, measured on the scale of the average zero spacing. Independent choices from \(\mathcal Z_U(e^{i\theta})\) must have \(\theta\) separated by \(Y/N\), to cover the circle with \(J=N/Y\) independent choices. Making that substitution in \hyperref[Jsamples]{({\xreffont\ref{Jsamples}})} gives main terms%
\begin{equation}
\log N - \frac{2 \log Y + \log\log N}{4}\text{.}\label{how_wide-5-10}
\end{equation}
Equating to \hyperref[FHKconjecture]{({\xreffont\ref{FHKconjecture}})} with \(C=\frac34\) and \(N=\log(T/2\pi)\) we find \(Y=\log\log(T/2\pi)\), as claimed.%
\par
Shifted moments can provide information about how quickly the size of the \(\zeta\)-function can change. For fixed \(a, b \gt 0\) consider%
\begin{equation}
I_+(a, b; T) := \int_2^T \zeta\bigl(\tfrac12 + \tfrac{a}{\log T} + it\bigr)
\overline{\zeta\bigl(\tfrac12 + \tfrac{a}{\log T} +  \tfrac{i b}{\log T} + it\bigr)}
dt\text{.}\label{Iplus}
\end{equation}
By a theorem of Ingham \hyperlink{Ing}{[{\xreffont 71}]},%
\begin{equation}
I_+(a, b; T) \sim \frac{1 - e^{-2a-i b}}{2a + i b} T\log T\text{.}\label{how_wide-6-5}
\end{equation}
If the size of \(\zeta(\frac12 + i t)\) were typically the same as \(\zeta(\frac12 +  ib/\log t + it)\), then the size of \(I_+(a, b; T)\) would not depend on \(b\).  But it does depend on \(b\), and in fact if \(b\) grows with \(T\) then the size of the integral changes.%
\par
Does this suggest that the scale of the carrier wave, \(M\), cannot grow with~\(T\)?  That is not a valid conclusion, because the carrier wave concerns the bulk of the distribution, and moments are sensitive to rare large values.  Indeed, points where the \(\zeta(\frac12 + i t)\) is of size at most \(\log\log(t)^A\), which happens \(100\)\% of the time, do not contribute to the main term of \(I_+(a, b; T)\) when \(a\) and \(b\) are bounded. Thus, a fixed set of moments, even if they are shifted moments, cannot tell us about the scale of the carrier wave.%
\end{subsectionptx}
\typeout{************************************************}
\typeout{Subsection 11.3 Small values of the wave, and valleys within valleys...}
\typeout{************************************************}
\begin{subsectionptx}{Subsection}{Small values of the wave, and valleys within valleys...}{}{Small values of the wave, and valleys within valleys...}{}{}{smallwave}
It is easy to lose sight of the fact that the carrier wave also concerns very small values, which occupy just as much space as the large values.  Modify \hyperref[Iplus]{({\xreffont\ref{Iplus}})}, but to be sensitive to small values:%
\begin{equation}
I_-(a, b; T) := \int_2^T \zeta\bigl(\tfrac12 + \tfrac{a}{\log T} + it\bigr)^{-1}
\overline{\zeta\bigl(\tfrac12 + \tfrac{a}{\log T}  + \tfrac{i b}{\log T} + it\bigr)^{-1}}
dt\text{.}\label{Iminus}
\end{equation}
By the ratios conjecture~\hyperlink{CFZ}{[{\xreffont 32}]},%
\begin{equation}
I_-(a, b; T) \sim  \frac{15}{\pi^2} \frac{1}{2a+ib} T\log T\text{.}\label{smallwave-2-5}
\end{equation}
For the analogous conjecture for an arbitrary primitive \(L\)-function \(L(s)\), replace the constant \(15/\pi^2\) by \(\sum \abs{\mu_L(n)}^2/n^2\), where \(L(s)^{-1} = \sum \mu_L(n) n^{-s}\).%
\par
The same arguments as for \(I_+\) apply here: the typical small values are not relevant to the size of \(I_-\), and the atypically small values do not persist over the span of many zeros. Indeed, it may seem initially surprising that the shapes of \(I_+\) and \(I_-\) are so similar. It is not clear what this suggests about the bias in the value distribution at low heights, which is skewed toward small values as illustrated in \hyperref[fig_nongaussian]{Figure~{\xreffont\ref{fig_nongaussian}}}.%
\par
It would be interesting to adapt the approach of FHK, or use some other method, to model the small values of \(\log\abs{\zeta(\frac12 + it)}\) on an interval.  The zeros prevent considering that question literally, so some modification would be needed.  Considering values at local maxima probably runs afoul of small maxima from occasional small zero gaps. Averaging over a small interval, a few widths of the average zero gaps, should capture the wave but may be difficult to handle analytically. Or perhaps moving away from the critical line, on a scale comparable to the average zero gap as in~\hyperref[Iminus]{({\xreffont\ref{Iminus}})}, is the right way to think about it.  See \hyperlink{BuFl}{[{\xreffont 28}]} for recent work in this direction.%
\par
The example carrier waves in this paper were created by generating a large random matrix and then finding all of its eigenvalues. That enabled us to exhibit behavior far beyond what can be computed for the \(\zeta\)-function, but it limits our ability to produce  very large examples.  To quote Fyodorov and Keating \hyperlink{FyKe}{[{\xreffont 56}]} (slightly out of context), the carrier wave should exhibit a  hierarchical structure of ``valleys within valleys within valleys''. It would be interesting to get a glimpse of this structure, which would require directly generating example carrier waves corresponding to much larger random matrices. We mention this again at the end of \hyperref[gramslawrevisited]{Subsection~{\xreffont\ref{gramslawrevisited}}}. It is not clear whether on such a large scale there would be qualitative differences between the \(L\)-function and the characteristic polynomial worlds.%
\end{subsectionptx}
\typeout{************************************************}
\typeout{Subsection 11.4 Log correlation and moments of moments}
\typeout{************************************************}
\begin{subsectionptx}{Subsection}{Log correlation and moments of moments}{}{Log correlation and moments of moments}{}{}{momentsofmoments}
This short section is primarily a pointer to work on ``moments of moments''.  It the random matrix world, this refers to expressions like%
\begin{equation}
\int_{A\in U(N)}
\biggl(
\int_0^{2\pi} |\Lambda_A(e^{i\theta})|^{2\beta} \,d\theta
\biggr)^k dA\text{,}\label{momentsofmoments-2-2}
\end{equation}
and in the \(L\)-function world%
\begin{equation}
\int_{T}^{2T}
\biggl(
\int_0^{1} |\zeta(\tfrac12 + i t + i h)|^{2\beta} \,dh
\biggr)^k dt\text{.}\label{momentsofmoments-2-4}
\end{equation}
\par
The log-correlation manifests itself via a ``freezing transition'' where the behavior of those expressions changes when \(k\beta^2 \gt 1\), and by extra cancellation when \(\beta = 1\) and \(k \lt 1\). See [\hyperlink{BK1}{{\xreffont 8}}, \hyperlink{BK2}{{\xreffont 9}}, \hyperlink{Cur}{{\xreffont 39}}, \hyperlink{Harper1}{{\xreffont 64}}].%
\end{subsectionptx}
\typeout{************************************************}
\typeout{Subsection 11.5 Gaussian multiplicative chaos}
\typeout{************************************************}
\begin{subsectionptx}{Subsection}{Gaussian multiplicative chaos}{}{Gaussian multiplicative chaos}{}{}{gmc}
We have seen that \(\log \zeta(\frac12 + i t)\) is log-correlated. And even better: by Selberg's CLT it is (to leading order) Gaussian, so we have a Gaussian log-correlated field. Its (suitably normalized) exponential is an example of \terminology{Gaussian multiplicative chaos}, GMC, a topic which has received considerable attention due to its connection to turbulence, mathematical finance, quantum gravity, and elsewhere. Here we briefly describe the work of Saksman and Webb; see [\hyperlink{SaksWebb}{{\xreffont 100}}, \hyperlink{SaksWebbSurv}{{\xreffont 101}}] for details.%
\par
We have been viewing \(\zeta(\frac12 + it)\) as random by choosing \(t \in [T, 2T]\) uniformly.  That approach treats the \emph{values} of the \(\zeta\)-function as random.  Now we want to generate a random \emph{function} from the \(\zeta\)-function.%
\par
Let \(\omega \in [0,1]\) be uniformly distributed, and consider the random function%
\begin{equation}
\mu_T(x) := \zeta(\tfrac12 + i x + i \omega T)
\ \ \ \ \ \text{for}\ \ \ \ \ \ 
x\in (0, 1)\text{.}\label{gmc-4-2}
\end{equation}
The goal is to understand the limiting statistics of \(\mu_T\) as \(T\to\infty\).%
\par
The main result of \hyperlink{SaksWebb}{[{\xreffont 100}]} is that as \(T\to\infty\) the random function \(\mu_T\) converges to the \terminology{randomized zeta function} \(\zeta_{\text{rand}}(\frac12 + i x)\), where%
\begin{equation}
\zeta_{\text{rand}}(s) = \prod_p (1 - p^{-s} e^{2 \pi i \theta_p})^{-1}\label{gmc-5-6}
\end{equation}
with the \(\theta_p\) i.i.d. uniform on \((0, 1)\). Note that \(\zeta_{\text{rand}}\) is a generalized function. The convergence of \(\mu_T(x)\) to \(\zeta_{\text{rand}}(\frac12 + i x)\) is convergence in law with respect to the strong topology of the Sobolev space \(W^{-\alpha, 2}(0, 1)\) for any \(\alpha > \frac12\). Furthermore, we have%
\begin{equation}
\zeta_{\text{rand}}(\tfrac12 + i x) = g(x) v(x)\label{gmc-5-14}
\end{equation}
where \(v\) is a Gaussian multiplicative chaos distribution and \(g\) is a random smooth function which almost surely has no zeros. (There are additional conditions on \(v\) and \(g\), see [\hyperlink{SaksWebb}{{\xreffont 100}}, \hyperlink{SaksWebbSurv}{{\xreffont 101}}].) A similar factorization holds for random characteristic polynomials, where the GMC factor is the same but the smooth part is slightly different. In both Selberg's CLT and the Keating-Snaith analogue for characteristic polynomials, the value distribution is Gaussian, but only to leading order. Similarly, \(\zeta_{\text{rand}}\) is a perturbation of a GMC object.%
\par
Since the randomized zeta function is a generalized function, and not an actual function or a measure, it is not clear how to visualize it or how to have it shed light on the ``valleys within valleys'' mentioned above.%
\end{subsectionptx}
\end{sectionptx}
\typeout{************************************************}
\typeout{Section 12 The claimed reasons to doubt RH: properties of \(Z(t)\)}
\typeout{************************************************}
\begin{sectionptx}{Section}{The claimed reasons to doubt RH: properties of \(Z(t)\)}{}{The claimed reasons to doubt RH: properties of \(Z(t)\)}{}{}{rhissues_Z}
\begin{introduction}{}%
~Several of the claimed reasons to doubt RH involve analytic properties of \(Z(t)\).  We address those reasons first, calling upon the Principles from the previous sections in this paper.%
\par
We have described how the analogy between the \(\zeta\)-function and the characteristic polynomials of random unitary matrices has provided insight to the behavior of the \(\zeta\)-function. We now call on a fact which has been implicit in our previous discussions and now will play a decisive role.%
\begin{quote}%
RH is true for characteristic polynomials of random unitary matrices.%
\end{quote}
That is, the zeros of unitary polynomials lie on the unit circle.%
\par
If we are confronted with information about the \(\zeta\)-function, and the analogous information also applies to characteristic polynomials, then what can we conclude?  Certainly we cannot conclude that RH may be false, because that conclusion would also apply to unitary polynomials, contradicting the fact that the analogue of RH is true. Similarly, every numerical calculation of the \(\zeta\)-function has found all zeros on the critical line, therefore any information arising from explicit computations cannot throw doubt on RH. We summarize those observations:%
\begin{principle}{Principle}{}{}{prin_minimal_unitary}%
Any fact which directly translates to a statement about unitary polynomials cannot be used as evidence against RH.%
\end{principle}
\begin{principle}{Principle}{}{}{prin_minimal_computation}%
Any fact arising from numerical computations of the \(\zeta\)-function, except for an actual counterexample, cannot be used as evidence against RH.%
\end{principle}
We will see that these principles, coupled with material from earlier sections of this paper, immediately refute some of the arguments against RH. However, it is important to keep in mind that those arguments against RH were reasonable at the time they were made: it is only subsequent discoveries which have revealed new ideas which allow these principles to be applied.%
\end{introduction}%
\typeout{************************************************}
\typeout{Subsection 12.1 Reason 1: Lehmer's phenomenon}
\typeout{************************************************}
\begin{subsectionptx}{Subsection}{Reason 1: Lehmer's phenomenon}{}{Reason 1: Lehmer's phenomenon}{}{}{rhissues_Z-3}
This reason for doubting RH comes from Section~2 of \hyperlink{Iv2}{[{\xreffont 73}]}.%
\par
Computing zeros has been a fundamental part of exploring the \(\zeta\)-function, going back to Riemann who calculated approximations to the first few zeros.  When D.H.~Lehmer computed the first \(15,000\) zeros in 1956, he found that \(\tilde{\gamma}_{6710} - \tilde{\gamma}_{6709} \lt 0.054\). Those zeros are unusually close, in a sense which can be made precise. Such zeros are now called a \terminology{Lehmer pair}.%
\par
The existence of Lehmer pairs is a claimed reason to doubt RH. The justification is that if a real holomorphic function \(f(t)\) has two real zeros very close together, then \(f(t) + \varepsilon\) will have a pair of complex zeros for some small~\(\varepsilon\). For example, near the Lehmer pair noted above, \(Z(t)-0.004\) has a pair of complex zeros. So, in a sense, the Lehmer pairs suggest:%
\begin{principle}{Principle}{}{}{barelytrue}%
If the Riemann Hypothesis is true, it is ``barely true''.%
\end{principle}
The claimed argument is:  Lehmer pairs indicate that RH is barely true, therefore we should be skeptical of~RH.%
\par
The above argument contradicts \hyperref[prin_minimal_unitary]{Principle~{\xreffont\ref{prin_minimal_unitary}}}. The gaps between neighboring zeros of random unitary polynomials can be arbitrarily small.  That is, arbitrarily close Lehmer pairs exist for random unitary polynomials: the normalized nearest-neighbor spacing is supported on all of \(\oointerval{0}{\infty}\). Furthermore, the calculations of Odlyzko~\hyperlink{Odl}{[{\xreffont 88}]} and subsequent computations indicate that the distribution of small gaps between zeros of the \(\zeta\)-function closely match the (suitably scaled) distribution of small gaps between random unitary eigenvalues. Thus, the data on Lehmer pairs support the connection between \(\zeta\)-zeros and the \(\GUE\) (or equivalently the \(\CUE\) because the connection only concerns leading-order behavior).  But by \hyperref[prin_minimal_unitary]{Principle~{\xreffont\ref{prin_minimal_unitary}}} the data provide no reason to doubt RH.%
\begin{paragraphs}{Other ways in which RH is barely true.}{rhissues_Z-3-8}%
There are other ways in which RH is ``barely true''. The function \(Z(t) + \delta\) has complex zeros for any nonzero (positive or negative) \(\delta\in \mathbb R\), and the reason has nothing to do with Lehmer pairs. Carrier waves (the small values, not the large values, see \hyperref[wave_intuition]{Subsection~{\xreffont\ref{wave_intuition}}}) cause \(Z(t)\) to stay small on intervals containing many zeros, without those zeros being close together.  Thus, for any \(\delta \gt 0\) there will exist intervals \(I\) containing arbitrarily many zeros such that \(\abs{Z(t)} \lt \delta\) on \(I\). Therefore \(Z(t)+\delta\) and \(Z(t)-\delta\) will have many positive local minima (respectively, negative local maxima) on \(I\), and thus will have many pairs of non-real zeros.%
\par
Here is yet another way.  RH is equivalent to \(\Lambda \le 0\), where \(\Lambda\) is the de Bruijn-Newman constant~\hyperlink{New}{[{\xreffont 85}]}. Rodgers and Tao~\hyperlink{RT}{[{\xreffont 94}]} showed that \(\Lambda \ge 0\), extended by Dobner~\hyperlink{Dob}{[{\xreffont 43}]} to all \(L\)-functions. So \(\Lambda = 0\) is the only possibility consistent with~RH.%
\par
Thus, RH is at best ``barely true'' in multiple ways. That is just a fact one has to live with. It does not tip the scales for or against~RH.%
\end{paragraphs}%
\begin{paragraphs}{``The Riemann Hypothesis is not an analysis problem''.}{rhissues_Z-3-9}%
That is a quote from Brian Conrey, although he suggests that others have expressed a similar sentiment. There are several ways to understand its meaning, one of which comes from  the fact that RH is at best barely true. There are many functions which are known to have all (or almost all) of their zeros on the real line or the unit circle.  Examples are sine and cosine, Bessel and other special functions, Eisenstein series, period polynomials, and many families of orthogonal polynomials.  What all those examples have in common is that the techniques come from analysis and not number theory, and asymptotically the zeros are close to equally spaced.  In other words, the analogue of RH is not barely true:  it remains true under a small perturbation. That is a feature of analysis techniques, and that is why those techniques are not going to prove RH.%
\par
We are not saying analysis techniques are not important or useful.  Indeed, some of the refutations of reasons to doubt RH are purely analysis arguments. But if there is a proof of RH, the deep facts in that proof will not come from analysis.%
\par
The above discussion suggests it is inappropriate to use terminology like ``the Riemann Hypothesis'' to refer to the fact that \(\sin\), \(\cos\), Bessel functions, and some orthogonal polynomials have only real zeros, or that period polynomials and some other orthogonal polynomials have all zeros on the unit circle. Such terminology overinflates the depth of those results.%
\end{paragraphs}%
\end{subsectionptx}
\typeout{************************************************}
\typeout{Subsection 12.2 Reason 2: Large values of \(Z(t)\)}
\typeout{************************************************}
\begin{subsectionptx}{Subsection}{Reason 2: Large values of \(Z(t)\)}{}{Reason 2: Large values of \(Z(t)\)}{}{}{rhissues_Z-4}
In Section 5 of \hyperlink{Iv2}{[{\xreffont 73}]} the claimed reason to doubt RH directly addresses carrier waves.  Specifically, it is noted that RH combined with conjectures about the maximum size of~\(Z(t)\) indicate that (this is quoted from \hyperlink{Iv2}{[{\xreffont 73}]}).%
\begin{quote}%
...the graph of \(Z(t)\) will consist of tightly packed spikes, which will be more and more condensed as \(t\) increases, with larger and larger oscillations. This I find hardly conceivable.%
\end{quote}
The author agrees that it is hard to believe, and the difficulty of believing it was specifically mentioned in the introduction to \hyperref[waves]{Section~{\xreffont\ref{waves}}}. But believe it we must, because that is how carrier waves work. Also, what the graph looks like is just a matter of perspective: a snapshot of the \(Z\)-function covering a range of 200 zeros will look like tightly packed spikes no matter the size of the function and whether or not RH is true.%
\par
Thus, this reason for doubting RH is not valid. Indeed, a significant portion of the zeros are known to be on the critical line, so the ``tightly packed spikes'' must exist whether or not RH is true. The appearance of the graph is merely a matter of perspective: you can either see a pleasantly undulating curve, or tightly packed spikes, depending on the choice of scales on the axes.%
\end{subsectionptx}
\typeout{************************************************}
\typeout{Subsection 12.3 Reason 3: Derivatives near a large maximum}
\typeout{************************************************}
\begin{subsectionptx}{Subsection}{Reason 3: Derivatives near a large maximum}{}{Reason 3: Derivatives near a large maximum}{}{}{blanc}
The approach of Blanc~\hyperlink{Blanc1}{[{\xreffont 18}]} is based on the following interesting formula~\hyperlink{Blanc2}{[{\xreffont 19}]}.  Suppose \(f\) is a smooth function on an open interval containing \([-a, a]\) for some \(a>0\), and suppose \(\mathbf x = \{x_0,\ldots,x_n\}\) are distinct zeros of \(f\) in that interval. Then%
\begin{align}
\sum_{k=1}^r \Psi_{2k-1}(\mathbf x, a)f^{(2k-1)}(a) -\mathstrut\amp
\sum_{k=1}^r \Psi_{2k-1}(\mathbf x, -a)f^{(2k-1)}(-a) \label{blanckey}\\
\amp \phantom{xxxx} = \int_{-a}^a \Psi_{2r-1}(\mathbf x, t) f^{(2r)}(t) \, dt\text{.}\notag
\end{align}
The exact form of \(\Psi_\ell\) is tangential to our discussion, but we include it for completeness:%
\begin{equation*}
\Psi_\ell(\mathbf x, t) = \frac{(4a)^\ell}{(\ell+1)!}
\sum_{k=0}^n \mu_k(\mathbf x)\left(
B_{\ell-1}\Bigl(\tfrac12 + \frac{x+x_k}{4a}\Bigr)
+
B_{\ell-1}\Bigl(\Bigl\{\frac{x-x_k}{4a}\Bigr\}\Bigr)
\right)
\end{equation*}
where%
\begin{equation*}
\mu_k(\mathbf x) = 2^{-n} \prod_{\substack{0\le j \le n\\ j\not= k}}
\biggl(\sin\Bigl(\pi\frac{x_k}{2a}\Bigr) - \sin\Bigl(\pi\frac{x_j}{2a}\Bigr)
\biggr) ^{-1}\text{.}
\end{equation*}
Here \(B_\ell\) is the Bernoulli polynomial and \(\{x\} = x -[x]\) is the fractional part of~\(x\).  The relevant fact about \(\mu_k\) is that \(\sum_k \mu_k = 0\).%
\par
Blanc's approach involves applying \hyperref[blanckey]{({\xreffont\ref{blanckey}})} in the context shown in \hyperref[fig_blanc]{Figure~{\xreffont\ref{fig_blanc}}}, where we have repurposed data from Bober and Hiary~\hyperlink{BobHia}{[{\xreffont 21}]} which previously appeared in \hyperref[fig_boberhiary]{Figure~{\xreffont\ref{fig_boberhiary}}}.%
\begin{figureptx}{Figure}{The function \(\log \abs{Z(t)}\) near a local extremum caused by a large zero gap, and \(S(t)\) in the same region.  The parameters \(T\) and~\(a\) are chosen so that \(Z'(T+a)\) is a (large) local extremum, \(Z'(T-a)\) is a (presumably small) local extremum, and \(S(T+a)\) and \(S(T-a)\) have opposite sign.}{fig_blanc}{}%
\begin{image}{0}{1}{0}{}%
\includegraphics[width=\linewidth]{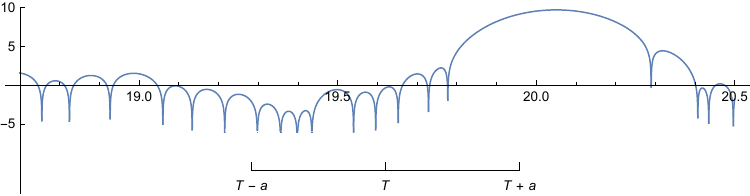}
\end{image}%
\begin{image}{0}{1}{0}{}%
\includegraphics[width=\linewidth]{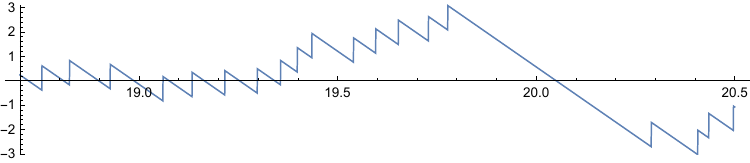}
\end{image}%
\tcblower
\end{figureptx}%
Blanc chooses \(T\) and \(a\) so that \(T+a\) lies in a large zero gap, \(\abs{Z(T + t)}\) has a very large maximum a bit to the right of \(a\), and \(T+a\) is a local extremum of \(Z'(T+a)\). Those choices are illustrated in \hyperref[fig_blanc]{Figure~{\xreffont\ref{fig_blanc}}}, although it is difficult to see an inflection point of \(Z(t)\) on the graph of \(\log\abs{Z(t)}\).  Note how the large zero gap causes the maximum of \(Z'(t)\) to shift toward the maximum of \(Z(t)\). Furthermore assume \(S(T+t)\) changes sign in \((-a,a)\). Blanc also chooses \(a\) so that \(Z'(T-a)\) is a local extremum, although that does not seem to be essential.%
\par
The setup above lets us describe Blanc's main idea.  With \(f(t)=Z(T+t)\), the quantity \(f^{(2k-1)}(a)\)  in \hyperref[blanckey]{({\xreffont\ref{blanckey}})} will be large for small \(k\).  For larger \(k\) there is no reason for \(f^{(2k-1)}(a)\) to be large, and for all \(k\) there is no reason for \(f^{(2k-1)}(-a)\) to be large. Thus, some of the quantities on the left side of \hyperref[blanckey]{({\xreffont\ref{blanckey}})} will be very large.  \hyperref[BHtable]{Table~{\xreffont\ref{BHtable}}} lists the quantities in \hyperref[blanckey]{({\xreffont\ref{blanckey}})} in the context of \hyperref[fig_blanc]{Figure~{\xreffont\ref{fig_blanc}}}. Similar to Table~1 of \hyperlink{Blanc1}{[{\xreffont 18}]}, we write%
\begin{gather}
\beta^\pm_{\ell} = (-1)^{(\ell-1)/2}\Psi_{\ell}(\mathbf x, T\pm a)\theta'(T)^{\ell}\label{betaeqn}\\
d^\pm_{\ell} = (-1)^{(\ell-1)/2}Z^{\ell}(T\pm a)/\theta'(T)^\ell\text{,}\label{deqn}
\end{gather}
so the left side of \hyperref[blanckey]{({\xreffont\ref{blanckey}})} is \(\sum \beta^+_{\ell} d^+_{\ell} - \sum \beta^-_{\ell} d^-_{\ell}\).%
\begin{tableptx}{Table}{\textbf{Quantities appearing in ({\xreffont\ref*{blanckey}}) using data from Bober and Hiary~[{\xreffont 21}] as shown in Figure~{\xreffont\ref*{fig_blanc}}}}{BHtable}{}%
\centering%
{\tabularfont%
\begin{tabular}{rrrrrrrr}
\multicolumn{1}{c}{{\bfseries{}\(\ell\)}}&\multicolumn{1}{c}{{\bfseries{}\(\beta^+_\ell\)}}&\multicolumn{1}{c}{{\bfseries{}\(d^+_\ell\)}}&\multicolumn{1}{c}{{\bfseries{}\(\beta^+_\ell d^+_\ell\)}}&\multicolumn{1}{c}{{\bfseries{}~}}&\multicolumn{1}{c}{{\bfseries{}\(\beta^-_\ell\)}}&\multicolumn{1}{c}{{\bfseries{}\(d^-_\ell\)}}&\multicolumn{1}{c}{{\bfseries{}\(\beta^-_\ell d^-_\ell\)}}\tabularnewline\hrulethin
1&0.000896&3209.249&2.877&~&-0.00561&0.24181&-0.001\tabularnewline[0pt]
3&0.004657&625.137&2.911&~&-0.01944&0.37736&-0.007\tabularnewline[0pt]
5&0.014213&182.535&2.549&~&-0.04274&0.52865&-0.022\tabularnewline[0pt]
7&0.033762&67.231&2.296&~&-0.07822&0.70713&-0.055\tabularnewline[0pt]
9&0.069663&28.657&1.996&~&-0.13156&0.91859&-0.120\tabularnewline[0pt]
11&0.132016&13.376&1.765&~&-0.21285&1.17131&-0.249
\end{tabular}
}%
\end{tableptx}%
We can obtain a contradiction if we can show that the left side of  \hyperref[blanckey]{({\xreffont\ref{blanckey}})} is large and the right side is small. For the data from \hyperref[fig_blanc]{Figure~{\xreffont\ref{fig_blanc}}}, the right side of \hyperref[blanckey]{({\xreffont\ref{blanckey}})} is approximately \(13.95852\), which is consistent with \hyperref[BHtable]{Table~{\xreffont\ref{BHtable}}}. That merely serves as a check on the numerical computations, because RH is true in the range of those data.%
\par
If we were only considering characteristic polynomials, then by \hyperref[prin_minimal_unitary]{Principle~{\xreffont\ref{prin_minimal_unitary}}} we cannot obtain such a contradiction. Therefore we need another ingredient which is specific to \(L\)-functions. This is provided by another interesting result of Blanc \hyperlink{Blanc1}{[{\xreffont 18}]}.%
\par
Recall that by \hyperref[eqn_thetat]{({\xreffont\ref{eqn_thetat}})} we have \(\theta'(t) \sim \frac12 \log \frac{t}{2\pi}\).%
\begin{theorem}{Theorem}{Blanc [{\xreffont 18}].}{}{thm_Z2Kbound}%
Suppose \(T\) is sufficiently large, \(\delta > \frac12\), and \(a = o(T/\theta'(T))\). If%
\begin{equation}
\delta \theta'(T)\le 2K \le 2\theta'(T)^2\label{thetabound}
\end{equation}
then%
\begin{equation}
\abs{Z^{(2 K)}(t)} \le 2 \zeta(\tfrac12 + \delta) \theta'(T)^{2K}\label{Z2kbound}
\end{equation}
for \(t\in [T-a, T+a]\).%
\end{theorem}
Blanc \hyperlink{Blanc1}{[{\xreffont 18}]} states the case \(\delta=\frac32\) and estimates \(2\zeta(2) \lt 4\). It might be worth investigating what choice of \(\delta\) is optimal for applications of \hyperref[thm_Z2Kbound]{Theorem~{\xreffont\ref{thm_Z2Kbound}}}.%
\par
By \hyperref[thm_Z2Kbound]{Theorem~{\xreffont\ref{thm_Z2Kbound}}} we can choose \(r\) so that the right side of \hyperref[blanckey]{({\xreffont\ref{blanckey}})} is small.  Therefore, if the setup described above makes the left side of  \hyperref[blanckey]{({\xreffont\ref{blanckey}})} large, we have obtained a contradiction.%
\par
The analogue of \hyperref[thm_Z2Kbound]{Theorem~{\xreffont\ref{thm_Z2Kbound}}} is not true for unitary polynomials, so we explore the implications of \hyperref[Z2kbound]{({\xreffont\ref{Z2kbound}})} in some detail before returning to Blanc's argument.%
\par
Note that \hyperref[Z2kbound]{({\xreffont\ref{Z2kbound}})} is not true for any fixed \(K\) as \(T\to\infty\), because that would contradict known \(\Omega\)-results and conjectures for \(Z^{(n)}(t)\). We now explain why, more generally, the bound \hyperref[Z2kbound]{({\xreffont\ref{Z2kbound}})} is not true for \(K\) far outside the range \hyperref[thetabound]{({\xreffont\ref{thetabound}})}. We do this by exploiting properties of repeated differentiation.%
\begin{paragraphs}{Everyone wants to be cosine.}{blanc-16}%
Differentiating a function like \(Z(t)\) can be viewed as an operation on its set of zeros, replacing the zeros of \(Z(t)\) by the zeros of \(Z'(t)\).  The relevant fact here is that differentiation causes the zeros to become more evenly spaced~\hyperlink{FR}{[{\xreffont 52}]}\hyperlink{Ki}{[{\xreffont 78}]}, with large gaps becoming smaller and small gaps becoming larger. Thus, differentiation damps the irregularities in the size of the function: very large maxima (measured relative to the size of nearby maxima) become less extreme. Under repeated differentiation the zero gaps are becoming more regular, eventually approaching equality.  For sufficiently large~\(k\) the \(k\)th derivative is locally approximately a function with equally spaced zeros. In other words, for \(t\) in any fixed interval, if \(k\) is sufficiently large, then%
\begin{equation}
Z^{(k)}(t) \approx A_k \cos(B_k t + C_k)\label{approxCosine}
\end{equation}
throughout that interval. Berry \hyperlink{Berry}{[{\xreffont 13}]} termed this phenomenon \terminology{cosine is a universal attractor}.%
\par
If the number of derivatives is too small for the zeros to be approximately equally spaced, then the variation in the zero spacing will prevent the bound \hyperref[Z2kbound]{({\xreffont\ref{Z2kbound}})} from holding.  So \hyperref[Z2kbound]{({\xreffont\ref{Z2kbound}})} will not hold if \(K\) is much smaller than the lower bound in~\hyperref[thetabound]{({\xreffont\ref{thetabound}})}.%
\par
For a moderate number of derivatives, the parameter \(B_k\) in \hyperref[approxCosine]{({\xreffont\ref{approxCosine}})} directly relates to the local density of zeros. Since the local density of zeros of \(Z(t)\) varies slowly, when the \(k\)th derivative initially approaches cosine we have  \(B_k \approx \theta'(T)\), as we previously saw in~\hyperref[bparameter]{({\xreffont\ref{bparameter}})}. That fact, and the chain rule, explains why powers of \(\theta'(T)\) appear on the right side of \hyperref[Z2kbound]{({\xreffont\ref{Z2kbound}})}, and the bound is optimal~\hyperlink{Blanc3}{[{\xreffont 20}]}.%
\par
Finally, if we take ``too many'' derivatives then the local density of zeros begins increasing, so \(B_k\) starts to increase and therefore \hyperref[Z2kbound]{({\xreffont\ref{Z2kbound}})} will not hold if \(K\) is much larger than the upper bound in~\hyperref[thetabound]{({\xreffont\ref{thetabound}})}.%
\end{paragraphs}%
\begin{paragraphs}{Some unitary polynomials approach cosine more slowly.}{blanc-17}%
We have already alluded to the fact that \hyperref[thm_Z2Kbound]{Theorem~{\xreffont\ref{thm_Z2Kbound}}} does not hold for all unitary polynomials.  More precisely, the analogue of the bound \hyperref[Z2kbound]{({\xreffont\ref{Z2kbound}})} does hold, but the number of derivatives needed to realize that bound can be much larger than the range \hyperref[thetabound]{({\xreffont\ref{thetabound}})}. To see this, it is sufficient to consider self-reciprocal polynomials in \terminology{polar form}:%
\begin{equation}
\mathcal Z(z) = z^{\frac{N}{2}} + a_{N-1}z^{\frac{N-1}{2}} +
\cdots
+ \overline{a}_{N-1}z^{-\frac{N-1}{2}} + z^{-\frac{N}{2}}\label{selfreciprocalpoly}
\end{equation}
with differentiation operator \(D = z \frac{d}{dz}\). Note also that the analogue of \(\theta'(T)\) is \(N/2\).%
\par
We see that%
\begin{align}
D^{2 K} \mathcal Z(z) =\mathstrut\amp \left(\frac{N}{2}\right)^{2K} z^{\frac{N}{2}} 
+ \left(\frac{N-1}{2}\right)^{2K} a_{N-1}z^{\frac{N-1}{2}}
+\notag\\
\amp \phantom{xxxx} 
\cdots
+  \left(\frac{N-1}{2}\right)^{2K} \overline{a}_{N-1}z^{-\frac{N-1}{2}} 
+ \left(\frac{N}{2}\right)^{2K} z^{-\frac{N}{2}}\notag\\
=\mathstrut\amp \left(\frac{N}{2}\right)^{2K} z^{-\frac{N}{2}} 
\biggl(z^N + (1-N^{-1})^{2K} a_{N-1} z^{N-1} +\notag\\
\amp \phantom{xxxx} 
\cdots
+ (1-2 N^{-1})^{2K} \overline{a}_{N-2} z^{2}
+ (1-N^{-1})^{2K} \overline{a}_{N-1} z + 1 \biggr)\notag\\
=\mathstrut\amp \left(\frac{N}{2}\right)^{2K} z^{-\frac{N}{2}}
\left(z^N + 1 + o(1)\right)\label{D2kpoly}
\end{align}
as \(K\to\infty\). By \hyperref[D2kpoly]{({\xreffont\ref{D2kpoly}})} the roots of \(D^{2 K} \mathcal Z(z)\) are approaching the \(N\)th roots of~\(1\), so in particular they are approaching equal spacing. It was not even necessary to assume the original self-reciprocal polynomial had all its zeros on the unit circle. Yet the analogue of \hyperref[thm_Z2Kbound]{Theorem~{\xreffont\ref{thm_Z2Kbound}}} fails without further restrictions on the polynomials because the required number of derivatives may be much larger than the analogue of the lower bound in \hyperref[thetabound]{({\xreffont\ref{thetabound}})}. One can see this from \hyperref[D2kpoly]{({\xreffont\ref{D2kpoly}})} because \((1-N^{-1})^{2 K}\) is not small if \(K=O(N)\).%
\par
For unitary polynomials the analogue of \hyperref[thm_Z2Kbound]{Theorem~{\xreffont\ref{thm_Z2Kbound}}} has counterexamples similar to the failure of the analogue of the Lindelöf hypothesis for some unitary polynomials. Consider \(z^{-N/2}(z+1)^N\), with differentiation operator \(D = z \frac{d}{dz}\). If we take \(\kappa N\) \(D\)-derivatives, with \(\kappa \lt 1\), the polynomial will still have a high multiplicity zero at \(z=-1\).  That will force the other zeros to be spread out, and so the function will be large.  Even if we take \(N\) derivatives, so all the zeros are simple, it will still take many more derivatives for the zeros to become approximately equally spaced. In particular, by considering only the \(N-\sqrt{N}\)th coefficient of \((z+1)^N\), we see that \(D^N z^{-N/2}(z+1)^N\) is large at \(z=1\).%
\end{paragraphs}%
\par\medskip
From the above discussion we see that the analogue of \hyperref[thm_Z2Kbound]{Theorem~{\xreffont\ref{thm_Z2Kbound}}} does not hold for characteristic polynomials, so Blanc's argument, which we analyze next, does not fall victim to \hyperref[prin_minimal_unitary]{Principle~{\xreffont\ref{prin_minimal_unitary}}}.%
\begin{paragraphs}{Analysis of Blanc's argument.}{blanc-19}%
Summarizing the ideas so far: we want to assume a very large gap between zeros, causing a very large isolated local maximum of \(Z(t)\).  With the choices of \(T\) and \(a\) as in \hyperref[fig_blanc]{Figure~{\xreffont\ref{fig_blanc}}}, the quantity \(Z^{(2k-1)}(a)\) in \hyperref[blanckey]{({\xreffont\ref{blanckey}})} will be very large for small~\(k\). This might lead to a contradiction if the parameter \(r\) allows \hyperref[thm_Z2Kbound]{Theorem~{\xreffont\ref{thm_Z2Kbound}}} to apply, because then the right side of \hyperref[blanckey]{({\xreffont\ref{blanckey}})} will be small.  If the left side of \hyperref[blanckey]{({\xreffont\ref{blanckey}})} is large and the right side is small, then we have obtained a contradiction.%
\par
Given a large gap between zeros, and choosing \(r\) so that \hyperref[thm_Z2Kbound]{Theorem~{\xreffont\ref{thm_Z2Kbound}}} applies, how can having \(Z^{(2k-1)}(a)\) be very large for small \(k\) fail to lead to a contradiction?  There are two ways. %
\begin{enumerate}
\item\hypertarget{Psismall}{}\(\Psi_{1}(\mathbf{x},a)\) might be very small, so actually the left side of \hyperref[blanckey]{({\xreffont\ref{blanckey}})} is small. (For example, see \hyperref[BHtable]{Table~{\xreffont\ref{BHtable}}})%
\item\hypertarget{otherterms}{}Some other terms on the left side might be large and of the opposite sign, causing cancellation.%
\end{enumerate}
If the above possibilities do not apply (so we do obtain a contradiction), what has been contradicted?  Have we contradicted RH, or have we contradicted something else?  We will argue that such a contradiction, if it happens, has nothing to do with RH but rather it comes from assuming an excessively large gap between zeros.  In other words, the theorems of Blanc may actually contain hidden information which might be used to provide an improved upper bound on gaps between zeros of the \(\zeta\)-function.%
\par
Blanc addresses \hyperlink{otherterms}{Item~{\xreffont 2}} above by noting that the setup in \hyperref[fig_blanc]{Figure~{\xreffont\ref{fig_blanc}}} specifically avoids the likelihood that other terms are large. Figure~2 in~\hyperlink{Blanc1}{[{\xreffont 18}]} provides good evidence for the intuition that \(\beta_\ell^+\) should be small for large \(\ell\). %
\par
To investigate the size of \(\Psi_{\ell}(\mathbf{x},a)\), we consider a scenario similar to \hyperref[fig_blanc]{Figure~{\xreffont\ref{fig_blanc}}} at height \(10^{100}\), which Figure~2 in \hyperlink{Blanc1}{[{\xreffont 18}]} suggests is sufficient to test the method.  Since \(\log 10^{100} \approx 230.26\), we can use degree \(230\) unitary polynomials.  Given a large zero gap, we can place the other zeros in their most likely location as described in \hyperref[nearlarge]{Subsection~{\xreffont\ref{nearlarge}}}. An example of this, with a gap 6 times the average in a degree \(74\) polynomial, is shown in \hyperref[fig_RMT74]{Figure~{\xreffont\ref{fig_RMT74}}}.  Since everything is explicit, we can evaluate every quantity in \hyperref[blanckey]{({\xreffont\ref{blanckey}})}. \hyperref[blancfakedata]{Table~{\xreffont\ref{blancfakedata}}} lists the terms in \hyperref[betaeqn]{({\xreffont\ref{betaeqn}})} and \hyperref[deqn]{({\xreffont\ref{deqn}})} for degree \(230\) polynomials with a large gap of \(5\), \(10\) and \(15\) times the average spacing, with all other zeros close to their most likely configuration.%
\begin{tableptx}{Table}{\textbf{}}{blancfakedata}{}%
\centering%
{\tabularfont%
\begin{tabular}{rrrrrrrrr}
\multicolumn{1}{c}{{\bfseries{}\(\ell\)}}&\multicolumn{1}{c}{{\bfseries{}\(\beta^+_\ell(5)\)}}&\multicolumn{1}{c}{{\bfseries{}\(d^+_\ell(5)\)}}&\multicolumn{1}{c}{{\bfseries{}}}&\multicolumn{1}{c}{{\bfseries{}\(\beta^+_\ell(10)\)}}&\multicolumn{1}{c}{{\bfseries{}\(d^+_\ell(10)\)}}&\multicolumn{1}{c}{{\bfseries{}~}}&\multicolumn{1}{c}{{\bfseries{}\(\beta^+_\ell(15)\)}}&\multicolumn{1}{c}{{\bfseries{}\(d^+_\ell(15)\)}}\tabularnewline\hrulethin
1&0.00099&363.5&&\(2.5\times10^{-6}\)&446988.0&~&\(1.9\times10^{-8}\)&\(4.9\times10^{8}\)\tabularnewline[0pt]
3&0.00329&95.2&&0.000016&57835.6&~&\(1.9\times10^{-7}\)&\(4.2\times10^7\)\tabularnewline[0pt]
5&0.00654&37.7&&0.000057&11436.1&~&\(1.0\times10^{-6}\)&\(5.6\times10^{6}\)
\end{tabular}
}%
\end{tableptx}%
The main takeaway from \hyperref[blancfakedata]{Table~{\xreffont\ref{blancfakedata}}} is that \(\beta_\ell^+\) is small for small \(\ell\) when \(d_1^+\) is large. Thus, we must question the possibility of obtaining a contradiction from \hyperref[blanckey]{({\xreffont\ref{blanckey}})}. Specifically, we suggest that there is no evidence to support Blanc's assertion (bottom of first column on page 5 of \hyperlink{Blanc1}{[{\xreffont 18}]}) ``the sum of the first \(\beta_{2\ell-1}^+ d_{2\ell-1}^+\) is probably large''.%
\par
The comments above do not address a key element of Blanc's argument, based on \hyperref[Z2kbound]{({\xreffont\ref{Z2kbound}})}, for which the analogous result for unitary polynomials is not true. Translating \hyperref[Z2kbound]{({\xreffont\ref{Z2kbound}})} to the realm of self-reciprocal polynomials gives the statement: in the context of the polar form \hyperref[selfreciprocalpoly]{({\xreffont\ref{selfreciprocalpoly}})}, with \(\mathcal Z^{2K}(z) := D^{2K} \mathcal Z(z)\),%
\begin{equation}
\text{if} \ \  K\ge \frac38 N
\ \ \text{then} \ \ 
\max_{\abs{z}=1} \abs{\mathcal Z^{2K}(z)} \le 3.29 \Bigl(\frac{N}{2}\Bigr)^{2K}\text{.}\label{polyderivbound}
\end{equation}
That bound on \(\abs{\mathcal Z^{2K}(z)}\) holds if \(K\) is sufficiently large, but not in general for \(K\ll N\). In particular, for the degree 230 polynomial with a zero gap 15 times the average used in \hyperref[blancfakedata]{Table~{\xreffont\ref{blancfakedata}}}, one must take \(197\) derivatives in order to achieve the estimate in \hyperref[polyderivbound]{({\xreffont\ref{polyderivbound}})}. The polynomial with gap 10 times the average requires \(123\) derivatives. This suggests to the author that the ingredients in Blanc's argument do not throw doubt on RH, but instead point to a possible new method for bounding the size of large gaps between zeros of the \(\zeta\)-function. We propose:%
\begin{problem}{Problem}{}{blanc-19-9}%
Suppose \(\mathcal Z\) is a polar unitary polynomial of degree \(N\) with a zero gap of \(\kappa\) times the average.  For \(C \gt 2\), determine a lower bound on \(\eta = \eta(\kappa, C)\) such that if \(2K \lt \eta N\) then%
\begin{equation*}
\max_{\abs{z}=1} \abs{\mathcal Z^{(2 K)}(z)} > C \Bigl(\frac{N}{2}\Bigr)^{2 K}\text{.}
\end{equation*}
\end{problem}
The techniques used to solve that problem might translate into a method for improving the upper bound on gaps between zeros of the \(\zeta\)-function.%
\end{paragraphs}%
\begin{paragraphs}{Further reflections on characteristic polynomials.}{blanc-20}%
There is an apparent discrepancy in our discussion of large gaps between zeros.  On the one hand we have described how random matrix theory provides a precise conjecture on the maximum gap size.  On the other hand, we have described properties of \(Z(t)\) which suggest constraints on large gaps between zeros, and those constraints do not apply to characteristic polynomials.%
\par
These apparent discrepancies might be resolved by examining in more detail how characteristic polynomials are used to model \(Z(t)\). By \hyperref[eNprin]{Principle~{\xreffont\ref{eNprin}}}, to model the largest gaps one chooses around \(e^N\) independent random matrices in \(U(N)\), where \(N\approx \log T\). Based on the enormous success of using RMT to shed light on \(L\)-functions, it it plausible that the analogue of \hyperref[thm_Z2Kbound]{Theorem~{\xreffont\ref{thm_Z2Kbound}}} holds almost surely (in the probabilistic sense) for \(e^N\) Haar-random matrices chosen from \(U(N)\).%
\end{paragraphs}%
\end{subsectionptx}
\end{sectionptx}
\typeout{************************************************}
\typeout{Section 13 Claimed reasons to doubt RH: other properties}
\typeout{************************************************}
\begin{sectionptx}{Section}{Claimed reasons to doubt RH: other properties}{}{Claimed reasons to doubt RH: other properties}{}{}{rhissues_other}
\begin{introduction}{}%
The final two claims we consider involve averages of high powers of \(Z(t)\), and properties of Dirichlet series which are not \(L\)-functions.%
\end{introduction}%
\typeout{************************************************}
\typeout{Subsection 13.1 Reason 4: Error terms in moment formulas, and maybe the Lindelöf Hypothesis is also false}
\typeout{************************************************}
\begin{subsectionptx}{Subsection}{Reason 4: Error terms in moment formulas, and maybe the Lindelöf Hypothesis is also false}{}{Reason 4: Error terms in moment formulas, and maybe the Lindelöf Hypothesis is also false}{}{}{rhissues_other-3}
The argument in Section 4 of \hyperlink{Iv2}{[{\xreffont 73}]} is based on the meager knowledge we have about moments of the \(\zeta\)-function:%
\begin{align*}
\int_0^T Z(t)^2 \, dt  =\mathstrut\amp T\log T +
(2\gamma - 1 - \log 2\pi) T   + O(T^{E_1 + \varepsilon})\\
\int_0^T Z(t)^4 \, dt  =\mathstrut\amp T P_2(\log T)  + O(T^{E_2 + \varepsilon})
\end{align*}
where \(P_2\) is a certain degree~4 polynomial. The conjectured optimal error terms are \(E_1 = \frac14\) and \(E_2 = \frac12\).%
\par
For higher moments, the conjecture \hyperlink{CG1}{[{\xreffont 33}]} is:%
\begin{align}
\int_0^T Z(t)^{2k}\, dt =\mathstrut\amp T P_k(\log T) + O(T^{E_k + \varepsilon})\label{twokthmoment}\\
\sim\mathstrut\amp  g_k a_k T\log^{k^2} T\text{,}\label{leadingtwokthmoment}
\end{align}
as \(T\to\infty\), for some \(E_k \lt 1\), where \(P_k(x)\) is a polynomial of degree \(k^2\) with leading coefficient \(g_k a_k \), where \(a_k\) is given explicitly, and \(g_k\) is an integer. That conjecture is open for \(k\ge 3\).%
\begin{paragraphs}{Why moments?}{rhissues_other-3-4}%
Moments of the \(\zeta\)-function were introduced by Hardy and Littlewood~\hyperlink{HL}{[{\xreffont 63}]} as an approach toward the \hyperref[conjLH]{Lindelöf Hypothesis~{\xreffont\ref{conjLH}}}.%
\par
RH implies LH, so it is natural to attack LH as an ``easier'' problem. LH is equivalent to%
\begin{equation}
\int_0^T Z(t)^{2 k} \, dt \ll_k T^{1 + \varepsilon}\label{twok1plusepsilon}
\end{equation}
for all integer \(k \gt 0\), which is why Hardy and Littlewood investigated moments.%
\par
Bounding such moments proved to be difficult: 100 years after Hardy and Littlewood we only know \hyperref[twok1plusepsilon]{({\xreffont\ref{twok1plusepsilon}})} for \(k=1\) and \(2\).  But instead of becoming a dead-end in terms of proving~LH, moments became a topic of independent interest. Other ``families'' of \(L\)-functions were introduced, with various types of moments. The focus shifted to finding precise formulas for the main terms in moments, not just on estimating the order of magnitude. Random matrix theory played a major role in understanding the leading order behavior and the structure of those moments, although ultimately number-theoretic heuristics provided the most precise conjectures. See \hyperlink{KS1}{[{\xreffont 76}]}\hyperlink{CF}{[{\xreffont 30}]}\hyperlink{KS2}{[{\xreffont 77}]}\hyperlink{CFKRS}{[{\xreffont 31}]}\hyperlink{DGH}{[{\xreffont 42}]}.%
\end{paragraphs}%
\begin{paragraphs}{Analysis of the argument.}{rhissues_other-3-5}%
The claimed argument against RH arises from the error term in the \(2k\)th moment.  Based on the conjectured best possible values \(E_1 = \frac14\) and \(E_2 = \frac12\) Ivić suggests \hyperlink{Iv1}{[{\xreffont 72}]}\hyperlink{Iv2}{[{\xreffont 73}]} that the best possible error term in the \(2k\)th moment is \(O(T^{k/4 + \varepsilon})\). There is slightly more to Ivić's reasoning, such as the apparent inability of available techniques to provide bounds on high moments, but the main input is the conjectured values for \(E_1\) and~\(E_2\).%
\par
If the error term \hyperref[twokthmoment]{({\xreffont\ref{twokthmoment}})} truly is of size \(T^{k/4+\varepsilon}\), meaning that \hyperref[twokthmoment]{({\xreffont\ref{twokthmoment}})} is false with \({k}/{4}\) replaced by a smaller number, then \hyperref[twok1plusepsilon]{({\xreffont\ref{twok1plusepsilon}})} is also false.  Therefore LH is false, so RH is false. That is the essence of that argument against RH.%
\par
Basing an argument on such limited data for small \(k\) is problematic, and the conjecture of \(E_k = k/4\) is not based on any underlying principles. Suppose someone instead conjectured \(E_k = \frac34 - 2^{-k}\) for \(k\ge 1\).  That also is a simple formula which fits the limited data, and it has the added benefit of implying~LH.  But without any good evidence, that conjecture should also not be taken seriously. Thus, this proposed reason to doubt RH is not based on a mathematical foundation, but on naive pattern matching.%
\par
Before the reader criticizes the author for being unkind, let us note that this author fully recognizes the folly of speculating on the size of error terms, and indeed is guilty of similarly unwise speculation.  In a foundational paper on moments of \(L\)-functions~\hyperlink{CFKRS}{[{\xreffont 31}]} this author and coauthors conjecture that in great generality the error term in any conjectured higher moment is~\(E_k = \frac12\). The conjecture of Ivić is as pessimistic as possible, and the conjecture in~\hyperlink{CFKRS}{[{\xreffont 31}]} is as naively optimistic as possible. Unfortunately, that optimistic error term is not always correct. Specifically, in some cases~\hyperlink{Dia}{[{\xreffont 41}]}\hyperlink{DGH}{[{\xreffont 42}]} there is a secondary main term of size \(X^{\frac34}\). It may even be that similar secondary main terms are common. The lesson is that speculating on the error term in moment conjectures is difficult, and it is folly to draw conclusions from such  speculation.%
\end{paragraphs}%
\begin{paragraphs}{The structure of moments.}{rhissues_other-3-6}%
Immediately after making a poorly founded conjecture about error terms, Ivić goes on \hyperlink{Iv1}{[{\xreffont 72}]}\hyperlink{Iv2}{[{\xreffont 73}]} to make the perceptive speculation: \begin{quote}%
the shape of the asymptotic formula...changes when \(k=4\).%
\end{quote}
 That speculation has been borne out by subsequent work on the structure of moments.%
\par
The \(2k\)th moment \hyperref[twokthmoment]{({\xreffont\ref{twokthmoment}})} is only known for \(k=0\), \(1\), and \(2\).  Conjectures for larger moments are a fairly recent development. One reason those moments are difficult is that the natural way to proceed is via an ``approximate functional equation'' for \(Z(t)^k\), which is a sum of two Dirichlet polynomials of length \(t^{k/2}\). (The Riemann-Siegel formula \hyperref[RiemannSiegel]{({\xreffont\ref{RiemannSiegel}})} can be viewed as an example: there are \(\approx t^{\frac12}\) terms in the sum.) When used to evaluate \(\int_0^T Z(t)^{2k}\, dt\) the approximate functional equation will have length~\(T^{k/2}\). The difficulty is the lack of tools to handle Dirichlet polynomials where the length of the polynomial is greater than the length of the integral.%
\par
Conrey and Ghosh~\hyperlink{CG2}{[{\xreffont 34}]} developed a heuristic approach which gave a conjecture when \(k=3\), and Conrey and Gonek~\hyperlink{CoGo}{[{\xreffont 36}]} pushed the method to its limit to make a conjecture when \(k=4\).  That truly is the limit of the method, because when applied to the case \(k=5\) those methods produce an answer which is negative! Thus, Ivić was correct: some new phenomenon appears, beginning at \(k=4\).  But that new phenomenon has more implications for that main term than for the error term.%
\end{paragraphs}%
\end{subsectionptx}
\typeout{************************************************}
\typeout{Subsection 13.2 Reason 5: The Deuring-Heilbronn phenomenon}
\typeout{************************************************}
\begin{subsectionptx}{Subsection}{Reason 5: The Deuring-Heilbronn phenomenon}{}{Reason 5: The Deuring-Heilbronn phenomenon}{}{}{rhissues_other-4}
The Riemann \(\zeta\)-function is the simplest of an infinite family of functions known as ``\(L\)-functions''. The next simplest examples are the Dirichlet \(L\)-functions, \(L(s,\chi)\), where \(\chi\) is a primitive Dirichlet character. RH for all \(L(s,\chi)\) is known as the \terminology{Generalized Riemann Hypothesis}.%
\par
Informally, by \terminology{\(L\)-function} we mean a Dirichlet series with a functional equation and an Euler product of a particular form.  To give a more precise definition, there are two primary approaches. The \emph{axiomatic approach} considers \(L\)-functions as analytic functions with certain properties. Within the axiomatic approach at one extreme we have the Selberg class~\hyperlink{Sel}{[{\xreffont 104}]}\hyperlink{CG3}{[{\xreffont 35}]}, which has a minimal set of general axioms.  At the other extreme are the tempered balanced analytic \(L\)-functions~\hyperlink{FPRS}{[{\xreffont 51}]} which have a large number of precise axioms.%
\par
The \emph{structural approach} describes \(L\)-functions as arising from arithmetic or automorphic objects.  For the perspective in this paper, the important fact is that conjecturally all approaches describe the same set of functions, and \emph{all such \(L\)-functions have an analogue of the Riemann Hypothesis}.  The collection of all such Riemann Hypotheses is called the \terminology{grand Riemann Hypothesis}.%
\par
Under any of the above definitions, each \(L\)-function has a functional equation analogous to \hyperref[zeta_fe]{({\xreffont\ref{zeta_fe}})}, a \(\xi_L\)-function analogous to \hyperref[eqn_xi_def]{({\xreffont\ref{eqn_xi_def}})}, and a \(Z_L\)-function analogous to \hyperref[eqn_Zdef]{({\xreffont\ref{eqn_Zdef}})}. Each \(Z_L\)-function is a real valued function which wiggles with the same type of randomness as~\(Z(t)\). It also has carrier waves as described in the earlier parts of this paper.%
\begin{paragraphs}{Linear combinations.}{rhissues_other-4-6}%
The Deuring-Heilbronn phenomenon concerns functions which are not \(L\)-functions, but rather are linear combinations of \(L\)-functions.  As we will see in \hyperref[prin_DH]{Principle~{\xreffont\ref{prin_DH}}}, a nontrivial linear combination of \(L\)-functions never satisfies~RH, in the worst possible way.%
\par
Consider two \(L\)-functions which are being added, and focus on a small interval of the real line (say, where one expects to see a few dozen zeros of each \(Z\)-function). If one of the \(Z\)-functions happens to be much larger than the other over that entire interval, then the sum looks just like the larger \(Z\)-function. In particular, if the larger \(Z\)-function satisfies RH, then their sum satisfies RH on that interval. That is the typical situation and this was the motivation for developing the idea of carrier waves: the logarithm of each \(Z\)-function has a Gaussian distribution, and because of the carrier waves its size does not change too rapidly.  Thus, 100\% of the time one of them is much larger than the other and stays larger over a span of many zeros. Under some mild additional hypotheses, that is how Bombieri and Hejhal~\hyperlink{BomHej}{[{\xreffont 24}]} proved that if two \(L\)-functions have the same functional equation and each satisfies RH, then their sum has 100\% of its zeros on the critical line.%
\par
But, 100\% is not everything. As the two \(Z\)-functions take turns being the largest, occasionally they are about the same size over the span of a few zeros.  When that happens, and we are adding two different wiggling functions of similar size, not all the zeros of the sum will be real. In other words:%
\begin{principle}{Principle}{}{}{prin_lin_comb}%
A nontrivial linear combination of \(L\)-functions will not satisfy the Riemann Hypothesis.%
\end{principle}
\end{paragraphs}%
\begin{paragraphs}{Naturally occurring examples.}{rhissues_other-4-7}%
The above discussion makes it seem obvious that a nontrivial linear combination of \(L\)-functions will never satisfy RH. But the situation was not always so clear. Let \(\chi_{5,2}\) be the Dirichlet character \(\mathstrut\bmod 5\) with \(\chi(2)=i\), and define~\(\theta\) by \(\tan \theta = (\sqrt{10 - 2 \sqrt{5}} - 2)/(\sqrt{5} - 1)\). The \terminology{Deuring-Heilbronn function} is given by%
\begin{equation}
F_{DH}(s) = \tfrac12 \sec \theta \left(
e^{-i\theta} L(s, \chi_{5,2}) + e^{i\theta} L(s, \overline{\chi}_{5,2})\right)\label{DHfunction}
\end{equation}
which satisfies the functional equation%
\begin{equation}
\xi_{DH}(s) := 5^{\frac12 s}\pi^{-\frac12 s}\Gamma(\tfrac12 s + \tfrac12) F_{DH}(s)
=\xi_{DH}(1-s)\text{.}\label{DHfe}
\end{equation}
Except for the sign, that is the same functional equation satisfied by \(L(s,\chi_{5,2})\). However \(F_{DH}(s)\) does not satisfy RH:  its first zero off the critical line is near \(0.8085 + 85.699 i\). Furthermore~\hyperlink{T}{[{\xreffont 110}]} it has infinitely many zeros in~\(\sigma > 1\), although it does not seem that an explicit example has ever been computed.%
\par
We have two Dirichlet series, \(F_{DH}(s)\) and \(L(s,\chi_{5,2})\), which satisfy similar functional equations.  We know that one of them does not satisfy RH, so why should the other? That is the essence of the argument in Section~3 of \hyperlink{Iv2}{[{\xreffont 73}]}: why should we believe RH for one function when we know it is false for another, and those functions (quoting from~\hyperlink{Iv2}{[{\xreffont 73}]}) ``share many common properties''?%
\end{paragraphs}%
\begin{paragraphs}{The Euler product.}{rhissues_other-4-8}%
The standard answer to the above question (as noted in~\hyperlink{Iv2}{[{\xreffont 73}]}) is that the functions are fundamentally different because \(L(s,\chi_{5,2})\) is an \(L\)-function, but \(F_{DH}(s)\) is not \textemdash{} because it does not have an Euler product. It is an exercise to check that every axiom in the Selberg class~\hyperlink{Sel}{[{\xreffont 104}]} is essential: omit any one of them and there will be functions which do not satisfy~RH. The Euler product condition is not more important than the other axioms, but it is important, and without it one has no expectation of~RH.%
\par
The Deuring-Heilbronn function is one of many examples constructed by taking linear combinations of \(L\)-functions.  It is an exercise to show that a nontrivial linear combination of Euler products is not an Euler product.  Thus, any such linear combination is not expected to satisfy~RH, even if the constituent functions do satisfy~RH.  This explains why examples like \(F_{DH}(s)\) do not cast doubt on RH. However, there is a bit more to be said if we consider the locations of the zeros of those functions.%
\end{paragraphs}%
\begin{paragraphs}{Zeros outside the critical strip.}{rhissues_other-4-9}%
The Deuring-Heilbronn function failed RH in the worst possible way: it has zeros in the region where the Dirichlet series converges absolutely. In fact it has \(\gg T\) zeros in \(\sigma \gt 1\) up to height~\(T\), but that is not any extra information because Dirichlet series are almost periodic functions (of \(t\)) in \(\sigma \gt 1\): so once there is one zero in that region, there must be \(\gg T\) zeros in that region. Those zeros in \(\sigma \gt 1\) are a general phenomenon, so we obtain a refinement of \hyperref[prin_lin_comb]{Principle~{\xreffont\ref{prin_lin_comb}}}:%
\begin{principle}{Principle}{}{}{prin_DH}%
A nontrivial linear combination of \(L\)-functions will have infinitely many zeros in~\(\sigma > 1\).%
\end{principle}
\hyperref[prin_DH]{Principle~{\xreffont\ref{prin_DH}}} is a theorem for Dirichlet \(L\)-functions~\hyperlink{Saias}{[{\xreffont 99}]}, for automorphic \(L\)-functions~\hyperlink{Booker}{[{\xreffont 26}]}, and in an axiomatic setting~\hyperlink{Rig}{[{\xreffont 95}]}. Those results are interesting because one has a vector space of Dirichlet series, all of which satisfy the same functional equation. Conjecturally only the functions in the span of individual basis elements (i.e.\@~the actual \(L\)-functions) satisfy RH, yet every other element of the space fails RH in the worst possible way. This highlights the key role of the Euler product.%
\par
Why exactly is the Euler product so important?  Maybe it isn't. In terms of the analytic properties of the \(L\)-function, the most obvious consequence is that the Euler product prevents the \(L\)-function from having any zeros in~\(\sigma \gt 1\). So, is it the Euler product that matters, or the lack of zeros in~\(\sigma \gt 1\)?%
\end{paragraphs}%
\end{subsectionptx}
\end{sectionptx}
\typeout{************************************************}
\typeout{Section 14 Discussion of the Mistaken Notions}
\typeout{************************************************}
\begin{sectionptx}{Section}{Discussion of the Mistaken Notions}{}{Discussion of the Mistaken Notions}{}{}{misleadingrevisited}
\begin{introduction}{}%
We revisit the topics in \hyperref[misleading]{Section~{\xreffont\ref{misleading}}} in the context of the Principles. Specifically, we re-examine the mistaken notions that the largest values of the \(\zeta\)-function arise from particularly large zero gaps, that those large values occur when many of  the initial terms in the Riemann-Siegel formula have the same sign, that large values are places to look for possible counterexamples to RH, and that the Gram points are special.%
\end{introduction}%
\typeout{************************************************}
\typeout{Subsection 14.1 Large values and large gaps}
\typeout{************************************************}
\begin{subsectionptx}{Subsection}{Large values and large gaps}{}{Large values and large gaps}{}{}{misleadingrevisited-3}
\hyperref[mistakegap]{Mistaken Notion~{\xreffont\ref{mistakegap}}}, equating the largest values of the \(\zeta\)-function with the largest gaps, is the core idea for three of the Notions.  As has already been adequately addressed by the discussion of carrier waves, see \hyperref[prin_carrier]{Principle~{\xreffont\ref{prin_carrier}}}, that Notion does not accurately describe the typical large values, nor the truly largest values.%
\par
However, \hyperref[mistakegap]{Mistaken Notion~{\xreffont\ref{mistakegap}}} \emph{does} describe the behavior in the range of current computations, and sometimes a wrong idea can lead to the right answer. Kotnik~\hyperlink{kotnik}{[{\xreffont 74}]} sought computational evidence that the known \(\Omega\)-result \hyperref[eqn_Omega]{({\xreffont\ref{eqn_Omega}})} was not the true order of growth. The method was to find the sequence of largest values of \(\abs{\zeta(\frac12 + it)}\) for small~\(t\).  Those computations did make it appear plausible that the true order of growth is faster than the \(\Omega\)-result. Since those large values arise from large gaps, the computations were not revealing the true order of growth.  But by \hyperref[prin_large_Z_gaps]{Principle~{\xreffont\ref{prin_large_Z_gaps}}}, the maxima in the large gaps should be \(\gg e^{c\sqrt{\log t}}\), which is larger than the current best \(\Omega\)-result.%
\end{subsectionptx}
\typeout{************************************************}
\typeout{Subsection 14.2 Finding large values}
\typeout{************************************************}
\begin{subsectionptx}{Subsection}{Finding large values}{}{Finding large values}{}{}{misleadingrevisited-4}
\hyperref[mistakeRS]{Mistaken Notion~{\xreffont\ref{mistakeRS}}} asserts that the largest values of the \(\zeta\)-function occur when many initial terms in the Riemann-Siegel formula have the same sign.  The Notion makes sense: the first few terms are the largest, so if you want the entire sum to be large, it is efficient to focus on the initial terms. That argument ignores the fact that there are \(\gg t^{\frac12}\) terms of size \(\gg t^{-\frac14}\). Even a small bias in such a large number of terms could contribute much more than the initial terms.  Indeed, one of the arguments for the conjectured maximum size of the \(\zeta\)-function is based on a random model for the tail of that sum, which is larger than can be obtained from a strong bias in the initial terms.%
\par
Nevertheless, \hyperref[mistakeRS]{Mistaken Notion~{\xreffont\ref{mistakeRS}}} is how the largest known values of the \(\zeta\)-function have been found~\hyperlink{BobHia}{[{\xreffont 21}]}\hyperlink{Tih2}{[{\xreffont 109}]}.  It is possible to take that idea too far.  Consider%
\begin{equation}
F(t, g(t)) = 2\sum_{n \lt g(t)}
n^{-\frac12} \cos(\theta(t) - t\log n)\text{.}\label{misleadingrevisited-4-3-6}
\end{equation}
The Riemann-Siegel formula has \(g(t) = \sqrt{t/2\pi}\). One might expect that \(F(t, t^\varepsilon)\) is a good approximation for the large values of \(Z(t)\) because (assuming the \hyperref[conjLH]{Lindelöf Hypothesis}) a similarly short sum is a good approximation to \(\zeta(s)\) in~\(\sigma \gt \frac12\).  Kotnik~\hyperlink{kotnik}{[{\xreffont 74}]} considered \(F(t, t^\delta)\) for \(\delta = \frac13\) and \(\delta=\frac14\). Those functions are faster to evaluate than the Riemann-Siegel formula, and this sped up the computations by more quickly locating regions where \(Z(t)\) was more likely to be large.%
\par
Tihanyi \hyperlink{Tih1}{[{\xreffont 108}]} took it a step too far by considering \(F(t, \log t/2\pi)\).  Since \(F(t, g(t)) \ll \sqrt{\mathstrut g(t)}\), Tihanyi's approximation has no hope of capturing the largest values.  Yet, using \(F(t, \log t/2\pi)\) as a filter (the same strategy as Kotnik~\hyperlink{kotnik}{[{\xreffont 74}]}) turned out to be useful for finding many points where \(Z(t)\) was larger than \(1000\), and in subsequent work~\hyperlink{Tih2}{[{\xreffont 109}]}, found a value of \(Z(t)\) even larger than that found by Bober and Hiary~\hyperlink{BobHia}{[{\xreffont 21}]}.  So, a Notion may be wrong or misleading, yet still serve as the seed for a new discovery.%
\par
Although it is out of reach to compute in regions  where the carrier wave, and not the local zero gaps, provide the main contribution to the size of \(Z(t)\), it would still be of theoretical interest to find heuristics for locating truly large values.%
\end{subsectionptx}
\typeout{************************************************}
\typeout{Subsection 14.3 Trying to refute RH}
\typeout{************************************************}
\begin{subsectionptx}{Subsection}{Trying to refute RH}{}{Trying to refute RH}{}{}{howtorefute}
If RH were not true, how would one go about searching for a pair of zeros off the line? \hyperref[mistakeRH]{Mistaken Notion~{\xreffont\ref{mistakeRH}}} suggests that a large gap might ``push some zeros off the line''. That is contrary to the idea behind \hyperref[prin_minimal_unitary]{Principle~{\xreffont\ref{prin_minimal_unitary}}}: random unitary polynomials can have large zero gaps \textemdash{} conjecturally with the same size and frequency as the \(\zeta\)-function \textemdash{} which influence the nearby gaps but do not destroy the fact that the polynomial is unitary.%
\par
We offer two more arguments against \hyperref[mistakeRH]{Mistaken Notion~{\xreffont\ref{mistakeRH}}}.  The first is similar to an idea used in the discussion of carrier waves in \hyperref[waves]{Section~{\xreffont\ref{waves}}}: run the argument in reverse.  Suppose there is a pair of zeros off the line. Would one expect there to be a particularly large gap nearby?  Surely not. That pair off the line, if they were not too far from the line, would only cause a slightly larger than average gap in zeros on the critical line, and furthermore the function would be small (not large) within that zero gap.%
\par
The second argument concerns the fact that no plausible suggestion has been made for how RH would fail if it were false.  Maybe the zeros in \(\sigma \gt \frac12\), if they exist, all lie on \(\sigma = \frac34\)?  Is that a particularly absurd suggestion?  Zeros that far from the line have a very small influence on the behavior of \(Z(t)\) \textemdash{} in particular there is no associated large value or large zero gap.  It is perfectly reasonable to seek fame and fortune by trying to computationally disprove RH.  But searching for zeros off the line should be accompanied by, and tailored to, a model for how the failure will occur. A scenario in which RH fails, and those failures are typically associated to a large zero gap, is arguably even less plausible than having  all non-critical zeros on \(\sigma = \frac34\).%
\par
The carrier wave frequently causes the \(Z\)-function to stay very small over a wide span of zeros.  In that region a slight perturbation would cause many zeros to fall off the critical line. Maybe attempts to computationally disprove RH should begin by looking for regions where the initial terms in the Riemann-Siegel formula mostly cancel to nothing? %
\end{subsectionptx}
\typeout{************************************************}
\typeout{Subsection 14.4 Gram's law, as it was originally formulated}
\typeout{************************************************}
\begin{subsectionptx}{Subsection}{Gram's law, as it was originally formulated}{}{Gram's law, as it was originally formulated}{}{}{gramslawrevisited}
\hyperref[misleadinggramslaw]{Mistaken Notion~{\xreffont\ref{misleadinggramslaw}}} contains two statements: that Gram points are special, and that Gram's law is helpful for understanding properties of the zeros.  First we address Gram's law.  For reasons why Gram points are not special, see~\hyperref[grampointsrevisited]{Subsection~{\xreffont\ref{grampointsrevisited}}}.%
\par
Gram's law, as usually stated, misses a key point of Gram's original observation. What Gram noted is that the zeros and the Gram points ``separate each other''. In other words, in the range computed by Gram not only does the \(n\)th Gram interval \((g_n, g_{n+1})\) contain exactly one zero, that zero is the \(n+2\)nd zero. (The \(+2\) arises because the definition of Gram point has \(g_0 \approx 17.84\).) We propose:%
\begin{heuristic}{Reformulation}{The Precise Gram's Law.}{}{gramslawrevisited-4}%
The \terminology{precise Gram's law holds} for the \terminology{Gram interval} \((g_m, g_{m+1})\) if that interval contains the \(m+2\)nd zero of the \(Z\)-function and no other zeros.%
\end{heuristic}
That version of Gram's law also follows from \hyperref[pre_grams_law]{Principle~{\xreffont\ref{pre_grams_law}}}.%
\par
Gram's observation was based on the first \(15\) zeros, but note that he suggested the interlacing would not continue to hold indefinitely. In fact, Gram's law (either version) fails after the \(126\)th zero, \(\gamma_{126} \approx 279.2\),  because there is a large zero gap and the \(127\)th zero falls in the \(126\)th Gram interval. The Gram interval \((g_{3357},g_{3358})\) contains exactly one zero, but the precise Gram's law fails because the previous interval contains two zeros.  That is the first time there is a difference between the two versions of the law.  There are \(16\) differences among the first \(10,000\) zeros.%
\par
That the precise Gram's law holds initially can be seen either as a consequence of \(S(t)\) being small, or that the first term in the Riemann-Siegel formula is providing a good approximation for the zeros.  A similarity between those two reasons is that the sign and the \(\Gamma\)-factor in the functional equation give rise both to the first term in the Riemann-Siegel formula, and also to the main part of the counting function of the zeros. The Dirichlet series gives rise to the other terms, and also to \(S(t)\).%
\par
However, we suggest that \(S(t)\) provides a more fundamental insight into the behavior of the \(\zeta\)-function and its zeros. A large height, the Riemann-Siegel formula needs more terms, but it is not immediately obvious that this causes the precise Gram's law to fail most of the time. On the other hand, at large height \(\abs{S(t)}\) typically is big, which implies that the precise Gram's law holds rarely:%
\begin{observation}{Observation}{}{gramslawrevisited-9}%
A necessary (but not sufficient) condition for the precise Gram's law to hold on a Gram interval \((g_n, g_{n+1})\) is that \(\abs{S^+(g_n)} \le 1\).%
\end{observation}
The frequency of that occurrence at height \(t\) can be estimated using  Selberg's CLT. It is clear that in the limit it occurs 0\% of the time. The slow growth of \(S(t)\) makes this effect difficult to observe:  even for characteristic polynomials of random matrices in \(U(1000)\), as in the examples in \hyperref[rmtwaves]{Section~{\xreffont\ref{rmtwaves}}}, the precise Gram's law still holds more than 40\% of the time.%
\par
The fact that Gram's law (the non-precise version) holds around 66\% of the time should be viewed as an aspect of the random matrix model of zeros of the \(\zeta\)-function: the normalized nearest neighbor spacing of zeros\slash{}eigenvalues, see the leftmost plot in \hyperref[nn3]{Figure~{\xreffont\ref{nn3}}}, is concentrated around~\(1\). So it is reasonably likely that \emph{any} interval of normalized length~\(1\) contains exactly one zero\slash{}eigenvalue (although at low height a Gram interval is more likely than a random interval, see \hyperlink{HH}{[{\xreffont 62}]} or apply \hyperref[pre_grams_law]{Principle~{\xreffont\ref{pre_grams_law}}}).%
\par
The precise formulation of Gram's law describes a rigidity which only holds initially.%
\begin{principle}{Principle}{The sloshing water model.}{}{sloshing_water_model}%
At larger heights the zeros \terminology{slosh back-and-forth}.  That term is meant to evoke what happens when one tries to walk while carrying a wide but shallow pan of water. Large groups of zeros are shifted to the left or the right of their expected location, even though locally all sets of zeros behave similarly.%
\end{principle}
Because of the sloshing, \(S(t)\) is large (in absolute value) most of the time.  Long runs of zeros miss their intended Gram intervals: many in a row falling to the right, and later, many in a row falling to the left. (But locally the zeros are still obeying random matrix statistics, so a Gram interval (or any other interval of that length) will contain exactly one zero more than 66\% of the time.) The \(Z\)-function is wiggling with great amplitude for a while, and then wiggling with small amplitude.  The size of the carrier wave is independent of whether the nearby zeros arrive before or after their Gram interval, that is, independent of the size or the sign of \(S(t)\). All of these phenomena are manifestations of the density wave, a wave which cannot exist if the precise Gram's law were true too often.%
\begin{problem}{Problem}{}{understandwaves}%
Develop a random model for the density wave, equivalently the carrier wave, preferably as a theorem in the case of the \(\CUE\).%
\end{problem}
A good solution to that problem will allow generating examples from systems far larger than possible if one must first produce a large random matrix and then find its eigenvalues.%
\par
A solution to that problem must grapple with the issue that the definition of ``density wave'' depends on the range over which the density is being measured. A solution to that problem, along the lines of directly creating a density wave as in \hyperref[fig_density1000]{Figure~{\xreffont\ref{fig_density1000}}}, could be a step toward the next level of understanding the behavior of the \(\zeta\)-function a large height.  By that we mean:  at the first level we recognize that the local distribution of zeros is random; that randomness is now understood.  This is occurs in both the \(L\)-function and the characteristic polynomial worlds.  The next level is carrier\slash{}density waves.  This is beyond the reach of the \(L\)-function world, and observable but not understood in the random matrix world. Once carrier waves are understood, as in \hyperref[understandwaves]{Problem~{\xreffont\ref{understandwaves}}}, then it is natural to ask whether those waves are sitting in a yet larger structure. The analogy looks like%
\begin{equation}
\text{local zero spacing}:\text{carrier wave}\ ::\ \text{carrier wave}:\text{?????}\text{.}\label{gramslawrevisited-17-7}
\end{equation}
If one could directly produce density\slash{}carrier waves, then perhaps those can be generated without needing to first directly produce an enormous random matrix and then compute its eigenvalues. This may allow one to exhibit interesting (self similar?) structures, as mentioned in \hyperref[smallwave]{Subsection~{\xreffont\ref{smallwave}}}.%
\end{subsectionptx}
\typeout{************************************************}
\typeout{Subsection 14.5 Gram points are of transient interest}
\typeout{************************************************}
\begin{subsectionptx}{Subsection}{Gram points are of transient interest}{}{Gram points are of transient interest}{}{}{grampointsrevisited}
A recent paper of Shanker \hyperlink{shanker}{[{\xreffont 105}]} used computational data for the \(\zeta\)-function to make two conjectures about Gram intervals, and to suggest that ``Gram points have interesting properties which distinguish them from random points on the critical line''.  We examine that work in the context of the Principles, reaching a different conclusion.%
\par
The conjectures in \hyperlink{shanker}{[{\xreffont 105}]} are:%
\begin{enumerate}
\item{}If \(\{g_n\}\) are the Gram points, then the distribution of \(Z(g_{2n})\) is the negative of the distribution of \(Z(g_{2n+1})\).%
\item{}The properties of the \(Z\)-function on a set of consecutive Gram points is the same if that set is reversed.%
\par
For example, the probability that the number of zeros between consecutive Gram points is \(3, 1, 4, 1, 5, 9\) is the same as the probability that the number of zeros is \(9, 5, 1, 4, 1, 3\).%
\end{enumerate}
We will see that both conjectures are true for random unitary matrices, therefore both conjectures are immediate consequences of the Keating-Snaith Law. %
\par
The second part of \hyperlink{shanker}{[{\xreffont 105}]} uses two sets of data to suggest that Gram points are special.  Say that a Gram point \(g_n\) is \terminology{good} if \((-1)^n Z(g_n) \gt 0\) and \terminology{bad} otherwise. \hyperref[shankertab1]{Table~{\xreffont\ref{shankertab1}}} contains data from Table~2 in  \hyperlink{shanker}{[{\xreffont 105}]}.%
\begin{tableptx}{Table}{\textbf{The proportion of good and bad Gram intervals containing \(0, 1, 2\), or \(3\) zeros, and the combined proportion ignoring whether the interval is good or bad. Data from \(10^7\) zeros near \(t=10^{28}\).  The bottom row is the limiting prediction from the GUE Hypothesis (Principle~{\xreffont\ref*{guehypothesis}}).}}{shankertab1}{}%
\centering%
{\tabularfont%
\begin{tabular}{rrrrrr}
\multicolumn{1}{c}{{\bfseries{}}}&\multicolumn{1}{c}{{\bfseries{}\(m=0\)}}&\multicolumn{1}{c}{{\bfseries{}\(1\)}}&\multicolumn{1}{c}{{\bfseries{}\(2\)}}&\multicolumn{1}{c}{{\bfseries{}\(3\)}}&\multicolumn{1}{c}{{\bfseries{}size}}\tabularnewline\hrulethin
good&0.1097&0.7829&0.1071&0.00016&7373998\tabularnewline[0pt]
bad&0.3083&0.3838&0.3006&0.00709&2626002\tabularnewline[0pt]
all&0.1619&0.6781&0.1579&0.00198&10000000\tabularnewline[0pt]
GUE&0.1702&0.6614&0.1664&0.00186&
\end{tabular}
}%
\end{tableptx}%
The second set of data concerns consecutive Gram intervals which contain either \(2,1,\ldots,1,0\) or \(0,1,\ldots,1,2\) zeros.  Those \terminology{Gram blocks} are called \terminology{Type~I} and \terminology{Type~II}, respectively. As mentioned above, Type~I and Type~II, of the same length, occur with equal probability.  However, what is considered in \hyperlink{shanker}{[{\xreffont 105}]} are intervals which are shifted from Gram intervals by \(k\delta\) where \(\delta\) is the length of a Gram interval and \(k \in \{-0.2, -0.1, 0, 0.1, 0.2\}\). \hyperref[shankertab2]{Table~{\xreffont\ref{shankertab2}}} contains part of Table~5 from~\hyperlink{shanker}{[{\xreffont 105}]}.%
\begin{tableptx}{Table}{\textbf{The ratio Type~II\slash{}Type~I intervals for displaced Gram points, for \(10^7\) zeros near \(t=10^{28}\).}}{shankertab2}{}%
\centering%
{\tabularfont%
\begin{tabular}{rrrrrr}
\multicolumn{1}{c}{{\bfseries{}Length}}&\multicolumn{1}{c}{{\bfseries{}\(-0.2\delta\)}}&\multicolumn{1}{c}{{\bfseries{}\(-0.1\delta\)}}&\multicolumn{1}{c}{{\bfseries{}\(0\)}}&\multicolumn{1}{c}{{\bfseries{}\(0.1\delta\)}}&\multicolumn{1}{c}{{\bfseries{}\(0.2\delta\)}}\tabularnewline\hrulethin
\multicolumn{1}{c}{2}&2.268&1.504&0.999&0.663&0.441\tabularnewline[0pt]
\multicolumn{1}{c}{3}&3.624&1.896&0.998&0.526&0.274\tabularnewline[0pt]
\multicolumn{1}{c}{4}&5.588&2.367&1.001&0.426&0.178\tabularnewline[0pt]
\multicolumn{1}{c}{5}&8.849&2.923&1.011&0.343&0.115\tabularnewline[0pt]
\multicolumn{1}{c}{6}&14.373&3.728&1.008&0.266&0.070
\end{tabular}
}%
\end{tableptx}%
The third conjecture in \hyperlink{shanker}{[{\xreffont 105}]} is that in each row of \hyperref[shankertab2]{Table~{\xreffont\ref{shankertab2}}}, the product of the entries for \(k\delta\) and \(-k\delta\) equals~\(1\). For characteristic polynomials, that conjecture is immediate from the formula for Haar measure; see below for a more complete explanation.%
\par
The numbers \hyperref[shankertab1]{Table~{\xreffont\ref{shankertab1}}} were used to suggest that, while the behavior of all Gram points seems close to the limiting GUE prediction, the difference between the good and bad intervals casts doubt on the random matrix prediction. The numbers in \hyperref[shankertab2]{Table~{\xreffont\ref{shankertab2}}} were used to suggest that the Gram points are special, distinguishing them from typical points having the same spacing. We will reach different conclusions by invoking the Keating-Snaith Law and performing similar experiments with random unitary polynomials.%
\begin{paragraphs}{Gram point for unitary polynomials.}{grampointsrevisited-10}%
Recall that Gram points for the \(\zeta\)-function are the extrema of the first term in the Riemann-Siegel formula~\hyperref[RiemannSiegel]{({\xreffont\ref{RiemannSiegel}})}. As described in \hyperlink{HH}{[{\xreffont 62}]} we define the \terminology{Gram points} for a unitary polynomial in an analogous way.  Recall that \(\mathcal Z_A(\theta)\), defined in \hyperref[RMTZ]{({\xreffont\ref{RMTZ}})}, is real for \(\theta\in \R\).  Therefore there exists \(\theta_A\) such that%
\begin{equation}
\mathcal Z_A(\theta) = 2 \cos\Bigl(\frac{N}{2} (\theta - \theta_A)\Bigr)
+ \sum_{0\le n \lt N/2} a_n \cos(n (\theta - \theta_A))\text{,}\label{RSrmt}
\end{equation}
with \(a_n\in \R\) and where for convenience we have assumed \(N\) is even. The leading term is analogous to the leading term \(2\cos(\theta(t))\) in the Riemann-Siegel formula, so we define the \terminology{Gram points} for the matrix \(A\) to be the points where the leading term in \hyperref[RSrmt]{({\xreffont\ref{RSrmt}})} has a maximum or minimum: \((\theta_A + 2\pi n/N)_{n=0,1,2,\ldots}\). Note that if \(\det A = 1\) then \(\theta_A=0\). See \hyperlink{HH}{[{\xreffont 62}]} for a more detailed discussion.%
\par
Now it is straightforward to explain the first two conjectures in \hyperlink{shanker}{[{\xreffont 105}]}. The first conjecture follows from \hyperref[hardyZpm]{Principle~{\xreffont\ref{hardyZpm}}}. The first two  conjectures are immediate from the Keating-Snaith Law and \hyperref[cbetaesymmetries]{Lemma~{\xreffont\ref{cbetaesymmetries}}}, using the above definition of Gram point.%
\par
The third conjecture, concerning the statistics for a sequence of Gram intervals shifted by \(k\delta\) and \(-k\delta\) follows from \hyperref[cbetaesymmetries]{Lemma~{\xreffont\ref{cbetaesymmetries}}} because any statistic for a sequence of Gram intervals shifted by \(k\delta\) is formally the same as the reverse sequence shifted by \(-k\delta\).%
\par
We now address the larger issue of interpreting the data in \hyperref[shankertab1]{Table~{\xreffont\ref{shankertab1}}} and \hyperref[shankertab2]{Table~{\xreffont\ref{shankertab2}}}.%
\par
By the Keating-Snaith Law, we match the size of the matrices using the rule \(N=\log(t/2\pi)\). In the context of the data from  \hyperlink{shanker}{[{\xreffont 105}]}, this means that the appropriate matrix size is \(N=62\). \hyperref[goodbadrmt]{Table~{\xreffont\ref{goodbadrmt}}} and \hyperref[deltagramrmt]{Table~{\xreffont\ref{deltagramrmt}}} show the results from \(10^7\) random eigenvalues from matrices in \(U(N)\) for \(N=62\), \(250\), and \(1000\), analogous to \hyperref[shankertab1]{Table~{\xreffont\ref{shankertab1}}} and \hyperref[shankertab2]{Table~{\xreffont\ref{shankertab2}}}, respectively%
\begin{tableptx}{Table}{\textbf{The proportion of good and bad Gram intervals containing \(0, 1, 2\), or \(3\) zeros, and the combined proportion ignoring whether the interval is good or bad. Data from \(10^7\) eigenvalues of Haar-random matrices in \(U(N)\) for \(N=62\), \(250\), and \(1000\). }}{goodbadrmt}{}%
\centering%
{\tabularfont%
\begin{tabular}{rrrrrrr}
\multicolumn{1}{c}{{\bfseries{}\(N\)}}&\multicolumn{1}{c}{{\bfseries{}}}&\multicolumn{1}{c}{{\bfseries{}\(m=0\)}}&\multicolumn{1}{c}{{\bfseries{}\(1\)}}&\multicolumn{1}{c}{{\bfseries{}\(2\)}}&\multicolumn{1}{c}{{\bfseries{}\(3\)}}&\multicolumn{1}{c}{{\bfseries{}size}}\tabularnewline\hrulethin
\multicolumn{1}{c}{62}&good&0.1111&0.7802&0.1083&0.00021&7295392\tabularnewline[0pt]
&bad&0.2996&0.4009&0.2924&0.00687&2704608\tabularnewline[0pt]
&all&0.1621&0.6776&0.1581&0.00201&10000000\tabularnewline\hrulethin
\multicolumn{1}{c}{250}&good&0.1269&0.7484&0.1241&0.00050&6621923\tabularnewline[0pt]
&bad&0.2488&0.5031&0.2433&0.00461&3378077\tabularnewline[0pt]
&all&0.1681&0.6655&0.1643&0.00189&10000000\tabularnewline\hrulethin
\multicolumn{1}{c}{1000}&good&0.1377&0.7262&0.1351&0.00081&6147041\tabularnewline[0pt]
&bad&0.2205&0.5610&0.2148&0.00356&3852959\tabularnewline[0pt]
&all&0.1696&0.6625&0.1658&0.00187&10000000
\end{tabular}
}%
\end{tableptx}%
\begin{tableptx}{Table}{\textbf{The ratio Type~II\slash{}Type~I intervals for displaced Gram points, among \(10^7\) eigenvalues of Haar-random matrices in \(U(N)\) for \(N=62\), \(250\), and \(1000\).}}{deltagramrmt}{}%
\centering%
{\tabularfont%
\begin{tabular}{rrrrrrrr}
\multicolumn{1}{c}{{\bfseries{}\(N\)}}&\multicolumn{1}{c}{{\bfseries{}~~}}&\multicolumn{1}{c}{{\bfseries{}Length}}&\multicolumn{1}{c}{{\bfseries{}\(-0.2\delta\)}}&\multicolumn{1}{c}{{\bfseries{}\(-0.1\delta\)}}&\multicolumn{1}{c}{{\bfseries{}\(0\)}}&\multicolumn{1}{c}{{\bfseries{}\(0.1\delta\)}}&\multicolumn{1}{c}{{\bfseries{}\(0.2\delta\)}}\tabularnewline\hrulethin
\multicolumn{1}{c}{62}&~~&\multicolumn{1}{c}{2}&2.245&1.498&1.001&0.670&0.448\tabularnewline[0pt]
&~~&\multicolumn{1}{c}{3}&3.499&1.876&0.998&0.534&0.288\tabularnewline[0pt]
&~~&\multicolumn{1}{c}{4}&5.048&2.257&1.003&0.444&0.198\tabularnewline[0pt]
&~~&\multicolumn{1}{c}{5}&7.372&2.732&1.011&0.372&0.138\tabularnewline[0pt]
&~~&\multicolumn{1}{c}{6}&9.756&3.155&0.994&0.318&0.100\tabularnewline\hrulethin
\multicolumn{1}{c}{250}&~~&\multicolumn{1}{c}{2}&1.730&1.322&1.002&0.762&0.579\tabularnewline[0pt]
&~~&\multicolumn{1}{c}{3}&2.216&1.500&1.005&0.667&0.454\tabularnewline[0pt]
&~~&\multicolumn{1}{c}{4}&2.673&1.638&0.998&0.611&0.373\tabularnewline[0pt]
&~~&\multicolumn{1}{c}{5}&3.221&1.812&1.011&0.560&0.318\tabularnewline[0pt]
&~~&\multicolumn{1}{c}{6}&3.683&1.987&1.018&0.524&0.273\tabularnewline\hrulethin
\multicolumn{1}{c}{1000}&~~&\multicolumn{1}{c}{2}&1.484&1.218&0.995&0.816&0.672\tabularnewline[0pt]
&~~&\multicolumn{1}{c}{3}&1.750&1.337&1.000&0.749&0.570\tabularnewline[0pt]
&~~&\multicolumn{1}{c}{4}&2.013&1.432&1.993&0.705&0.502\tabularnewline[0pt]
&~~&\multicolumn{1}{c}{5}&2.148&1.447&0.982&0.664&0.451\tabularnewline[0pt]
&~~&\multicolumn{1}{c}{6}&2.424&1.566&1.024&0.661&0.436
\end{tabular}
}%
\end{tableptx}%
\end{paragraphs}%
\begin{paragraphs}{The data are consistent, therefore Gram points are not important.}{grampointsrevisited-11}%
\hyperref[shankertab1]{Table~{\xreffont\ref{shankertab1}}} and \hyperref[goodbadrmt]{Table~{\xreffont\ref{goodbadrmt}}} both show that the good and bad gram points behave differently, and the proportions in each table are similar.  However, this is not evidence against the prediction that the distinction between good and bad gram points is meaningless in the long run. Indeed, a matrix size of \(N=62\) is small, particularly when some of the limiting behavior reveals itself on the scale of \(\sqrt{\log N}\). The \(N=250\) and \(1000\) sections of \hyperref[goodbadrmt]{Table~{\xreffont\ref{goodbadrmt}}} provide a nice illustration of the fact that, as \(N\to\infty\), the good and bad intervals become equally likely. As described in \hyperref[sloshing_water_model]{Principle~{\xreffont\ref{sloshing_water_model}}}, at large height most zeros have slid far away from their expected location, so it is meaningless to try to separate ``good'' from ``bad''.  Yet the local spacing between zeros still looks the same everywhere, and that is the real reason that Grams law (but not the precise Grams law) continues to hold more than 66\% of the time.%
\par
\hyperref[deltagramrmt]{Table~{\xreffont\ref{deltagramrmt}}} shows that at \(N=62\), shifting the Gram points has a significant effect.  But \(62\) is a small number, and the \(N=250\) and \(1000\) cases make it less surprising that as \(N\to \infty\) every entry in that table will approach~\(1\).%
\end{paragraphs}%
\begin{paragraphs}{The data differ significantly.}{grampointsrevisited-12}%
We have argued that the numbers in \hyperref[shankertab1]{Table~{\xreffont\ref{shankertab1}}} and the \(N=62\) portion of \hyperref[goodbadrmt]{Table~{\xreffont\ref{goodbadrmt}}} are similar, but they also differ significantly.  The sample sizes are \(10^7\), so the respective data differ by more than 10 standard deviations. This is a consequence of \hyperref[KSoverestimate]{Principle~{\xreffont\ref{KSoverestimate}}}.  The lower order terms mentioned in that principle can be (statistically) significant, particularly when we are sampling at a small matrix size such as \(N=62\). Note that the differences are in the direction indicated by that principle: for the random matrices larger gaps and clusters of zeros are more likely.%
\par
The differences between \hyperref[shankertab2]{Table~{\xreffont\ref{shankertab2}}} and the \(N=62\) portion of \hyperref[deltagramrmt]{Table~{\xreffont\ref{deltagramrmt}}} are even more pronounced. The differences are greatest for the unlikely events:  there are \(10^7\) data points, but the lowest row of each table corresponding to \(\text{length} = 6\) only occurs a few thousand times in the sample. So, those apparent discrepancies are also explained by the principle that for random matrices there are secondary terms which make extreme events more likely.%
\end{paragraphs}%
\end{subsectionptx}
\end{sectionptx}
\typeout{************************************************}
\typeout{Section 15 The chicken or the egg?}
\typeout{************************************************}
\begin{sectionptx}{Section}{The chicken or the egg?}{}{The chicken or the egg?}{}{}{root-1-2-18}
Which came first: the \(\zeta\)-function or its zeros? That sounds like a meaningless question, but we suggest that the answer factors into whether or not one is skeptical of RH.%
\par
Suppose you believe the function is what matters, either the \(\zeta\)-function given by its Dirichlet series or some other expression, or the \(Z\)-function given by the Riemann-Siegel formula.  From that perspective, the zeros are determined by the function, and it would take an unlikely conspiracy for RH to be true.  It might still bother you that RH is ``barely true''. The suggestion at the very end of \hyperref[howtorefute]{Subsection~{\xreffont\ref{howtorefute}}} might be viewed as an intriguing idea worth exploring. A~failure of RH could occur for accidental reasons.%
\par
Now suppose you believe it is the zeros which matter.  The zeros determine the function, with  the density wave setting the scale and the local arrangement providing lower-order adjustments. Every zero is important and meaningful: a purported list of zeros containing one tiny error would be detected (by Weil's explicit formula or some other means). If there were zeros off the line, they would be there to serve a specific purpose and their location would be meaningful. The suggestion at the end of \hyperref[howtorefute]{Subsection~{\xreffont\ref{howtorefute}}} is just silly:  the \(Z\)-function stays small over a wide region only because the density of zeros is higher.%
\par
It is hoped that this paper helps cultivate an appreciation for  the explanatory power which comes from viewing the zeros as the basic object.%
\end{sectionptx}
\addtocontents{toc}{\vspace{\normalbaselineskip}}
\typeout{************************************************}
\typeout{References  Bibliography}
\typeout{************************************************}
\begin{references-section-numberless}{References}{Bibliography}{}{Bibliography}{}{}{bibliography}
\begin{referencelist}
\bibitem[1]{ArBa1}\hypertarget{ArBa1}{}Louis-Pierre Arguin and Emma Bailey, \textit{Large Deviation Estimates of Selberg’s Central Limit Theorem and Applications}, IMRN no.\@\,23, \textbf{2023} (2023) pp.\@\,20574-20612.arXiv: 2202.06799; 
\bibitem[2]{ArBa2}\hypertarget{ArBa2}{}Louis-Pierre Arguin and Emma Bailey, \textit{Lower bounds for the large deviations of Selberg's central limit theorem}, Mathematika no.\@\,1, \textbf{71} (2025) e70002.arXiv: 2403:19803; 
\bibitem[3]{ABR1}\hypertarget{ABR1}{}Louis-Pierre Arguin, Paul Bourgade, and Maksym Radziwiłł, \textit{The Fyodorov-Hiary-Keating Conjecture. I}, arXiv: 2007.00988; (2020) 
\bibitem[4]{ABR2}\hypertarget{ABR2}{}Louis-Pierre Arguin, Paul Bourgade, and Maksym Radziwiłł, \textit{The Fyodorov-Hiary-Keating Conjecture. II}, arXiv: 2307.00982; (2023) 
\bibitem[5]{ADH}\hypertarget{ADH}{}Louis-Pierre Arguin, Guillaume Dubach, and Lisa Hartung, \textit{Maxima of a Random Model of the Riemann Zeta Function over Intervals of Varying Length}, Ann. Inst. Henri Poincaré Probab. Stat. no.\@\,1, \textbf{60} (2024) pp.\@\,174-200.arXiv: 2103.04817; 
\bibitem[6]{A}\hypertarget{A}{}F. V. Atkinson, \textit{The mean value of the zeta-function on the critical line}, Proc. London Math. Soc. no.\@\,2, \textbf{47} (1941) pp.\@\,174-200.
\bibitem[7]{Bac}\hypertarget{Bac}{}R. Backlund, \textit{Über die Beziehung zwischen Anwachsen und Nullstellen der Zetafunktion}, Öfversigt Finska Vetensk. Soc. no.\@\,9, \textbf{61} (1918-19) 
\bibitem[8]{BK1}\hypertarget{BK1}{}E.C. Bailey and J.P. Keating, \textit{On the moments of the moments of the characteristic polynomials of random unitary matrices}, Comm. Math. Phys. \textbf{371} no.\@\,2, (2019) pp.\@\,689-726.
\bibitem[9]{BK2}\hypertarget{BK2}{}E.C. Bailey and J.P. Keating, \textit{On the moments of the moments of \(\zeta(1/2+it\)}, J. Number Theory \textbf{223} (2021) pp.\@\,79-101.
\bibitem[10]{BaKe}\hypertarget{BaKe}{}E.C. Bailey and J.P. Keating, \textit{Maxima of log-correlated fields: some recent developments}, J. Phys. A: Math. Theor. \textbf{55} no.\@\,5, arXiv: 2106:15141; (2022) 
\bibitem[11]{BCNN}\hypertarget{BCNN}{}Yacine Barhoumi-Andréani, Christopher Hughes, Joseph Najnudel, and Ashkan Nikeghbali, \textit{On the number of zeros of linear combinations of independent characteristic polynomials of random unitary matrices}, Int. Math. Res. Not. IMRN no.\@\,23, (2015) pp.\@\,12366-12404.
\bibitem[12]{BA_B}\hypertarget{BA_B}{}Gérard Ben Arous and Paul Bourgade, \textit{Extreme gaps between eigenvalues of random matrices}, Ann. Probab. \textbf{41} no.\@\,4, (2013) pp.\@\,2648-2681.
\bibitem[13]{Berry}\hypertarget{Berry}{}M. V. Berry, \textit{Universal oscillations of high derivatives}, Proc. R. Soc. Lond. Ser. A Math. Phys. Eng. Sci. \textbf{461} no.\@\,2058, (2005) pp.\@\,1735-1751.
\bibitem[14]{BerrySO}\hypertarget{BerrySO}{}M. V. Berry, \textit{Superoscillations and leaky spectra}, J. Phys. A: Math. Theor. \textbf{52} (2019) pp.\@\,1-11.
\bibitem[15]{BerryFTZ}\hypertarget{BerryFTZ}{}M. V. Berry, \textit{Riemann zeros in radiation patterns: II. Fourier transforms of zeta}, J. Phys. A: Math. Theor. \textbf{48} (2015) 
\bibitem[16]{berryA}\hypertarget{berryA}{}M. V. Berry, \textit{Semiclassical formula for the number variance of the Riemann zeros}, Nonlinearity \textbf{1} (1988) pp.\@\,399-407.
\bibitem[17]{berrykeatingA}\hypertarget{berrykeatingA}{}M. V. Berry and J. P. Keating, \textit{The Riemann zeros and eigenvalue asymptotics}, SIAM Review \textbf{41} (1999) pp.\@\,236-266.
\bibitem[18]{Blanc1}\hypertarget{Blanc1}{}Philippe Blanc, \textit{A new reason for doubting the Riemann Hypothesis}, Exp. Math. \textbf{31} no.\@\,1, (2022) pp.\@\,88-92.10.1080\slash{}10586458.2019.1586600
\bibitem[19]{Blanc2}\hypertarget{Blanc2}{}Philippe Blanc, \textit{An unexpected property of odd order derivatives of Hardy's function}, Publ. Inst. Math. (Beograd) (N.S.) \textbf{95} no.\@\,109, (2014) pp.\@\,173-188.
\bibitem[20]{Blanc3}\hypertarget{Blanc3}{}Philippe Blanc, \textit{Optimal upper bound for the maximum of the \(k\)-th derivative of Hardy's function}, JNT \textbf{154} (2015) pp.\@\,105-117.
\bibitem[21]{BobHia}\hypertarget{BobHia}{}Jonathan Bober and Ghaith Hiary, \textit{New computations of the Riemann zeta function on the critical line}, Exp. Math. \textbf{27} no.\@\,2, (2018) pp.\@\,125-137.
\bibitem[22]{zetaomega}\hypertarget{zetaomega}{}Andriy Bondarenko and Kristian Seip, \textit{Extreme values of the Riemann zeta function and its argument}, Math. Ann. \textbf{372} no.\@\,3-4, (2018) pp.\@\,999-1015.
\bibitem[23]{millen}\hypertarget{millen}{}E. Bombieri, \textit{Problems of the Millennium: the Riemann Hypothesis}, https:\slash{}\slash{}www.claymath.org\slash{}wp-content\slash{}uploads\slash{}2022\slash{}05\slash{}riemann.pdf; 
\bibitem[24]{BomHej}\hypertarget{BomHej}{}E. Bombieri and D. Hejhal, \textit{On the distribution of zeros of linear combinations of Euler products}, Duke Math. J. \textbf{80} no.\@\,3, (1995) pp.\@\,821-862.
\bibitem[25]{BookerTuring}\hypertarget{BookerTuring}{}Andrew Booker, \textit{Artin's Conjecture, Turing's Method, and the Riemann Hypothesis}, Experiment. Math. \textbf{15} no.\@\,4, (2006) pp.\@\,385-408.
\bibitem[26]{Booker}\hypertarget{Booker}{}Andrew Booker and Frank Thorne, \textit{Zeros of L-functions outside the critical strip}, Algebra Number Theory \textbf{8} no.\@\,9, (2014) pp.\@\,2027-2042.
\bibitem[27]{brent}\hypertarget{brent}{}Richard P. Brent, \textit{On the zeros of the Riemann zeta function in the critical strip}, Math. Comp. \textbf{33} no.\@\,148, (1979) pp.\@\,1361-1372.
\bibitem[28]{BuFl}\hypertarget{BuFl}{}H. Bui and A. Florea, \textit{Negative moments of the Riemann zeta-function}, Journal für die reine und angewandte Mathematik (Crelle) no.\@\,806, (2024) pp.\@\,247-288.arXiv: 2302.07226; 
\bibitem[29]{con25}\hypertarget{con25}{}J. B. Conrey, \textit{More than two fifths of the zeros of the Riemann zeta function are on the critical line}, J. Reine Angew. Math. \textbf{399} (1986) pp.\@\,1-26.
\bibitem[30]{CF}\hypertarget{CF}{}J. B. Conrey and D. W. Farmer, \textit{Mean values of \(L\)-functions and symmetry}, Internat. Math. Res. Notices (2000) \textbf{17} pp.\@\,883-908.
\bibitem[31]{CFKRS}\hypertarget{CFKRS}{}J. B. Conrey, D. W. Farmer, J. P. Keating, M. O. Rubinstein, and N. C. Snaith, \textit{Autocorrelation of random matrix polynomials}, Commun. Math. Phys \textbf{237} (2003) no.\@\,3, pp.\@\,365-395.
\bibitem[32]{CFZ}\hypertarget{CFZ}{}J. B. Conrey, D. W. Farmer, and M. R. Zirnbauer, \textit{Autocorrelation of ratios of L-functions}, Commun. Number Theory Phys \textbf{2} (2008) no.\@\,3, pp.\@\,593-636.
\bibitem[33]{CG1}\hypertarget{CG1}{}J. B. Conrey and A. Ghosh, \textit{Mean values of the Riemann zeta-function}, Mathematika \textbf{31} (1984) pp.\@\,159-161.
\bibitem[34]{CG2}\hypertarget{CG2}{}J. B. Conrey and A. Ghosh, \textit{A conjecture for the sixth power moment of the Riemann zeta-function}, Int. Math. Res. Not. \textbf{15} (1998) pp.\@\,775-780.
\bibitem[35]{CG3}\hypertarget{CG3}{}J. B. Conrey and A. Ghosh, \textit{On the Selberg class of Dirichlet series: small degrees}, Duke Math. J. \textbf{72} (1993) pp.\@\,673-693.
\bibitem[36]{CoGo}\hypertarget{CoGo}{}J. B. Conrey and S. M. Gonek, \textit{High moments of the Riemann zeta-function}, Duke Math. Jour. (2001) \textbf{107} pp.\@\,577-604.
\bibitem[37]{CI}\hypertarget{CI}{}J. B. Conrey and H. Iwaniec, \textit{Spacing of zeros of Hecke L-functions and the class number problem}, Acta Arith. (2002) \textbf{103} no.\@\,3, pp.\@\,259-312.
\bibitem[38]{CS}\hypertarget{CS}{}J. B. Conrey and N. C. Snaith, \textit{Applications of the L-functions ratios conjectures}, Proc. Lond. Math. Soc. (2007) \textbf{94} no.\@\,3, pp.\@\,594-646.
\bibitem[39]{Cur}\hypertarget{Cur}{}M. J. Curran, \textit{Freezing transition and moments of moments of the Riemann zeta function}, Q. J. Math. (2024) \textbf{75} no.\@\,4, pp.\@\,1481-1505.
\bibitem[40]{DiaSha}\hypertarget{DiaSha}{}P. Diaconis and M. Shahshahani, \textit{On the eigenvalues of random matrices}, J. Appl. Probab. \textbf{31A} (1994) pp.\@\,49-62.
\bibitem[41]{Dia}\hypertarget{Dia}{}A. Diaconu, \textit{On the third moment of \(L(\tfrac12,\chi_d)\) I: The rational function field case}, J. Number Theory. \textbf{198} (2019) pp.\@\,1-42.
\bibitem[42]{DGH}\hypertarget{DGH}{}A. Diaconu, D. Goldfeld, and J. Hoffstein, \textit{Multiple Dirichlet series and moments of zeta- and \(L\)-functions}, Compositio Math. \textbf{139} (2003) no.\@\,3, pp.\@\,297-360.
\bibitem[43]{Dob}\hypertarget{Dob}{}Alexander Dobner, \textit{A proof of Newman's conjecture for the extended Selberg class}, Acta Arith. \textbf{201} no.\@\,1, pp.\@\,29-62.(2021) 
\bibitem[44]{Dy}\hypertarget{Dy}{}Freeman Dyson, \textit{A Brownian-motion model for the eigenvalues of a random matrix}, J. Mathematical Phys. \textbf{3} pp.\@\,1191-1198.(1962) 
\bibitem[45]{DE}\hypertarget{DE}{}Ioana Dumitriu and Alan Edelman, \textit{Matrix models for beta ensembles}, J. Math. Phys. \textbf{43} no.\@\,11, pp.\@\,5830-5847.(2002) 
\bibitem[46]{Fratios}\hypertarget{Fratios}{}David W. Farmer, \textit{Long mollifiers of the Riemann zeta-function}, Mathematika \textbf{40} no.\@\,1, pp.\@\,71-87.(1993) 
\bibitem[47]{jensen}\hypertarget{jensen}{}David W. Farmer, \textit{Jensen polynomials are not a plausible route to proving the Riemann Hypothesis}, Adv. Math. \textbf{411} (2022) 14 pages.Paper No. 108781https:\slash{}\slash{}arxiv.org\slash{}abs\slash{}2008.07206; 
\bibitem[48]{allreal}\hypertarget{allreal}{}David W. Farmer, \textit{When are the zeros of a polynomial distinct and real: a graphical view}, Preprint (2021) https:\slash{}\slash{}arxiv.org\slash{}abs\slash{}2010.15608; 
\bibitem[49]{FGH}\hypertarget{FGH}{}D.W. Farmer, S.M. Gonek, and C.P. Hughes, \textit{The maximum size of \(L\)-functions}, J. Reine Angew. Math. \textbf{609} (2007) pp.\@\,215-236.
\bibitem[50]{FMS}\hypertarget{FMS}{}David W. Farmer, Francesco Mezzadri, and Nina Snaith, \textit{Random polynomials, random matrices and L-functions}, Nonlinearity \textbf{19} (2006) no.\@\,4, pp.\@\,919-936.
\bibitem[51]{FPRS}\hypertarget{FPRS}{}David W. Farmer, Ameya Pitale, Nathan Ryan, and Ralf Schmidt, \textit{Analytic L-functions: definitions, theorems, and connections}, Bull. Amer. Math. Soc. (N.S.) \textbf{56} (2019) no.\@\,2, pp.\@\,261-280.
\bibitem[52]{FR}\hypertarget{FR}{}David W. Farmer and Robert Rhoades, \textit{Differentiation evens out zero spacings}, Trans. Amer. Math. Soc. \textbf{357} (2005) no.\@\,9, pp.\@\,3789-3811.
\bibitem[53]{FKLR}\hypertarget{FKLR}{}David W. Farmer, Sally Koutsoliotas, Stefan Lemurell, and David P. Roberts, \textit{Patterns in some landscapes of L-functions}, preprint; 
\bibitem[54]{feng}\hypertarget{feng}{}Shaoji Feng, \textit{Zeros of the Riemann zeta function on the critical line}, J. Number Theory \textbf{132} (2012) no.\@\,4, pp.\@\,511-542.
\bibitem[55]{FMN1}\hypertarget{FMN1}{}Valentin Féray, Pierre-Loïc Méliot, and Ashkan Nikeghbali, \textit{Mod-phi convergence I: Normality zones and precise deviations}, arXiv: 1304.2934; (2015) 
\bibitem[56]{FyKe}\hypertarget{FyKe}{}Yan Fyodorov and Jonathan Keating, \textit{Freezing transition and extreme values: random matrix theory, \(\zeta(\frac12 + it)\) and disordered landscapes}, Phil. Trans. R. Soc. A \textbf{372} (2014) no.\@\,20120503, 32.
\bibitem[57]{FGK}\hypertarget{FGK}{}Yan Fyodorov, Sven Gnutzmann, and Jonathan Keating, \textit{Extreme values of CUE characteristic polynomials: a numerical study}, J. Phys. A \textbf{51} (2018) no.\@\,46, 22.
\bibitem[58]{FHK}\hypertarget{FHK}{}Yan Fyodorov, Ghaith Hiary, and Jonathan Keating, \textit{Freezing Transition, Characteristic Polynomials of Random Matrices, and the Riemann Zeta-Function}, Phys. Rev. Lett. \textbf{108} (2012) no.\@\,170601, 
\bibitem[59]{hybrid}\hypertarget{hybrid}{}Gonek, S. M.; Hughes, C. P.; Keating, J. P., \textit{A hybrid Euler-Hadamard product for the Riemann zeta function}, Duke Math. J. \textbf{136} no.\@\,3, (2007) pp.\@\,507-549.
\bibitem[60]{XG}\hypertarget{XG}{}Xavier Gourdon, \textit{Computation of zeros of the Zeta function}, preprint; http:\slash{}\slash{}numbers.computation.free.fr\slash{}Constants\slash{}Miscellaneous\slash{}zetazeroscompute.html (2004) 
\bibitem[61]{gram}\hypertarget{gram}{}J. Gram, \textit{Sur les zéros de la fonction \(\zeta(s)\) de Riemann}, Acta Math \textbf{27} (1903) pp.\@\,289-304.
\bibitem[62]{HH}\hypertarget{HH}{}Cătălin Hanga and Christopher Hughes, \textit{Probabilistic models for Gram's Law}, preprint; arXiv: 1911.03190; (2020) 
\bibitem[63]{HL}\hypertarget{HL}{}G. H. Hardy and J. E. Littlewood, \textit{Contributions to the theory of the Riemann zeta-function and the theory of the distribution of primes}, (1918) Acta Mathematica pp.\@\,119-196.\textbf{41} 
\bibitem[64]{Harper1}\hypertarget{Harper1}{}A. Harper, \textit{On the partition function of the Riemann zeta function, and the Fyodorov-{}-{}Hiary-{}-{}Keating conjecture}, preprint; arXiv: 1906.05783; (2019) 
\bibitem[65]{Hej}\hypertarget{Hej}{}Dennis Hejhal, \textit{On the distribution of zeros of a certain class of Dirichlet series}, Internat. Math. Res. Notices no.\@\,4, (1992) pp.\@\,83-91.
\bibitem[66]{Hia}\hypertarget{Hia}{}Ghaith Hiary, \textit{Personal website}, https:\slash{}\slash{}people.math.osu.edu\slash{}hiary.1\slash{}fastmethods.html (Accessed March 3, 2022) 
\bibitem[67]{HiOd}\hypertarget{HiOd}{}Ghaith Hiary and Andrew Odlyzko, \textit{The zeta function on the critical line: numerical evidence for moments and random matrix theory models}, Math. Comp. \textbf{81} no.\@\,279, (2012) pp.\@\,1723-1752.
\bibitem[68]{HPZ}\hypertarget{HPZ}{}A. Huckleberry, A. Püttmann, and M. R. Zirnbauer, \textit{Haar expectations of ratios of random characteristic polynomials}, Complex Anal. Synerg. no.\@\,1, (2016) \textbf{2} Paper No. 1, 73 pp..
\bibitem[69]{HKO}\hypertarget{HKO}{}C.P. Hughes, J.P. Keating, and Neil O'Connell, \textit{Random matrix theory and the derivative of the Riemann zeta function}, R. Soc. Lond. Proc. Ser. A Math. Phys. Eng. Sci. no.\@\,2003, \textbf{456} (2000) pp.\@\,2611-2627.
\bibitem[70]{hutch}\hypertarget{hutch}{}J. I. Hutchinson, \textit{On the roots of the Riemann zeta-function}, Trans. Amer. Math. Soc. no.\@\,27, (1925) pp.\@\,49-60.
\bibitem[71]{Ing}\hypertarget{Ing}{}A. E. Ingham, \textit{Mean-value theorems in the theory of the Riemann zeta-function}, Proceedings of the London Mathematical Society no.\@\,92, (1926) \textbf{27} pp.\@\,273-300.
\bibitem[72]{Iv1}\hypertarget{Iv1}{}A. Ivić, \textit{On some results concerning the Riemann Hypothesis}, \pubtitle{Analytic Number Theory}LMS LNS no.\@\,247, Cambridge University Press; Y. Motohashi, (1997) pp.\@\,139-167.
\bibitem[73]{Iv2}\hypertarget{Iv2}{}A. Ivić, \textit{On some reasons for doubting the Riemann Hypothesis}, preprint; arXiv: 0311162; (2003) 
\bibitem[74]{kotnik}\hypertarget{kotnik}{}Tadej Kotnik, \textit{Computational estimation of the order of \(\zeta(1/2+it)\)}, Math. Comp. \textbf{73} no.\@\,246, (2004) pp.\@\,949-956.
\bibitem[75]{KaSa}\hypertarget{KaSa}{}N. M. Katz and P. Sarnak, \textit{Random matrices, Frobenius eigenvalues, and monodromy}, AMS Colloquium Publications \textbf{45} AMS, Providence, RI; (1999) 
\bibitem[76]{KS1}\hypertarget{KS1}{}J. P. Keating and N. C. Snaith, \textit{Random matrix theory and \(\zeta(\frac12 +it)\)}, Comm. Math. Phys. \textbf{214} (2000) pp.\@\,57-89.
\bibitem[77]{KS2}\hypertarget{KS2}{}J. P. Keating and N. C. Snaith, \textit{Random matrix theory and \(L\)-functions at \(s=\frac12 \)}, Comm. Math. Phys. \textbf{214} (2000) pp.\@\,91-110.
\bibitem[78]{Ki}\hypertarget{Ki}{}Haseo Ki, \textit{The Riemann \(Xi\)-function under repeated differentiation}, J. Number Theory \textbf{120} no.\@\,1, (2006) pp.\@\,120-131.
\bibitem[79]{KN}\hypertarget{KN}{}Rowan Killip and Irina Nenciu, \textit{Matrix models for circular ensembles.}, Int. Math. Res. Not. no.\@\,50, (2004) pp.\@\,2665-2701.
\bibitem[80]{lev}\hypertarget{lev}{}Norman Levinson, \textit{More than one third of zeros of Riemann's zeta-function are on \(\sigma=1/2\)}, Advances in Math. \textbf{13} (1974) pp.\@\,383-436.
\bibitem[81]{LMQ-H}\hypertarget{LMQ-H}{}Meghann Moriah Lugar, Micah B. Milinovich, and Emily Quesada-Herrera, \textit{On the number variance of zeta zeros and a conjecture of Berry}, Mathematika \textbf{69} 2arXiv 2211.14918; (2023) 
\bibitem[82]{Meh}\hypertarget{Meh}{}M. Mehta, \textit{Random Matrix Theory}, Academic Press. Boston; (1991) 
\bibitem[83]{Mez}\hypertarget{Mez}{}Francesco Mezzadri, \textit{How to generate random matrices from the classical compact groups}, Notices Amer. Math. Soc. \textbf{54} no.\@\,5, (2007) pp.\@\,592-604.
\bibitem[84]{Mon}\hypertarget{Mon}{}H.L. Montgomery, \textit{The pair correlation of zeros of the Riemann zeta-function}, Proc. Symp. Pure Math. \textbf{24} (1973) pp.\@\,81-93.
\bibitem[85]{New}\hypertarget{New}{}Charles M. Newman, \textit{Fourier transforms with only real zeros}, Proc. Amer. Math. Soc. \textbf{61} no.\@\,2, (1977) pp.\@\,245-251.
\bibitem[86]{OdlOld}\hypertarget{OdlOld}{}A. Odlyzko, \textit{The \(10^{20}\)th zero of the Riemann zeta-function and 70 million of its neighbors.}, Preprint (1989) 
\bibitem[87]{odl12}\hypertarget{odl12}{}A. Odlyzko, \textit{Zeros number \(10^{12}+1\) through \(10^{12}+10^4\) of the Riemann zeta function}, available at http:\slash{}\slash{}www.dtc.umn.edu\slash{}\textasciitilde{}odlyzko\slash{}zeta\textunderscore{}tables; 
\bibitem[88]{Odl}\hypertarget{Odl}{}A. Odlyzko, \textit{On the distribution of spacings between zeros of the zeta function}, Math. Comp. \textbf{48} (1987) pp.\@\,273-308.
\bibitem[89]{PlTr}\hypertarget{PlTr}{}Dave Platt and Tim Trudgian, \textit{The Riemann hypothesis is true up to \(3\cdot10^{12}\)}, Bull. Lond. Math. Soc. \textbf{53} no.\@\,3, (2021) pp.\@\,792-797.
\bibitem[90]{prattetal}\hypertarget{prattetal}{}Kyle Pratt, Nicolas Robles, Alexandru Zaharescu, and Dirk Zeindler, \textit{More than five-twelfths of the zeros of \(\zeta\) are on the critical line}, Res. Math. Sci. \textbf{7} no.\@\,2, (2020) paper no. 2, 74.
\bibitem[91]{PTX}\hypertarget{PTX}{}The PreTeXt project, \textit{The PreTeXt document authoring system}, https:\slash{}\slash{}pretextbook.org 
\bibitem[92]{Rad1}\hypertarget{Rad1}{}Maksym Radziwiłł, \textit{Large deviations in Selberg's central limit theorem}, arXiv: 1108.5092; (2011) 
\bibitem[93]{RhVa}\hypertarget{RhVa}{}Rémi Rhodes and Vincent Vargas, \textit{Gaussian multiplicative chaos and applications: a review}, Probability Surveys \textbf{11} (2014) pp.\@\,315-392.arXiv 1305.6221; 
\bibitem[94]{RT}\hypertarget{RT}{}Brad Rodgers and Terence Tao, \textit{The de Bruijn-Newman constant is non-negative}, Forum Math. Pi no.\@\,e6, \textbf{8} (2020) 62.
\bibitem[95]{Rig}\hypertarget{Rig}{}Mattia Righetti, \textit{Zeros of combinations of Euler products for \(\sigma\gt1\)}, Monatsh. Math. \textbf{180} no.\@\,2, (2016) pp.\@\,337-356.
\bibitem[96]{Rub1}\hypertarget{Rub1}{}Michael Rubinstein, \textit{Low-lying zeros of L-functions and random matrix theory}, Duke Math. J. \textbf{109} no.\@\,1, (2001) pp.\@\,147-181.
\bibitem[97]{Rub2}\hypertarget{Rub2}{}Michael Rubinstein, \textit{Computational methods and experiments in analytic number theory}, \pubtitle{Recent perspectives in random matrix theory and number theory, London Math. Soc. Lecture Note Ser.}Cambridge Univ. Press; no.\@\,322, (2005) pp.\@\,425-506.
\bibitem[98]{RS}\hypertarget{RS}{}Z. Rudnick and P. Sarnak, \textit{Zeros of principal \(L\)-functions and random matrix theory}, Duke Math. J. \textbf{81} (1996) pp.\@\,269-322.
\bibitem[99]{Saias}\hypertarget{Saias}{}Eric Saias and Andreas J. Weingartner, \textit{Zeros of Dirichlet series with periodic coefficients}, Acta Arithmetica no.\@\,4, \textbf{140} (2008) pp.\@\,335-344.
\bibitem[100]{SaksWebb}\hypertarget{SaksWebb}{}E. Saksman and C. Webb, \textit{The Riemann zeta function and Gaussian multiplicative chaos: statistics on the critical line}, Ann. Probab. no.\@\,2, \textbf{48} (2020) pp.\@\,2680-2754.
\bibitem[101]{SaksWebbSurv}\hypertarget{SaksWebbSurv}{}E. Saksman and C. Webb, \textit{On the Riemann zeta function and Gaussian multiplicative chaos}, Advancements in complex analysis -{}-{} from theory to practice Springer; (2020) pp.\@\,473-496.
\bibitem[102]{saw}\hypertarget{saw}{}Will Sawin, \textit{A refined random matrix model for function field L-functions}, Preprint (2024) https:\slash{}\slash{}arxiv.org\slash{}abs\slash{}2409.02876; 
\bibitem[103]{SelS}\hypertarget{SelS}{}A. Selberg, \textit{Contributions to the theory of the Riemann zeta-function}, Arch. Math. Naturvid. \textbf{48} no.\@\,5, (1946) pp.\@\,89-155.
\bibitem[104]{Sel}\hypertarget{Sel}{}A. Selberg, \textit{Old and new conjectures and results about a class of Dirichlet series}, \pubtitle{Proceedings of the Amalfi Conference on Analytic Number Theory (Maiori, 1989)}Univ. Salerno, Salerno; (1992) 
\bibitem[105]{shanker}\hypertarget{shanker}{}O. Shanker, \textit{Good-to-bad Gram point ratio for Riemann zeta-function}, Experimental Math. \textbf{30} no.\@\,1, (2021) pp.\@\,76-85.
\bibitem[106]{soundICM}\hypertarget{soundICM}{}K. Soundararajan, \textit{The distribution of values of zeta and L-functions}, ICM-{}-{}International Congress of Mathematicians. Vol. 2. Plenary lectures pp.\@\,1260-1310.(2023) arXiv 2112.03389; 
\bibitem[107]{speiser}\hypertarget{speiser}{}A. Speiser, \textit{Geometrisches zur Riemannschen Zetafunktion}, Math Ann. \textbf{110} (1935) pp.\@\,514-521.
\bibitem[108]{Tih1}\hypertarget{Tih1}{}Norbert Tihanyi, \textit{Fast method for locating peak values of the Riemann zeta function on the critical line}, 2014 16th International Symposium on Symbolic and Numeric Algorithms for Scientific Computing IEEE; (2014) pp.\@\,91-96.
\bibitem[109]{Tih2}\hypertarget{Tih2}{}Norbert Tihanyi, Atilla Kovács, and József Kovács, \textit{Computing extremely large values of the Riemann zeta function}, J. Grid Computing \textbf{15} (2017) pp.\@\,527-534.
\bibitem[110]{T}\hypertarget{T}{}E. C. Titchmarsh, \textit{The theory of the Riemann zeta-function. Second edition. Edited and with a preface by D. R. Heath-Brown}, The Clarendon Press, Oxford University Press, New York; (1986) 
\bibitem[111]{TT1}\hypertarget{TT1}{}T. Trudgian, \textit{On the success and failure of Gram's law and the Rosser rule}, Acta Arith. \textbf{148} (2011) pp.\@\,225-256.
\end{referencelist}
\end{references-section-numberless}
\end{document}